\input ./style/arxiv-general.cfg
\documentclass[aap,MSNbibl,seceqn,dvips]{arximspdf}
\makeatletter
   \@ifpackageloaded{graphicx}{}{\usepackage{graphicx}}
\makeatother

\doi{10.1214/14-AAP1086}
\volume{26}
\issue{1}
\pubyear{2016}
\firstpage{73}
\lastpage{135}
\docsubty{FLA}

\makeatletter
\newcommand{\rrvert}{\vert}
\newcommand{\llvert}{\vert}
\def\cal{\mathcal}

\newcommand{\eqref}[1]{(\ref{#1})}
\newcommand{\iint}{\int\!\!\!\int}
\newcommand{\iiint}{\int\!\!\!\int\!\!\!\int}
\newcommand{\iiiint}{\int\!\!\!\int\!\!\!\int\!\!\!\int}
\newcommand{\idotsint}{\int\cdots\int}
\newtheorem{theorem}{Theorem}[section]
\newproclaim{definition}[theorem]{Definition}
\newtheorem{proposition}[theorem]{Proposition}
\newtheorem{lemma}[theorem]{Lemma}
\newtheorem{corol}[theorem]{Corollary}
\newproclaim{remark}[theorem]{Remark}
\newproclaim{example}[theorem]{Example}

\newcommand{\dom}{\operatorname{dom}}
\newcommand{\V}{\operatorname{Var}}
\newcommand{\C}{\operatorname{Cov}}
\newcommand{\Lip}{{\operatorname{Lip}}}
\newcommand{\cI}{\mathcal{I}}
\newcommand{\Curvi}{\Phi_i}
\newcommand{\dist}{\operatorname{dist}}

\def\R{{\mathbb{R}}}
\def\S{{\mathbb{S}}}
\def\E{{\mathbb{E}}}
\def\N{\mathbb{N}}
\def\P{\mathbb{P}}
\def\1{\mathbf{1}}

\def\cC{{\mathcal C}}
\def\cK{{\mathcal K}}
\def\cI{{\mathcal I}}
\def\I{\1}

\newcommand{\BP}{\mathbb{P}}
\newcommand{\BE}{\mathbb{E}}
\newcommand{\Q}{\mathbb{Q}}
\newcommand{\BQ}{\mathbb{Q}}
\newcommand{\bN}{{\mathbf N}}
\newcommand{\cR}{{\mathcal R}}
\newcommand{\BX}{{\mathbf X}}

\newcommand{\cor}{\mathbb{\operatorname{Cor}}}

\makeatother

\begin{document}
\begin{frontmatter}

\title{Second-order properties and central limit theorems for geometric
functionals of~Boolean~models\thanksref{T1}}
\thankstext{T1}{Supported by the German
Research Foundation (DFG) through the research unit ``Geometry and
Physics of Spatial Random Systems'' under the Grant HU 1874/3-1.}
\runtitle{Second-order properties of Boolean models}

\begin{aug}
\author[A]{\fnms{Daniel}~\snm{Hug}\ead[label=e1]{daniel.hug@kit.edu}},
\author[A]{\fnms{G\"unter}~\snm{Last}\corref{}\ead[label=e2]{guenter.last@kit.edu}}
\and
\author[A]{\fnms{Matthias}~\snm{Schulte}\ead[label=e3]{matthias.schulte@kit.edu}}
\runauthor{D. Hug, G. Last and M. Schulte}
\affiliation{Karlsruhe Institute of Technology}
\address[A]{Karlsruhe Institute of Technology\\
Institute of Stochastics\\
76128 Karlsruhe\\
Germany\\
\printead{e1}\\
\phantom{E-mail:\ }\printead*{e2}\\
\phantom{E-mail:\ }\printead*{e3}}
\end{aug}

%
\received{\smonth{8} \syear{2013}}
%
\revised{\smonth{4} \syear{2014}}

%
\begin{abstract}
Let $Z$ be a Boolean model based on a stationary Poisson process $\eta$
of compact, convex particles in Euclidean space ${\mathbb{R}}^d$. Let $W$ denote
a compact,
convex observation
window. For a large class of functionals $\psi$,
formulas for mean values of $\psi(Z\cap W)$ are available in the literature.
The first aim of the present work is to study
the asymptotic covariances of general geometric (additive,
translation invariant and locally bounded) functionals of $Z\cap W$ for
increasing
observation window $W$, including convergence rates.
Our approach is based on the Fock space representation associated with
$\eta$.
For the important special case of intrinsic volumes,
the asymptotic covariance matrix is shown to be
positive definite and can be explicitly expressed in terms of suitable moments
of (local) curvature measures in the isotropic case. The second aim of
the paper is to prove
multivariate central limit theorems including Berry--Esseen bounds.
These are based on a general normal approximation
result obtained by the Malliavin--Stein method.
\end{abstract}

%
\begin{keyword}[class=AMS]
\kwd[Primary ]{60D05}
\kwd[; secondary ]{60F05}
\kwd{60G55}
\kwd{60H07}
\kwd{52A22}
\end{keyword}
\begin{keyword}
\kwd{Boolean model}
\kwd{central limit theorem}
\kwd{covariance matrix}
\kwd{intrinsic volume}
\kwd{additive functional}
\kwd{curvature measure}
\kwd{Fock space representation}
\kwd{Malliavin calculus}
\kwd{Wiener--It\^o chaos expansion}
\kwd{Berry--Esseen bound}
\kwd{integral geometry}
\end{keyword}
\end{frontmatter}

\section{Introduction}\label{secintro}

Let $\eta$ be a \emph{stationary (locally finite) Poisson process} on
the space
$\cK^d$ of \emph{convex bodies} in $\R^d$, that is, on
the space of compact, convex subsets of $\R^d$.
The \emph{Boolean model} associated with $\eta$ is the stationary
random closed set $Z$ defined by
%
\begin{equation}
\label{Bool} Z:=\bigcup_{K\in\eta} K,
\end{equation}
where the Poisson process $\eta$ is identified with its support.
This is a fundamental model of stochastic geometry and continuum percolation
with many applications in materials science and physics
\cite{SKM,Hall88,MeesterRoy96,Molchanov96,SW08}.
The intersection of $Z$ with a compact and convex set $W\subset\R^d$ is
a finite union of compact, convex sets, that is,
an element of the \emph{convex ring} $\cR^d$. It is a common strategy in
stochastic geometry
to extract and explore local information about $Z$ via functionals
of the intersection $Z\cap W$. Perhaps the most
prominent examples of such functionals on $\cR^d$ are the
intrinsic volumes $V_0,\ldots,V_d$,
which contain important geometric information about the sets to which
they are applied.
For instance,
for a set $K\subset\R^d$ from the convex ring, $V_d(K)$ is
the volume, $V_{d-1}(K)$ is half the surface area (if $K$ is the
closure of its interior),
and $V_0(K)$ is the \emph{Euler characteristic} of $K$;
see \cite{SW08}, Section~14.2, for more details. The intrinsic volumes
have several desirable properties.
In particular, they are \emph{additive}, in the
sense that $V_i(K\cup L)=V_i(K)+V_i(L)-V_i(K\cap L)$ for
all $K,L\in\cR^d$ and $i\in\{0,\ldots,d\}$. They are also
translation invariant, and continuous if restricted to the space of
convex bodies.

For a stationary and isotropic Boolean model, Miles \cite{Miles76} and
Davy \cite{Davy76}
obtained explicit formulas expressing the mean values
$\BE V_i(Z\cap W)$ in terms of the intensity measure of $\eta$.
We refer to \cite{SW08}, Section~9.1, for a discussion and more recent
developments related to
this fundamental result.

In the following, we are especially interested in second-order
properties and central limit theorems of the random vector $(V_0(Z\cap
W),\ldots,V_d(Z\cap W))$, for a compact and convex \emph{observation
window} $W$, but in fact we study more general additive functionals
of $Z\cap W$, namely so called \emph{geometric functionals}. A
functional on the convex ring will be called geometric if it is
additive, translation invariant, locally bounded, and measurable (see
Section~\ref{sec:3} for details).

While previous contributions focus on second-order properties and
central limit theorems for volume and surface area, to the best of our
knowledge we present here the first systematic mathematical
investigation of second-order properties and central limit theorems of
all intrinsic volumes and more general geometric functionals of a
stationary Boolean model $Z$. The volume functional was first studied
in \cite{Badd80,Mase82}, while in \cite{Heinrich2005} Berry--Esseen
bounds and large deviation inequalities were established. The surface
area was investigated in \cite{Molchanov95}, and the results were
extended in~\cite{HM99} to more general functionals and point
processes. Integrals over Boolean models are considered in
\cite{BaYuk05,Penrose07}, where the volume is included as a special
case and also the surface area in the latter one. Volume and surface
area of a more general Boolean model based on a Poisson process of
cylinders have been investigated in \cite{HeiSp09,HeiSp13}. From a
geometric point of view, volume and surface area are rather special
functionals of $Z$. They arise as the restriction of deterministic
measures to $Z$ or the boundary of $Z$ and do not involve the
curvature of the (possibly intersecting) grains. A different though
mathematically nonrigorous treatment of second moments of curvature
measures of an isotropic Boolean model with an interesting application
to morphological thermodynamics was presented in \cite{Mecke01}.

Our first main aim in this paper is to use the Fock space
representation of Poisson functionals \cite{LaPe11} to explore the
covariance structure of geometric functionals of $Z\cap W$. Combined
with some new integral geometric inequalities, which are derived by
methods and results from convex and integral geometry, this approach
appears to be perfectly tailored to our purposes. Under the minimal
assumption that the second moments of the intrinsic volumes of the
typical grain are finite, we show that for two geometric functionals
$\psi_1$ and $\psi_2$ the ratio $V_d(W)^{-1}\C(\psi_1(Z\cap
W),\psi_2(Z\cap W))$ tends to some limit $\sigma_{\psi_1,\psi_2}\in
\R$
as the observation window is increased in a proper way. For the case
that the third moments of the intrinsic volumes of the typical grain
are finite, we establish a rate of this convergence in terms of the
inradius of the observation window and show that it is optimal. Via
the Fock space representation the asymptotic covariances can be
expressed as series of second moments. In the important case of
intrinsic volumes of an isotropic Boolean model they can be
represented in terms of curvature based moment measures of the typical
grain. In particular, the covariance structure of the two-dimensional
isotropic Boolean model becomes surprisingly explicit. For a vector of
geometric functionals of the Boolean model, it is shown that the
asymptotic covariance matrix is positive definite under some
additional conditions, which are for example satisfied for the
intrinsic volumes. The second-order analysis is illustrated by
explicit formulas for intrinsic volumes of a Boolean model with
deterministic spherical grains, for which our formulas reduce to
three-dimensional integration of explicitly known integrands.

Our second main aim is to prove univariate and multivariate central
limit theorems for geometric functionals of $Z\cap W$. Under the same
second moment assumptions as for the existence of the asymptotic
covariances, we prove convergence in distribution. We also obtain
rates of convergence under slightly stronger moment assumptions. For
the multivariate central limit theorem, we argue that the rate is
optimal. Following common belief, we guess that our convergence rate
$V_d(W)^{-1/2}$ for the univariate case is optimal as well. In the
univariate case, we do not need to assume that the functional on the
convex ring is translation invariant. In the proofs, we use the
Malliavin--Stein method for Poisson functionals that was recently
developed in \cite{Peccatietal2010,PeccatiZheng2010}. In a sense, this
method builds on the Fock space representation and the closely related
Wiener--It\^o chaos expansion of Poisson functionals. The main
obstacle to the application of these results is the fact that, as a
rule, geometric functionals of $Z$ admit an infinite chaos
expansion. We can resolve this by bounding the kernels of the chaos
expansion by monotone functionals.

In the case of bounded grains, it is likely that
the central limit theorem and the convergence of covariances
can be derived with the theory of $m$-dependent random fields, perhaps
even with rates of convergence. From there, one might proceed to the general
case using a truncation argument as in \cite{Heinrich1993,HeiSp13}.
But such an approach would neither yield much information
on the asymptotic covariance structure nor rates of convergence in the
general case.
Stabilization is another common approach to central limit theorems
in stochastic geometry. We refer here to
\cite{BaYuk05,Penrose07,PenroseYukich2005},
where the first two references deal in particular with volume and
surface area
of
the Boolean model without discussing rates of convergence. It is
unclear whether
the intrinsic volumes (other than volume or surface area)
stabilize for Boolean models with unbounded grains. But even if they do,
the quantitative bounds for the normal approximation derived by stabilization
in \cite{PenroseYukich2005} suggest that the rates would probably be
suboptimal.
Moreover, in
our setting we would need to control boundary effects.

This paper is organized as follows. In the second section, we briefly
summarize some notation and basic facts about the Boolean model and
present a central limit theorem for the intrinsic volumes of the
Boolean model
to be generalized later.
In the third section, we establish the existence of the
asymptotic covariances of a vector of geometric functionals of $Z\cap
W$ and determine the rate of convergence; see Theorem~\ref{thm:variance}. Section~\ref{secpos} is devoted to the positive
definiteness of the asymptotic covariance matrix; see Theorem~\ref{posdef}. In Section~\ref{sintegra}, we focus on intrinsic volumes
and introduce a family of curvature based moment measures of the
typical grain to study infinite series of second moments arising in
the Fock space representation. The main result of this section
(Theorem~\ref{tdeltaij}) is of some independent interest and is
applied in Section~\ref{seciso} to derive formulas for the asymptotic
covariances of the intrinsic volumes of an isotropic Boolean model in
terms of the moment measures mentioned above; see Theorem~\ref{tsigmaij}. Section~\ref{secspherical} presents some explicit
results for a Boolean model with deterministic spherical grains. In
Section~\ref{secclt}, we provide a general result on the normal
approximation of Poisson functionals. We use this result in Section~\ref{secmclt} to establish multivariate and univariate central limit
theorems for geometric functionals of $Z$; see Theorems
\ref{thm:multiGeneral} and~\ref{thm:univariat}.

The extended arXiv-version \cite{HLSarxiv} of this paper contains two
additional appendices with a description of the curvature based moment
measures from Section~\ref{sintegra} in terms of mixed measures of
translative integral geometry and with integral formulas for the exact
(nonasymptotic) covariances of intrinsic volumes, which are rather
explicit in the two-dimensional case.

\section{Boolean models and intrinsic volumes}\label{sec:2}

In this section, we collect a few basic facts about the stationary
Poisson process $\eta$ and the associated Boolean model $Z$ in
Euclidean space $\R^d$ before
stating some of our main results for
the special case of intrinsic volumes. For more details on Boolean
models, we refer the reader to \cite{SKM}, Chapter~3,
\cite{Molchanov96} or \cite{SW08}, Chapter~4, whereas background
material on convex geometry can be found, for example, in \cite{Sch93}
or \cite{SW08}, Chapter~14. All random objects occurring in this
paper are defined on an abstract probability space
$(\Omega,\mathcal{F},\BP)$. A measure on $\cK^d$ is \emph{locally
finite} if it assigns a finite number to $\{K\in\cK^d\dvtx K\cap
C\ne\varnothing\}$ for all $C$ in the space $\cC^d$ of compact subsets
of $\R^d$. We consider the Poisson process $\eta$ as a random element
in the space $\bN$ of all locally finite counting measures on $\cK^d$,
equipped with the smallest $\sigma$-field such that the mappings
$\mu\mapsto\mu(A)$ are measurable for all $A$ in the Borel
$\sigma$-field (with respect to the Hausdorff metric) of $\cK^d$. We
assume that the \emph{intensity measure} $\Lambda:=\BE\eta$ of
$\eta$
is invariant under the shifts $K\mapsto K+x:=\{y+x\dvtx y\in K\}$,
$x\in\R^d$. This is equivalent to the \emph{stationarity} of $\eta$,
that is to the distributional invariance of $\eta$ under all
shifts. We also assume that $\Lambda$ is nontrivial and that
$\Lambda(\{\varnothing\})=0$, which effectively excludes empty grains.
Theorem~4.1.1 in \cite{SW08} implies that
%
\begin{equation}
\label{Lambda} \Lambda(\cdot)=\gamma\iint\I\{K+x\in\cdot\} \,dx \BQ(dK),
\end{equation}
where $\gamma\in(0,\infty)$ is the \emph{intensity} of $\eta$, ``$dx$''
denotes integration with respect to the $d$-dimensional Lebesgue
measure $\lambda_d$, and
$\BQ$ is a probability measure on $\cK^d$ satisfying $\BQ(\{
\varnothing\}
)=0$ as
well as
%
\begin{equation}
\label{int} \int V_d(K+C) \BQ(dK)<\infty,\qquad C\in\cC^d.
\end{equation}
Here, as usual, $K+C:=\{x+y\dvtx x\in K,y\in C\}$ is the \emph{Minkowski sum}
of $K$ and~$C$.
Let $Z_0$ denote a \emph{typical grain}, that is, a random
convex set with distribution $\BQ$. Then \eqref{int} can be written
as
%
\begin{equation}
\label{vi} v_i:=\BE V_i(Z_0)<\infty,\qquad
i=0,\ldots,d.
\end{equation}
This is a direct consequence of the \emph{Steiner formula} (see \cite{SW08},
equation (14.5))
%
\begin{equation}
\label{steiner} V_d\bigl(K+B^d_r\bigr)=\sum
_{i=0}^d \kappa_{d-i}
r^{d-i} V_i(K),\qquad r\ge0, K\in \cK^d,
\end{equation}
where $B^d$ is the closed unit ball centered at the origin, $B_r^d:=\{
rx\dvtx  x\in B^d\}$, and
$\kappa_n$ denotes the volume of the $n$-dimensional unit ball.

The Boolean model is given by $Z\equiv Z(\eta)$, where
\[
Z(\mu):=\bigcup_{K\in\mu} K,\qquad \mu\in\bN,
\]
and $K\in\mu$ means that $\mu(\{K\})>0$. Recall that the mapping
$\mu
\mapsto Z(\mu)$ from $\bN$ to the space
of all closed subsets of $\R^d$ (equipped with the Fell topology)
is Borel measurable (see \cite{SW08}, Theorem~3.6.2).
Without loss of generality, we can assume that $\BQ$ is concentrated on
$\cK^d_o$, where $\cK_o^d$ is
the space of nonempty convex bodies such that the center of the
circumscribed ball is at the origin. Since the center of the
circumscribed ball of a convex body is always contained in the convex
body, we have $0\in K$ for all $K\in\cK^d_o$.

Subsequently, we shall need integrability assumptions such as
%
\begin{equation}
\label{vi2} \BE V_i(Z_0)^2<\infty,\qquad i=0,
\ldots,d,
\end{equation}
or
%
\begin{equation}
\label{vi3} \BE V_i(Z_0)^3<\infty,\qquad i=0,
\ldots,d.
\end{equation}

We next introduce two basic characteristics of the Boolean model
$Z$. The \emph{volume fraction} $p:=\BE V_d(Z\cap[0,1]^d)$
of $Z$ can be expressed in the form
$p=1-e^{-\gamma v_d}$.
The mean \emph{covariogram} of the typical grain is given by
%
\begin{equation}
\label{meancov} C_d(x):=\BE V_d\bigl(Z_0
\cap(Z_0+x)\bigr),\qquad x\in\R^d.
\end{equation}
It follows from \eqref{vi} that $C_d(x)\le v_d<\infty$ and that
$C_d(x)\to0$ as $\|x\|\to\infty$, where $\|x\|$ denotes the Euclidean
norm of $x\in\R^d$.
It is well known (see, e.g., \cite{SKM}, equation (3.18)) that the
\emph{covariance} of $Z$ satisfies
%
\begin{equation}
\label{eqn:covarianceZ} \BP(0\in Z,x\in Z)=p^2+(1-p)^2
\bigl(e^{\gamma C_d(x)}-1 \bigr).
\end{equation}
For $W\in\cK^d$, we define by $C_W(x):=V_d(W\cap(W+x))$, $x\in\R^d$,
the set covariance function of $W$. Combining \eqref{eqn:covarianceZ}
with Fubini's theorem leads to the well-known formula
%
\begin{equation}
\label{eqn:VarianceVolume} \V V_d(Z\cap W) = (1-p)^2 \int
C_W(x) \bigl(e^{\gamma C_d(x)}-1 \bigr) \,dx, \qquad W\in\cK^d.
\end{equation}

Throughout this paper, we investigate the intersection $Z\cap W$ of
the Boolean model $Z$ with an expanding compact convex observation
window $W$. More precisely, we consider sequences of convex bodies
$(W_m)_{m\in\N}$ satisfying $\lim_{m\to\infty} r(W_m)=\infty$, where
$r(W)$ denotes the \emph{inradius} of $W\in\cK^d$. We describe this
situation by writing $r(W)\to\infty$ for short.
Combining our Theorems~\ref{thm:multiGeneral} and~\ref{posdef}
in the special case of intrinsic volumes of $Z\cap W$,
we obtain the following multivariate central
limit theorem.

\begin{theorem}
Assume that \eqref{vi2} is satisfied and let
$\mathcal{V}:=(V_0,\ldots,V_d)$. Then there exists a
$(d+1)$-dimensional centered Gaussian random vector $N$ with a
covariance matrix $\Sigma$ such that
\[
\frac{1}{\sqrt{V_d(W)}} \bigl(\mathcal{V}(Z\cap W)-\BE\mathcal{V}(Z\cap W)\bigr)
\stackrel{d} {\longrightarrow} N \qquad\mbox{as } r(W)\to\infty.
\]
If, additionally, the typical grain $Z_0$ has nonempty interior with
positive probability, the covariance matrix $\Sigma$ is positive
definite.
\end{theorem}

To the best of our knowledge,
this is the first central limit theorem for the intrinsic volumes
of the Boolean model
beyond volume and surface area. In fact,
our Theorem~\ref{thm:multiGeneral} generalizes this result in
several ways.
It concerns a broader class of functionals
and is also quantitative in the sense that it
provides rates of convergence for a suitable distance
under moment conditions slightly stronger than \eqref{vi3}.
Theorem~\ref{thm:univariat} yields presumably optimal rates for the Wasserstein
distance in the univariate case. As already mentioned above, our proofs
rely on the
Malliavin--Stein method for Poisson functionals.
We are not aware of any other approach that might yield
the same rates.
In Section~\ref{seciso},
we will derive formulas
for the asymptotic covariances between the intrinsic volumes
of an isotropic Boolean model.

\section{Covariance structure of geometric functionals}\label{sec:3}

In this paper, we study random variables of the form $\psi(Z\cap W)$,
where $\psi$ is a real-valued measurable function defined
on the \emph{convex ring} $\cR^d$ whose elements are finite unions of
compact, convex sets.
Measurability again refers to the Borel $\sigma$-field generated by
the Fell topology (or, equivalently, by the Hausdorff metric).
We shall assume that $\psi$ is \emph{additive}, that is,
$\psi(\varnothing)=0$ and $\psi(K\cup L)=\psi(K)+\psi(L)-\psi(K\cap L)$
for all $K,L\in\cR^d$. We shall also assume that $\psi$ is
\emph{translation invariant}, that is, $\psi(K+x)=\psi(K)$ for all
$(K,x)\in\cR^d\times\R^d$, and \emph{locally bounded} in the sense
that its absolute value is (uniformly) bounded on compact, convex sets
contained in a translate of the unit cube $Q_1:=[-1/2,1/2]^d$ by a constant
%
\begin{equation}
\label{eqn:locallybounded} M(\psi):=\sup\bigl\{ \bigl|\psi(K)\bigr|\dvtx K\in\cK^d, K
\subset Q_1+x, x\in\R^d\bigr\} <\infty.
\end{equation}
Note that this definition simplifies in the translation-invariant case
since one does not need the translations of $Q_1$.

In the following, we call a functional $\psi\dvtx \mathcal{R}^d\to\R$
\emph{geometric} if it is additive, translation invariant, locally bounded
and measurable. Examples of geometric functionals are
(1)~mixed volumes (see Section~5.1 in \cite{Sch93}) of the form $\psi
(K):=V(K[k],K_1,\ldots,K_{d-k})$,
where $k\in\{0,\ldots,d\}$, the notation $K[k]$ means that the body $K$
is repeated $k$ times, and $K_1,\ldots,K_{d-k}\in\mathcal{R}^d$ are fixed.
Up to normalization, intrinsic volumes are obtained for $K_i=B^d$,
$i=1,\ldots,d-k$;
(2)~integrals of surface area measures (see Sections~4.1 and 4.2 in
\cite{Sch93}) of the form
$
\psi_k(K):=\int_{\S^{d-1}}h(u) S_k(K,du)$,
where $\S^{d-1}$ is the unit sphere in $\R^d$ (the boundary of $B^d$),
$h\dvtx \S^{d-1}\to\R$ is measurable and bounded, and $k\in\{0,\ldots
,d-1\}$;
(3)~the centered support function $\psi(K):=h(K-s(K),u)$ in a fixed
direction $u$, where $u\in\R^d$
and $s(K)$ is the Steiner point of $K$ (see Section~1.7, Section~5.4,
equation~(5.100) in \cite{Sch93} and Lemma~6.1 in \cite{SW08});
(4)~total measures of translative integral geometry (see Section~6.4,
especially page~234, and page~383 in \cite{SW08}). These examples of
geometric functionals are substantially more general than the intrinsic
volumes. For instance, whereas intrinsic volumes are always rotation
invariant, no such invariance is built into these four classes of
examples in general. Moreover, it should be observed that linear
combinations of mixed volumes are dense in the (normed) vector space of
translation invariant, continuous valuations on convex bodies in $\R^n$
(see Section~6.5, page~406, in \cite{Sch93}).

Our main result of this section deals with the asymptotic behavior of the
covariance between two geometric functionals of $Z\cap W$
for expanding convex observation window $W$.
With a measurable functional $\psi\dvtx \cR^d\rightarrow\R$, we associate
another measurable function $\psi^*\dvtx \cK^d\to\R$ by
%
\begin{equation}
\label{psi*} \psi^*(K):=\BE\psi(Z\cap K)-\psi(K), \qquad K\in\cK^d,
\end{equation}
if $\BE|\psi(Z\cap K)|<\infty$ for all $K\in\cK^d$. Under assumption
\eqref{int}, it follows from~\eqref{smfinite} below that $\psi^*$ is
well defined for a geometric functional $\psi$. For $C\in\cC^d$ let
$N_C$ denote the number of all particles in $\eta$ intersecting $C$.

\begin{theorem}\label{thm:variance}
Let $\psi_1$ and $\psi_2$ be geometric functionals. If \eqref{vi2} is
satisfied, then the limit
%
\begin{equation}
\label{eqn:limitCovariance} \sigma_{\psi_1,\psi_2}= \lim_{r(W)\rightarrow\infty}
\frac{\C(\psi_1(Z\cap W),\psi
_2(Z\cap W))}{V_d(W)}
\end{equation}
exists, is finite, and is given by
%
\begin{eqnarray}
\label{eqn:sigma} \sigma_{\psi_1,\psi_2}&=&\gamma\sum_{n=1}^\infty
\frac{1}{n!} \iint\psi^*_1(K_1\cap K_2
\cap\cdots\cap K_n) \psi^*_2(K_1\cap
K_2\cap\cdots\cap K_n)
\nonumber
\\[-8pt]
\\[-8pt]
\nonumber
&&\hspace*{55pt}{}\times
\Lambda^{n-1}\bigl(d(K_2,\ldots,K_n)\bigr)
\Q(dK_1).
\end{eqnarray}
Assume that \eqref{vi3} holds and define
\begin{eqnarray*}
c_{\Lambda}&:=& 2^{d+2} \cdot4^{2^d}\cdot25^{2d}
d! \bigl(\BE2^{N_{Q_1}} +1 \bigr)^2 \\
&&{}\times\exp \Biggl(2^{2^d}
\cdot25^d (d+1)! \gamma\sum_{i=0}^d
\BE V_i(Z_0) \Biggr) \gamma\BE \Biggl(\sum
_{i=0}^d V_i(Z_0)
\Biggr)^3.
\end{eqnarray*}
Then, for $W\in\cK^d$ with $r(W)\geq1$,
%
\begin{equation}
\label{eqn:AsymptoticCovariances} \biggl\llvert \frac{\C(\psi_1(Z\cap W),\psi_2(Z\cap W))}{V_d(W)}-\sigma _{\psi
_1,\psi_2}\biggr
\rrvert \le\frac{c_{\Lambda} M(\psi_1)M(\psi_2)}{r(W)}.
\end{equation}
\end{theorem}

We start with some preparations. Our main probabilistic tool is
the following Fock space representation of Poisson functionals, derived
in \cite{LaPe11}.
For any measurable $f\dvtx \bN\rightarrow\R$ and $K\in\cK^d$, the function
$D_Kf\dvtx \bN\to\R$ is defined by
%
\begin{equation}
\label{addone} D_{K}f(\mu):=f(\mu+\delta_{K})-f(\mu),\qquad  \mu
\in\bN,
\end{equation}
where $\delta_{K}$ is the Dirac measure located at $K$.
The \emph{difference operator} $D_K$ and its iterations
play a central role in the analysis of Poisson processes.
For $n\ge2$ and $(K_1,\ldots,K_n)\in(\cK^d)^n$
we define a function
$D^{n}_{K_1,\ldots,K_n}f\dvtx \bN\rightarrow\R$ inductively by
\[
D^{n}_{K_1,\ldots,K_{n}}f:=D^1_{K_{1}}D^{n-1}_{K_2,\ldots,K_{n}}f,
\]
where $D^1:=D$.
Note that
\[
D^n_{K_1,\ldots,K_n}f(\mu)= \sum_{J \subset\{1,2,\ldots,n\}}(-1)^{n-|J|}
f \biggl(\mu+\sum_{j\in J}\delta_{K_j}
\biggr),
\]
where $|J|$ denotes the number of elements of $J$. This
shows that the operator $D^n_{K_1,\ldots,K_n}$
is symmetric in $K_1,\ldots,K_n$, and that
$(\mu,K_1,\ldots,K_n)\mapsto D^n_{K_1,\ldots,K_n}f(\mu)$
is measurable.
From \cite{LaPe11}, Theorem~1.1, we obtain for any
measurable $f,g\dvtx \bN\rightarrow\R$ satisfying $\BE f(\eta)^2<\infty$
and $\BE g(\eta)^2<\infty$ that
%
\begin{eqnarray}
\label{2.8}&& \C\bigl(f(\eta),g(\eta)\bigr)
\nonumber
\\[-8pt]
\\[-8pt]
\nonumber
&&\qquad= \sum^\infty_{n=1}
\frac{1}{n!}\int\BE D^n_{K_1,\ldots,K_n} f(\eta) \BE
D^n_{K_1,\ldots,K_n} g(\eta) \Lambda^n
\bigl(d(K_1,\ldots,K_n)\bigr).
\end{eqnarray}

For given $W\in\cK^d$ and a functional $\psi\dvtx \cR^d\rightarrow\R$
we shall apply \eqref{2.8} to functions $f_{\psi,W}\dvtx \bN\rightarrow
\R$
defined by $f_{\psi,W}(\mu):=\psi(Z(\mu)\cap W)$. Induction yields the
following lemma.

\begin{lemma}\label{lem:kernels}
Let $\psi\dvtx \cR^d\rightarrow\R$ be additive. Then, for $n\in\N$,
$K_1,\ldots,K_n\in\cK^d$, and $\mu\in\bN$,
\begin{eqnarray*}
&&D^n_{K_1,\ldots,K_n}f_{\psi,W}(\mu)\\
&&\qquad =(-1)^n\bigl(
\psi\bigl(Z(\mu)\cap K_1\cap\cdots\cap K_n\cap W\bigr)-
\psi(K_1\cap \cdots \cap K_n\cap W)\bigr).
\end{eqnarray*}
\end{lemma}

\begin{lemma}\label{lem:boundintrinsic}
Let $\psi$ be an additive, locally bounded and measurable functional
and assume that \eqref{int} is satisfied.
Then, for all $n\in\N$, $K_1,\ldots,K_n\in\mathcal{K}^d$, and
$W\in\cK^d$,
%
\begin{equation}
\label{Tn} \BE D^n_{K_1,\ldots,K_n} f_{\psi,W}(\eta)
=(-1)^n\psi^*(K_1\cap\cdots\cap K_n\cap W)
\end{equation}
and
%
\begin{equation}
\label{est3a} \bigl|\E D^n_{K_1,\ldots,K_n}f_{\psi,W}(\eta)\bigr| \le
\beta(\psi) \sum_{i=0}^d
V_i(K_1\cap\cdots\cap K_n\cap W)
\end{equation}
with $\beta(\psi)=2^{2^d} \cdot5^d M(\psi)  ( \BE
2^{N_{Q_1}}+1
)$. Moreover, for any $A\in\cK^d$,
%
\begin{equation}
\label{smfinite} \E\psi(Z\cap A)^2<\infty.
\end{equation}
\end{lemma}

\begin{pf}
We start by proving that there is a constant $c_1>0$ such that
%
\begin{equation}
\label{2.1a} \E\bigl|\psi(Z\cap A)\bigr| \leq c_1 M(\psi) \sum
_{i=0}^d V_i(A)
\end{equation}
for $A\in\cK^d$. Since \eqref{2.1a} is obviously true for
$A=\varnothing$,
we assume $A\neq\varnothing$ in the following.
We define ${\mathcal Q}(A):=\{Q_1+z\dvtx z\in\mathbb{Z}^d, (Q_1+z)\cap
A\neq\varnothing\}$.
By the inclusion--exclusion formula for additive functionals (see,
e.g., \cite{Sch93}, (6.2), page~330),
we have
\[
\bigl|\psi(Z\cap A)\bigr| = \biggl|\psi \biggl(Z\cap A\cap\bigcup_{Q\in
{\mathcal
Q}(A)}
Q \biggr) \biggr| \le\sum_{\varnothing\neq\cI\subset{\mathcal Q}(A)}
\biggl|\psi \biggl(Z\cap A\cap
\bigcap_{Q\in\cI} Q \biggr)\biggr |.
\]
For each nonempty\vspace*{1pt} subset $\cI\subset{\mathcal Q}(A)$, we fix some cube
$Q_{\cI}\in\cI$. Let $Z_1,\ldots,Z_{N_{Q_{\cI}}}$ denote the
particles hitting $Q_{\cI}$. Then, for $\varnothing\neq J \subset\{
1,\ldots,N_{Q_{\cI}}\}$,
assumption \eqref{eqn:locallybounded} yields that
\[
\biggl|\psi \biggl(\bigcap_{j\in J}Z_j\cap A
\cap\bigcap_{Q\in\cI}Q \biggr) \biggr|\le M(\psi).
\]
By the inclusion--exclusion formula and taking into account that $\psi
(\varnothing)=0$, we get
%
\begin{eqnarray}
\label{ref1} \bigl|\psi(Z\cap A)\bigr| &\le&
\sum_{\varnothing\neq\cI\subset{\mathcal Q}(A)} \Biggl|\psi
\Biggl(\bigcup_{j=1}^{N_{Q_{\cI}}}Z_j
\cap A\cap\bigcap_{Q\in\cI} Q \Biggr) \Biggr|
\nonumber
\\
&\le&\sum_{\varnothing\neq\cI\subset{\mathcal Q}(A)}\sum_{\varnothing
\neq
J\subset\{1,\ldots,N_{Q_{\cI}}\}}
\biggl|\psi \biggl(\bigcap_{j\in J}Z_j\cap A\cap
\bigcap_{Q\in\cI}Q \biggr) \biggr|\nonumber
\\
&\le&\sum_{\varnothing\neq\cI\subset{\mathcal Q}(A)} \mathbf {1} \biggl\{ \bigcap
_{Q\in\mathcal{I}}Q\neq\varnothing \biggr\} 2^{N_{Q_{\cI}}} M(
\psi).
\end{eqnarray}
The cubes in ${\mathcal Q}(A)$ form a grid, hence
\[
\biggl| \biggl\{\varnothing\neq\cI\subset{\mathcal Q}(A)\dvtx \bigcap
_{Q\in\cI
}Q\neq\varnothing \biggr\} \biggr| \le c_2 \bigl|{
\mathcal Q}(A)\bigr|
\]
with $c_2:=2^{2^d}$. By stationarity of $\eta$, we have $ \BE
2^{N_{Q_{\cI}}}= \BE2^{N_{Q_1}}$, and thus
%
\begin{equation}
\label{est1} \BE\bigl| \psi(Z\cap A)\bigr| \leq c_3 M(\psi)
\bigl|{\mathcal Q}(A)\bigr|
\end{equation}
with $c_3:=c_2\BE2^{N_{Q_1}}$. Here, we have used that
$
\BE z^{N_C}=\exp ((z-1)\gamma\BE V_d(Z_0+C^*) )<\infty
$
holds for $C\in\cK^d$ and $z\ge0$,
where $C^*:=\{-x\dvtx x\in C\}$ is the reflection of $C$ in the origin.

Since $|{\mathcal Q}(A)| \leq V_d(A+ \sqrt{d}B^d)$, Steiner's formula
\eqref{steiner} yields
%
\begin{equation}
\label{est0} \bigl|{\mathcal Q}(A)\bigr| \leq\sum_{i=0}^d
\kappa_{d-i} d^{(d-i)/2} V_i(A)\leq c_4
\sum_{i=0}^d V_i(A)
\end{equation}
with $c_4:=5^d$. In the last step, we used that $\kappa
_{d-i}d^{(d-i)/2}\leq5^d$, $i\in\{0,\ldots,d\}$, which can be deduced
by elementary calculus from the representation of $\kappa_{d-i}$ in
terms of the Gamma function and from Stirling's formula. Now \eqref
{est1} together with \eqref{est0} yields \eqref{2.1a} with
$c_1:=c_3c_4$. Combining \eqref{2.1a} with Lemma~\ref{lem:kernels} and
the definition of $\psi^*$ in \eqref{psi*} shows \eqref{Tn}. For
$A\in
\cK^d$, we can argue as in the derivation of \eqref{ref1}, and then
use \eqref{est0}, to get
%
\begin{equation}
\label{est0a} \bigl|\psi(A)\bigr| \le c_2 \bigl|\mathcal{Q}(A)\bigr|M(\psi)\le
c_{2}c_{4} M(\psi)\sum_{i=0}^d
V_i(A).
\end{equation}
Combining \eqref{2.1a} and \eqref{est0a} for $A=K_1\cap\cdots\cap
K_n\cap W$ with Lemma~\ref{lem:kernels} yields~\eqref{est3a}.

In order to show that $\psi(Z\cap A)$ is square integrable, we first
derive an upper bound for
%
\begin{equation}
\label{Mloc} M_A(\psi):=\sup\bigl\{\bigl|\psi(L)\bigr|\dvtx L\in
\mathcal{K}^d,L\subset A\bigr\}.
\end{equation}
Let $L\in\mathcal{K}^d$ with $L\subset A$. Then, using the
inclusion--exclusion formula for additive functionals
and \eqref{eqn:locallybounded}, we get
\[
\bigl|\psi(L)\bigr|= \biggl|\psi \biggl(L\cap\bigcup_{Q\in\mathcal{Q}(A)} Q
\biggr)\biggr |\le2^{|\mathcal{Q}(A)|}M(\psi),
\]
and hence
$
M_A(\psi)\le2^{|\mathcal{Q}(A)|}M(\psi)$.
Again by the inclusion--exclusion formula, we have
%
\begin{equation}
\label{zitier} \bigl|\psi(Z\cap A)\bigr|\le\bigl(2^{N_A}-1\bigr)M_A(
\psi)\le2^{N_A}M_A(\psi),
\end{equation}
and, therefore,
\[
\E\psi(Z\cap A)^2\le\E \bigl[4^{N_A} \bigr]
4^{|\mathcal
{Q}(A)|}M(\psi )^2<\infty,
\]
which completes the proof.
\end{pf}

\begin{lemma}\label{lem:translativeIntegral}
Define $\beta_1:=2^{2^d} \cdot25^d d!$. Then, for all $k\in\{
0,\ldots
,d\}$ and $W,K\in\cK^d$,
%
\begin{equation}
\int V_k\bigl(W\cap(K+x)\bigr) \,d x  \leq\beta_1 \sum
_{i=0}^d V_i(W) \sum
_{r=k}^d V_r(K) \label{tra1}.
\end{equation}
\end{lemma}

\begin{pf}
Using the same notation as in the proof of Lemma~\ref{lem:boundintrinsic} and the fact that $V_k$
is increasing and translation invariant, we obtain that
\begin{eqnarray*}
\int V_k\bigl(W\cap(K+x)\bigr) \,d x & \leq&\sum
_{\varnothing\neq{\cal I}\subset
{\cal Q}(W)} \int V_k \biggl(W\cap\bigcap
_{Q\in{\cal I}} Q \cap (K+x) \biggr) \,dx
\\
& \leq&\sum_{\varnothing\neq{\cal I}\subset{\cal Q}(W)} \I \biggl\{ \bigcap
_{Q\in{\cal I}} Q\neq\varnothing \biggr\} \int V_k\bigl(K
\cap(Q_1 + x)\bigr) \,dx.
\end{eqnarray*}
Let $B'$ denote a ball of radius $\sqrt{d}/2$. Then the kinematic
formula (see \cite{SW08}, Theorem~5.1.3, and note that
$c_{j,d}^{k,d-k+j}\leq1$) and the rotation invariance of $B'$ yield that
\[
\int V_k\bigl(K\cap(Q_1 + x)\bigr) \,dx \leq\int
V_k\bigl(K\cap\bigl(B'+x\bigr)\bigr) \,dx \le
c_5 \sum_{r=k}^d
V_r(K)
\]
with $c_5:=5^d d!$. On the other hand, it was shown in the proof of
Lemma~\ref{lem:boundintrinsic} that
\[
\biggl| \biggl\{\varnothing\neq{\cal I} \subset{\cal Q}(W)\dvtx \bigcap
_{Q\in
{\cal I}} Q \neq\varnothing \biggr\} \biggr|\leq c_2
c_4 \sum_{i=0}^d
V_i(W).
\]
Combining the preceding inequalities, we obtain the assertion of the lemma.
\end{pf}

\begin{lemma}\label{lem:kinematic}
For $A\in\cK^d$ and $n\in\N$,
\[
\int\sum_{k=0}^d V_k(A\cap
K_1\cap\cdots\cap K_n) \Lambda ^n
\bigl(d(K_1,\ldots,K_n)\bigr) \leq\alpha^n
\sum_{k=0}^d V_k(A),
\]
where $\alpha=\gamma(d+1)\beta_1 \sum_{i=0}^d \BE V_i(Z_0)$ with
$\beta_1$ as in Lemma~\ref{lem:translativeIntegral}.
\end{lemma}

\begin{pf}
In the following calculation and also later, we use the convention
$\int c\, d\Lambda^0:=c$. We apply \eqref{Lambda} and \eqref{tra1} to get
\begin{eqnarray*}
&&\int\sum_{k=0}^d V_k(A\cap
K_1\cap\cdots\cap K_n) \Lambda ^n
\bigl(d(K_1,\ldots,K_n)\bigr)
\\
&&\qquad=\sum_{k=0}^d \gamma\iiint
V_k\bigl(A\cap K_1\cap\cdots\cap K_{n-1}\cap
(K_n+y)\bigr) \,dy \BQ(dK_n)\\
&&\hspace*{81pt}{}\times \Lambda^{n-1}
\bigl(d(K_1,\ldots,K_{n-1})\bigr)
\\
&&\qquad\le\sum_{k=0}^d \gamma\iint
\beta_1\sum_{i=0}^d
V_i(A\cap K_1\cap \cdots\cap K_{n-1})\\
&&\hspace*{73pt}{}\times \sum
_{r=k}^d V_r(K_n)
\BQ(dK_n) \Lambda ^{n-1}\bigl(d(K_1,
\ldots,K_{n-1})\bigr)
\\
&&\qquad \le\gamma(d+1) \beta_1 \sum_{i=0}^d
\BE V_i(Z_0) \int\sum_{k=0}^d
V_k(A\cap K_1\cap\cdots\cap K_{n-1})\\
&&\hspace*{166pt}{}\times\Lambda^{n-1}\bigl(d(K_1,\ldots,K_{n-1})\bigr).
\end{eqnarray*}
By iterating this step $(n-1)$ more times, we obtain the assertion.
\end{pf}

\begin{lemma}\label{bounds}
Define $\beta_2:=2\cdot25^d$. Then, for $K,W\in\cK^d$,
\[
\lambda_d\bigl(\bigl\{x\in\R^d\dvtx (K+x)\cap\partial W
\neq\varnothing\bigr\}\bigr)\le\beta _2 \sum
_{i=0}^{d-1}V_i(W)\sum
_{r=0}^dV_r(K).
\]
\end{lemma}

\begin{pf} Let $W\neq\varnothing$ and let $\mathcal{Q}(\partial W):=\{
Q_1+z\dvtx z\in\mathbb{Z}^d,(Q_1+z)\cap\partial W\neq\varnothing\}$. Then
we have
\begin{eqnarray*}
&&\lambda_d\bigl(\bigl\{x\in\R^d\dvtx (K+x)\cap\partial W
\neq\varnothing\bigr\}\bigr) \\
&&\qquad\le\sum_{Q\in\mathcal{Q}(\partial W)}\int \I
\bigl\{(K+x)\cap Q\neq\varnothing\bigr\} \,dx
\\
&&\qquad= \sum_{Q\in\mathcal{Q}(\partial W)}V_d(K+Q_1)
\le\bigl|\mathcal{Q}(\partial W)\bigr| c_4\sum_{r=0}^d
V_r(K)
\end{eqnarray*}
with the same constant $c_4$ as in \eqref{est0}. Let $\dist
(x,A):=\inf\{
\| x-y\|\dvtx y\in A\}$ for $x\in\R^d$ and a closed set $A\subset\R^d$,
and let
$\partial^-_r W:=\{x\in W\dvtx \dist(x,\partial W)\le r\}$ for $r\geq0$. Then
%
\begin{equation}
\label{innerest} V_d \bigl(\partial^-_{r} W \bigr)\le
V_d\bigl(W+B^d_{r}\bigr)-V_d(W).
\end{equation}
To see this, let $p_W\dvtx \R^d\to W$ denote the metric projection to $W$
and consider the map
$T\dvtx (W+B^d_r)\setminus W\to\R^d$, $x\mapsto2p_W(x)-x$. Let
$x\in\partial^-_rW$ and choose a point $y\in\partial W$ such that
$\|x-y\|=\dist(x,\partial W)\le r$. Using that $y-x$ is an outer
normal of $W$ at $y$, it is easy to see that
$T(2y-x)=x$. Hence, $\partial^-_rW\subset T((W+B^d_r)\setminus W)$.
Since the metric projection is $1$-Lipschitz, it is not hard to
prove that $T$ has the same property. Therefore, \eqref{innerest}
follows. This yields that
\begin{eqnarray*}
\bigl|\mathcal{Q}(\partial W)\bigr|&\le&\lambda_d\bigl(\bigl\{x\in
\R^d\dvtx \operatorname {dist}(x,\partial W)\le\sqrt{d}\bigr\}\bigr)
\le 2\bigl(V_d\bigl(W+B^d_{\sqrt{d}}
\bigr)-V_d(W)\bigr)\\
&\le&2c_4\sum
_{i=0}^{d-1}V_i(W),
\end{eqnarray*}
where Steiner's formula was used.
\end{pf}

\begin{lemma}\label{lem:BoundInradius}
Let $W\in\cK^d$ be such that $r(W)>0$ and let $k\in\{0,\ldots,d-1\}
$. Then
\[
\frac{V_k(W)}{V_d(W)}\leq\frac{2^d-1}{\kappa_{d-k}r(W)^{d-k}}\leq \frac
{2^d d!}{r(W)^{d-k}}.
\]
\end{lemma}

\begin{pf}
Steiner's formula and the fact that $V_i(W)\geq0$, for $i=0,\ldots
,d-1$, imply that
\begin{eqnarray*}
\bigl(2^d-1\bigr)V_d(W)&=&V_d(2W)-V_d(W)\\
&\ge& V_d\bigl(W+r(W) B^d\bigr)-V_d(W)\\
&=&\sum
_{i=0}^{d-1}\kappa_{d-i}r(W)^{d-i}
V_i(W)\geq\kappa_{d-k} r(W)^{d-k} V_k(W).
\end{eqnarray*}
Now the inequality $\kappa_n\geq1/n!$, $n\in\N$, concludes the proof.
\end{pf}

\begin{pf*}{Proof of Theorem~\ref{thm:variance}}
Let $W\in\cK^d$ with $r(W)\geq1$.
In order to compute the numerator in \eqref{eqn:limitCovariance},
we shall apply \eqref{2.8} with $f=f_{\psi_1,W}$ and $g=f_{\psi_2,W}$.
From \eqref{smfinite}, we conclude that indeed $\E f(\eta)^2<\infty$
and $\E g(\eta)^2<\infty$.
Since $Z$ is stationary, the translation invariance of
a functional $\psi\dvtx \cR^d\to\R$ implies that
$\psi^*\dvtx \cK^d\to\R$ defined by \eqref{psi*} is translation
invariant as well.
From \eqref{Tn}, we get
\begin{eqnarray*}
&& \frac{1}{n!}\int\BE D^n_{K_1,\ldots,K_n}f_{\psi_1,W}(
\eta) \E D^n_{K_1,\ldots,K_n}f_{\psi_2,W}(\eta)
\Lambda^n\bigl(d(K_1,\ldots ,K_n)\bigr)
\\
&&\qquad = \frac{\gamma}{n!} \iiint\psi^*_1\bigl((K+x)\cap K_2
\cap\cdots\cap K_n\cap W\bigr) \\
&&\hspace*{68pt}{}\times\psi^*_2\bigl((K+x)\cap
K_2\cap\cdots\cap K_n\cap W\bigr)
 \\
 &&\hspace*{68pt}{}\times \Lambda^{n-1}\bigl(d(K_2,\ldots,K_n)\bigr)
\Q(dK) \,dx.
\end{eqnarray*}
For $n\in\N$,
we define $f_{W,n}\dvtx \cK^d\to\R$ by
\begin{eqnarray*}
f_{W,n}(K)&:=& \frac{1}{V_d(W)}\iint \psi ^*_1\bigl((K+x)
\cap K_2\cap\cdots\cap K_n\cap W\bigr)\\
 &&\hspace*{51pt}{}\times\psi^*_2\bigl((K+x)\cap K_2 \cap\cdots\cap
K_n\cap W\bigr) \\
&&\hspace*{51pt}{}\times \Lambda ^{n-1}\bigl(d(K_2,
\ldots,K_n)\bigr) \,dx,
\end{eqnarray*}
and
$f_{n}\dvtx \cK^d\to\R$ by
\[
f_{n}(K) := \int\psi^*_1(K\cap K_2\cap\cdots
\cap K_n) \psi ^*_2(K\cap K_2 \cap\cdots\cap
K_n) \Lambda^{n-1}\bigl(d(K_2,
\ldots,K_n)\bigr).
\]
Our aim is to prove that
\[
\sum_{n=1}^\infty\frac{\gamma}{n!}\int
f_{W,n}(K) \BQ(dK)\to\sum_{n=1}^\infty
\frac{\gamma}{n!}\int f_{n}(K) \BQ(dK)
\]
as $r(W)\to\infty$. Since we want to apply the dominated convergence
theorem, we
provide an upper bound for $\sum_{n=1}^\infty\frac{\gamma
}{n!}|f_{W,n}|$, which is independent of $W$. It follows from \eqref
{est3a} in Lemma~\ref{lem:boundintrinsic}, the translation invariance
of $V_i$ and $\Lambda$ and the monotonicity of the intrinsic volumes that
\begin{eqnarray*}
\bigl|f_{W,n}(K)\bigr| &\leq&\sum_{i,j=0}^d
\frac{\beta(\psi_1)\beta(\psi_2)}{V_d(W)} \iint V_i\bigl((K+x)\cap K_2\cap
\cdots\cap K_n\cap W\bigr)
\\
&&\hspace*{99pt}{}\times V_j\bigl((K+x)\cap K_2\cap\cdots\cap K_n
\cap W\bigr)\\
&&\hspace*{99pt}{}\times \Lambda ^{n-1}\bigl(d(K_2,\ldots,K_n)
\bigr) \,dx
\\
& \leq&\sum_{i,j=0}^d \frac{\beta(\psi_1)\beta(\psi_2)}{V_d(W)}
\int V_i(K\cap K_2\cap\cdots\cap K_n)
\Lambda^{n-1}\bigl(d(K_2,\ldots,K_n)\bigr) \\
&&\hspace*{92pt}{}\times\int V_j\bigl((K+x)\cap W\bigr) \,dx
\end{eqnarray*}
for $K\in\cK^d$ and $n\in\N$. Combining this estimate with Lemmas
\ref
{lem:translativeIntegral} and~\ref{lem:kinematic}, we get
%
\begin{equation}
\label{eq:majorant} \qquad\frac{1}{n!}\bigl|f_{W,n}(K)\bigr| \leq(d+1)
\beta_1 \beta(\psi_1)\beta (\psi_2) \Biggl(
\sum_{i=0}^d V_i(K)
\Biggr)^2 \sum_{r=0}^d
\frac{V_r(W)}{V_d(W)} \frac{\alpha^{n-1}}{n!}.
\end{equation}
By \eqref{vi2}, the right-hand side of \eqref{eq:majorant} is
integrable. Moreover, Lemma~\ref{lem:BoundInradius} shows that it is
uniformly bounded for $W\in\cK^d$ with $r(W)\geq1$, and the same holds
if we sum over all $n\in\N$.

Next, we bound $|f_{W,n}(K)-f_n(K)|$ for $K\in\cK^d$ and $n\in\N$. By
using the translation invariance of $\psi^*_1,\psi^*_2$ and $\Lambda$,
we have
\begin{eqnarray*}
&& f_{W,n}(K)-f_n(K)
\\
&&\qquad =\frac{1}{V_d(W)}\iint \bigl( \psi^*_1\bigl((K+x)\cap
K_2\cap\cdots \cap K_n\cap W\bigr)\\
&&\hspace*{86pt}{}\times \psi^*_2
\bigl((K+x)\cap K_2 \cap\cdots\cap K_n\cap W\bigr)
\\
&&\hspace*{86pt}{} -\I\{x\in W\} \psi^*_1\bigl((K+x)\cap K_2\cap\cdots
\cap K_n\bigr) \\
&&\hspace*{86pt}{}\times\psi ^*_2\bigl((K+x)\cap K_2
\cap\cdots\cap K_n\bigr) \bigr) \,dx \Lambda^{n-1}\bigl(d(K_2,\ldots,K_n)
\bigr).
\end{eqnarray*}
Note that the integrand is zero if $x\in W$ and $K+x\subset W$. The
same holds for the
case that $x\notin W$ and $(K+x) \cap W=\varnothing$. This means that the
integrand can be only nonzero if $(K+x)\cap\partial W\neq\varnothing$.
On the other hand, the integrand is always bounded by
\begin{eqnarray*}
&&\bigl|\psi^*_1\bigl((K+x)\cap K_2\cap\cdots\cap
K_n\cap W\bigr) \psi ^*_2\bigl((K+x)\cap K_2
\cap\cdots\cap K_n\cap W\bigr)\bigr|
\\
&&\quad{} +\bigl|\psi^*_1\bigl((K+x)\cap K_2\cap\cdots\cap
K_n\bigr) \psi^*_2\bigl((K+x)\cap K_2 \cap
\cdots\cap K_n\bigr)\bigr|
\\
&& \qquad\leq2 \beta(\psi_1) \beta(\psi_2) \Biggl(\sum
_{i=0}^d V_i\bigl((K+x)\cap
K_2\cap\cdots\cap K_n\bigr) \Biggr)^2,
\end{eqnarray*}
where we have used Lemma~\ref{lem:boundintrinsic} and the monotonicity
of the intrinsic volumes. Hence, we obtain that
\begin{eqnarray*}
&&\bigl |f_{W,n}(K)-f_n(K)\bigr|
\\
&&\qquad \leq\frac{2 \beta(\psi_1) \beta(\psi_2)}{V_d(W)} \iint\I\bigl\{ (K+x)\cap \partial W\neq\varnothing\bigr
\} \\
&&\hspace*{113pt}{}\times\Biggl(\sum_{i=0}^d V_i
\bigl((K+x)\cap K_2\cap \cdots\cap K_n\bigr)
\Biggr)^2
 \,dx \\
 &&\hspace*{113pt}{}\times\Lambda^{n-1}\bigl(d(K_2,\ldots,K_n)\bigr)
\\
&&\qquad \leq\frac{2 \beta(\psi_1) \beta(\psi_2)}{V_d(W)} \sum_{i=0}^d
V_i(K) \int\I\bigl\{(K+x)\cap\partial W\neq\varnothing\bigr\} \,dx
\\
&&\hspace*{30pt}{}\times \int\sum_{r=0}^d V_r(K\cap
K_2\cap\cdots\cap K_n) \Lambda ^{n-1}
\bigl(d(K_2,\ldots,K_n)\bigr),
\end{eqnarray*}
where we have used the fact that $V_i$ is increasing and the
translation invariance of $V_i$ and $\Lambda$ in the last step. Now
Lemmas \ref{lem:kinematic} and~\ref{bounds} yield that
\[
\bigl|f_{W,n}(K)-f_n(K)\bigr| \leq\frac{2 \beta_2 \beta(\psi_1) \beta(\psi_2)
\alpha^{n-1}}{V_d(W)} \Biggl(\sum
_{i=0}^d V_i(K)
\Biggr)^3 \sum_{r=0}^{d-1}
V_r(W).
\]
Together with Lemma~\ref{lem:BoundInradius} and $r(W)\geq1$, this
shows that, for $K\in\cK^d$ and $n\in\N$,
\[
\bigl\llvert f_{W,n}(K)-f_n(K)\bigr\rrvert \le\beta(
\psi_1,\psi_2)\alpha^{n-1} \Biggl(\sum
_{i=0}^dV_i(K) \Biggr)^3
\frac{1}{r(W)}
\]
with $\beta(\psi_1,\psi_2):=2^{d+2} \cdot4^{2^d}\cdot25^{2d} d!
M(\psi
_1)M(\psi_2) (\BE2^{N_{Q_1}} +1 )^2$. Therefore, for $K\in
\cK^d$,
%
\begin{equation}\qquad
\Biggl\llvert \sum_{n=1}^\infty
\frac{\gamma}{n!}f_{W,n}(K)-\sum_{n=1}^\infty
\frac{\gamma}{n!}f_n(K)\Biggr\rrvert \le\gamma\beta(
\psi_1,\psi_2) e^{\alpha} \Biggl(\sum
_{i=0}^dV_i(K) \Biggr)^3
\frac{1}{r(W)}.\label{Rate}
\end{equation}
Now an application of the dominated convergence theorem yields the
convergence result for $r(W)\to\infty$ stated in the theorem.

Under the stronger moment assumption \eqref{vi3}, \eqref
{eqn:AsymptoticCovariances} follows from \eqref{Rate} by carrying out
the integration
with respect to $K$ and collecting all the constants.
\end{pf*}

If the geometric functional is the volume, the asymptotic variance has
a significantly easier representation than in \eqref{eqn:sigma}, namely
%
\begin{equation}
\label{eqn:sigmadd} \sigma_{d,d} := \lim_{r(W)\to\infty}
\frac{\V V_d(Z\cap
W)}{V_d(W)} = (1-p)^2 \int \bigl(e^{\gamma C_d(x)}-1 \bigr) \,dx.
\end{equation}
This follows from an application of the dominated convergence theorem
to the exact variance formula \eqref{eqn:VarianceVolume}. The
inequalities $e^t-1\le te^t$, $t\ge0$ and $C_d(x)\leq v_d$ imply that
\[
\int \bigl(e^{\gamma C_d(x)}-1 \bigr) \,dx\le \gamma e^{\gamma v_d}\int\BE
V_d\bigl(Z_0\cap(Z_0+x)\bigr) \,dx =\gamma
e^{\gamma v_d}\BE V_d(Z_0)^2<\infty.
\]
Together with $C_W(x)/V_d(W)\leq1$, this means that $e^{\gamma
C_d(x)}-1$ is integrable and is an upper bound for $(C_W(x)/V_d(W))
(e^{\gamma C_d(x)}-1 )$. Now the observation that $C_W(x)/V_d(W)\to
1$ as $r(W)\to\infty$ for any $x\in\R^d$ [this follows from
$V_d(W)-C_W(x)\leq V_d(\partial^-_{\|x\|}W)$, \eqref{innerest},
Steiner's formula, and Lemma~\ref{lem:BoundInradius}] yields \eqref
{eqn:sigmadd}. In Section~\ref{seciso}, formulas as \eqref{eqn:sigmadd}
are derived for the other intrinsic volumes.

The following proposition shows that the rate of convergence stated in
Theorem~\ref{thm:variance} is optimal.

\begin{proposition}\label{prop:optimalrate}
Assume that \eqref{vi2} is satisfied and that the typical grain is
full-dimensional with positive probability. Then there is a constant
$c_{d,d}>0$ depending on $\Lambda$ such that
\[
\biggl|\sigma_{d,d} - \frac{\V V_d(Z\cap W)}{V_d(W)}\biggr | \geq \frac
{c_{d,d}}{r(W)}
\]
for $W\in\cK^d$ with $r(W)\geq1$.
\end{proposition}

\begin{pf}
Recall from the proof of Lemma~\ref{bounds} that
$\partial^-_rW=\{z\in W\dvtx\break  \dist(z,\partial W)\le r\}$ for $r\geq0$. For
$s\ge0$, we define $D_W(s):=\{z\in W\dvtx\break \dist(z,\partial W)=s\}$. Then
\[
W_{-s}:=\bigl\{z\in W\dvtx \dist(z,\partial W)\ge s\bigr\}=\bigl\{z
\in\R ^d\dvtx z+B^d_s\subset W\bigr\}
\]
is convex, the boundary of $W_{-s}$ is $D_W(s)$, and $s\mapsto W_{-s}$
is strictly decreasing
with respect to set inclusion, for $s\in[0, r(W)]$.

It follows from \eqref{eqn:VarianceVolume} and \eqref{eqn:sigmadd} that
\begin{eqnarray*}
&&\sigma_{d,d} - \frac{\V(V_d(Z\cap W))}{V_d(W)}\\
&&\qquad =(1-p)^2 \int
\frac
{V_d(W)-V_d(W\cap(W+x))}{V_d(W)} \bigl(e^{\gamma C_d(x)} - 1 \bigr) \,dx.
\end{eqnarray*}
Since the typical grain is full-dimensional with positive probability,
there are constants $\tau>0$ and $r_0\in(0,1/2)$ such that $e^{\gamma
C_d(x)} - 1 \geq\tau$
for all $x\in B^d_{r_0}$. This means that
%
\begin{eqnarray}
\label{eqn:lowerBoundI} &&\sigma_{d,d} - \frac{\V(V_d(Z\cap W))}{V_d(W)}
\nonumber
\\[-8pt]
\\[-8pt]
\nonumber
&&\qquad\geq(1-p)^2
\frac
{\tau
}{V_d(W)} \int_{B_{r_0}^d} \bigl(V_d(W)-V_d
\bigl(W\cap(W+x)\bigr)\bigr) \,dx.
\end{eqnarray}
Denoting by $B^d(x,r)$ the closed ball with center $x$ and radius $r$,
we have
\begin{eqnarray*}
&&\int_{B_{r_0}^d} \bigl(V_d(W)-V_d\bigl(W
\cap(W+x)\bigr)\bigr) \,dx\\
&&\qquad = \int_{B^d_{r_0}} \int
_W \bigl(\I\{y\in W\} - \I\{y\in W, y\in W+x\}\bigr) \,dy \,dx
\\
&&\qquad = \int_W \bigl(V_d\bigl(B^d(y,r_0)
\bigr) - V_d\bigl(W\cap B^d(y,r_0)\bigr)\bigr)
\,dy
\\
&&\qquad \geq\int_{\partial^-_{r_0/2}W} \bigl(V_d\bigl(B^d(y,r_0)
\bigr) - V_d\bigl(W\cap B^d(y,r_0)\bigr)\bigr)
\,dy.
\end{eqnarray*}
Using that $V_d(B^d(y,r_0)) - V_d(W\cap B^d(y,r_0))\geq\tilde{c}
r_0^d$ for $y\in\partial^-_{r_0/2}W$ with $\tilde{c}>0$, we obtain
%
\begin{equation}
\label{eqn:lowerBoundII} \int_{B_{r_0}^d} \bigl(V_d(W)-V_d
\bigl(W\cap(W+x)\bigr)\bigr) \,dx \geq\tilde{c} r_0^d
V_d\bigl(\partial^-_{r_0/2}W\bigr).
\end{equation}

It follows from Lemma~3.2.34 in \cite{Federer1969} that
\[
V_d\bigl(\partial^-_{r}W\bigr) = \int
_0^r {\cal H}^{d-1}
\bigl(D_W(s)\bigr) \,ds
\]
for $r\in[0,r(W)]$. The discussion at the beginning of this proof
implies that ${\cal H}^{d-1}(D_W(\cdot))$ is strictly decreasing on
$[0,r(W)]$. Together with $V_d(\partial^-_{r(W)}W)=V_d(W)$ we get for
$r(W)\ge r_0/2$ that
\begin{eqnarray*}
V_d(W)&=&\int_0^{r(W)}
\mathcal{H}^{d-1}\bigl(D_W(s)\bigr) \,ds \le\int
_0^{r(W)}\mathcal{H}^{d-1}
\biggl(D_W \biggl(\frac
{r_0}{2r(W)}s \biggr) \biggr) \,ds
\\
&=& \int_0^{r_0/2}\mathcal{H}^{d-1}
\bigl(D_W(t) \bigr) \frac
{2r(W)}{r_0} \,dt =\frac{2r(W)}{r_0}V_d
\bigl(\partial^-_{r_0/2}W\bigr).
\end{eqnarray*}
Combining this with \eqref{eqn:lowerBoundI} and \eqref
{eqn:lowerBoundII} completes the proof.
\end{pf}

\section{Positive definiteness}\label{secpos}

In this section, we consider the positive definiteness of the
asymptotic covariance matrix for geometric functionals $\psi_0,\ldots
,\psi_d$ on $\mathcal{R}^d$.
We assume that $\psi_k$, for $k\in\{0,\ldots,d\}$,
is positively homogeneous of degree $k$ and
%
\begin{equation}
\label{eq0} \bigl|\psi_k(K)\bigr|\ge\tilde{\beta}(\psi_k)
r(K)^k,
\end{equation}
for $K\in\mathcal{K}^d$,
with a constant $\tilde{\beta}(\psi_k)>0$, which only depends on
$\psi
_k$. These conditions are motivated by the intrinsic volumes
$V_0,\ldots
,V_d$, where they are obviously true. The additional assumptions on
$\psi_0,\ldots,\psi_d$ required in this section are used in an
essential way in the proof of Theorem~\ref{posdef} [see \eqref
{eqn:boundML} and \eqref{eqe} below], but are presumably not necessary
conditions for the positive definiteness of the asymptotic covariance
matrix. In particular, \eqref{eq0} is always satisfied if the absolute
value of $\psi_k$ on $\cK^d$ is bounded from below by a functional
$\tilde{\psi}_k\dvtx \cK^d\to\R$ which is positive and monotone [i.e.,
$\tilde{\psi}_k(K)\geq\tilde{\psi}_k(L)$ for $K,L\in\cK^d$ with
$K\supset L$]. This applies to the second example given at the
beginning of Section~\ref{sec:3}. If we assume that there is a constant
$h_0>0$ with $h\geq h_0$, then
\[
\bigl|\psi(K)\bigr|=\int_{\S^{d-1}} h(u) S_k(K,du)\geq
h_0 \,d V\bigl(B^d[d-k], K[k]\bigr)\ge dh_0
\kappa_dr(K)^k
\]
for $K\in\cK^d$, which ensures that \eqref{eq0} is satisfied.

By Theorem~\ref{thm:variance}, for $k,l\in\{0,\ldots,d\}$, the
asymptotic covariances $\sigma_{\psi_k,\psi_l}$
exist under the assumption \eqref{vi2}. The following theorem shows
that the asymptotic covariance matrix
is positive definite. In particular, the result applies to
the intrinsic volumes $V_0,\ldots,V_d$, which also means that their
asymptotic variances are strictly positive.

\begin{theorem}\label{posdef}
Let the preceding assumptions and \eqref{vi2} be satisfied. Moreover,
assume that the typical grain $Z_0$ has nonempty interior
with positive probability. Then
the covariance matrix $\Sigma:= (\sigma_{\psi_k,\psi_l}
)_{k,l=0,\ldots,d}$ is
positive definite.
\end{theorem}

\begin{pf}
For a vector $a=(a_0,\ldots,a_d)^\top\in\R^{d+1}$, we have
\begin{eqnarray*}
a^\top\Sigma a &=&\gamma\sum_{n=1}^\infty
\frac{1}{n!} \iint \Biggl(\sum_{k=0}^d
a_k \psi^*_k(K_1\cap K_2\cap
\cdots\cap K_n) \Biggr)^2\\
&&\hspace*{54pt}{}\times \Lambda^{n-1}
\bigl(d(K_2,\ldots,K_n)\bigr) \Q(dK_1).
\end{eqnarray*}
Since each summand is nonnegative, the matrix $\Sigma$ is positive definite
if we can prove that one summand is greater than zero for a
given $a\in\R^{d+1}\setminus\{0\}$. Specifically, under the given assumptions
we shall show that the summand obtained for $n=d+1$ is positive.
In order to show this, we shall prove that for $K_1,\ldots,K_{d+1}$ in
the support of $\Q$ and
having nonempty interiors, there is a set of translation
vectors $x_2,\ldots,x_{d+1}\in\R^d$ of positive $\lambda_d^d$ measure
(recall that $\lambda_d$ denotes
$d$-dimensional Lebesgue measure) for which
\[
\sum_{k=0}^d a_k
\psi^*_k\bigl(K_1\cap(K_2+x_2)
\cap\cdots\cap (K_{d+1}+x_{d+1})\bigr)\neq0.
\]
For the rest of the proof, we argue with a nonempty convex body $L\in
\cK^d$. Properties which
will be required of $L$ will be provided by an application of Lemma~\ref
{lem:Intersection} and
$L$ of the form $L=K_1\cap(K_2+x_2)\cap\cdots\cap
(K_{d+1}+x_{d+1})\in
\mathcal{K}^d$, for a set of translation vectors
$x_2,\ldots,x_{d+1}\in\R^d$ of positive $\lambda_d^d$ measure. This
will finally prove the preceding assertion,
and thus the theorem.

Let $N_1(L)$ be the number of grains of $\eta$ that intersect $L$,
but do not contain it, and let $N_2(L)$ be the number of grains of
$\eta$
that contain $L$. Then $N_1(L)$ and $N_2(L)$ are independent,
Poisson distributed random variables with parameters
\begin{eqnarray*}
s_1(L)&=&\Lambda\bigl(\bigl\{K\in\cK^d\dvtx K\cap L\neq
\varnothing\mbox{ and } L \not \subset K\bigr\}\bigr) \quad\mbox{and}\\
s_2(L)&=&\Lambda\bigl(\bigl\{K\in\cK^d\dvtx L\subset K\bigr
\}\bigr).
\end{eqnarray*}
If $N_2(L)\neq0$, then $L\subset Z$ and, therefore,
$\psi_k(Z\cap L)-\psi_k(L)=\psi_k(L)-\psi_k(L)=0$.
If $N_1(L)=N_2(L)=0$, then $Z\cap L=\varnothing$, and hence
$
\psi_k(Z\cap L)-\psi_k(L)=0-\psi_k(L)=-\psi_k(L)$.
This leads to
%
\begin{eqnarray}
\label{eqa} \psi^*_k(L)&=&\BE \bigl[\psi_k(Z\cap L)-
\psi_k(L) \bigr]
\nonumber
\\[-8pt]
\\[-8pt]
\nonumber
& =&-\exp\bigl(-s_1(L)-s_2(L)
\bigr)\psi_k(L)+R_k(L),
\end{eqnarray}
where
\[
R_k(L)=\BE{\mathbf{1}}\bigl\{N_1(L)
\ge1,N_2(L)=0\bigr\}\bigl(\psi_k(Z\cap L)-
\psi_k(L)\bigr).
\]
Next, we bound $R_k(L)$ from above. So assume that
$N_1(L)\neq0$. Let $K_1,\ldots,\break K_{N_1(L)}$
denote the grains of $\eta$ which hit $L$, but do not contain $L$. With
the definition of
$M_L(\psi_k)$ from \eqref{Mloc}, we obtain from \eqref{zitier} that
\[
\bigl|\psi_k(Z\cap L)-\psi_k(L)\bigr| \le\bigl|\psi_k(Z
\cap L)\bigr|+\bigl|\psi_k(L)\bigr|\le 2^{N_1(L)} M_L({
\psi_k}).
\]
In the following, let $R(K)$ stand for the radius of the circumscribed
ball of $K\in\cK^d$. For $A\in\cK^d$ with $A\subset L$, let $\hat
{a}\in
\R^d$
be the center of the circumball of $A$, hence $A-\hat{a}\subset2R(A)Q_1$.
Since $\psi_k$ is translation-invariant, homogeneous of degree $k$, and
locally bounded, we get
\[
\bigl|\psi_k(A)\bigr|=\bigl(2R(A)\bigr)^k\bigl|\psi_k
\bigl(\bigl(2R(A)\bigr)^{-1}(A-\hat{a})\bigr) \bigr|\le\bigl(2R(A)
\bigr)^k M(\psi_k),
\]
and hence
%
\begin{equation}
\label{eqn:boundML} M_L(\psi_k)\le\bigl(2R(L)
\bigr)^k M(\psi_k).
\end{equation}
Thus, in the present case, we have
\[
\bigl|\psi_k(Z\cap L)-\psi_k(L)\bigr|\le2^{N_1(L)}
\bigl(2R(L)\bigr)^k M({\psi_k}).
\]
Hence, the remainder term can be bounded from above by
\begin{eqnarray*}
\bigl|R_k(L)\bigr| & \leq&\BE \bigl[{\mathbf{1}}\bigl\{N_1(L)
\ge1,N_2(L)=0\bigr\} 2^{N_1(L)}\bigl(2R(L)\bigr)^k
M(\psi_k) \bigr]
\\
& = &\exp\bigl(-s_2(L)\bigr) \bigl(2R(L)\bigr)^k M(
\psi_k) \exp\bigl(s_1(L)\bigr) \bigl(1-\exp
\bigl(-2s_1(L)\bigr) \bigr)
\\
&\leq&\exp\bigl(-s_2(L)\bigr) \bigl(2R(L)\bigr)^k M(
\psi_k) \exp\bigl(s_1(L)\bigr)2s_1(L).
\end{eqnarray*}
Next, we derive an upper bound for $s_1(L)$. By definition and the reflection
invariance of Lebesgue measure, we have
\[
s_1(L)=\gamma\iint\I\bigl\{(L+x)\cap K\neq\varnothing,L+x\not\subset
K\bigr\} \,dx \Q(dK).
\]
To bound the inner integral from above, we can assume that $L\in\cK
^d_o$, by the translation
invariance of Lebesgue measure. If the integrand is nonzero, then $x\in
(K+R(L)B^d)\setminus K$ or $x\in\partial K_{R(L)}^-$.
Then inequality \eqref{innerest} implies that the inner integral is
bounded from above by $2V_d((K+R(L)B^d)\setminus K)$.
Hence, if $R(L)\le1$, Steiner's formula and our moment assumption
yield that
\[
s_1(L)\le c_6 R(L),
\]
where $c_6$ denotes a constant depending on $\Lambda$. Hence, if $R(L)$
is sufficiently small, then $s_1(L)\le1$, and thus
%
\begin{eqnarray}
\label{eqb} \bigl|R_k(L)\bigr|&\le&6 \cdot\bigl(2R(L)\bigr)^k
M(\psi_k) s_1(L)\exp\bigl(-s_2(L)\bigr)
\nonumber
\\[-8pt]
\\[-8pt]
\nonumber
&\le&6\cdot2^k\cdot c_6 M(\psi_k)
R(L)^{k+1}\exp\bigl(-s_2(L)\bigr).
\end{eqnarray}
We also have from \eqref{eqn:boundML} that
%
\begin{eqnarray}
\label{eqc}&& \bigl|\exp\bigl(-s_1(L)-s_2(L)\bigr)
\psi_k(L)\bigr|
\nonumber
\\[-8pt]
\\[-8pt]
\nonumber
&&\qquad\leq M_L(\psi_k) \exp
\bigl(-s_2(L)\bigr)\le \bigl(2R(L)\bigr)^k M(
\psi_k) \exp\bigl(-s_2(L)\bigr).
\end{eqnarray}
Hence, if $R(L)$ is sufficiently small, we deduce from \eqref{eqa},
\eqref{eqb} and \eqref{eqc} that
%
\begin{equation}
\label{eqd}\bigl |\psi_k^*(L)\bigr|\le\bar{\beta}(\psi_k)
R(L)^k \exp\bigl(-s_2(L)\bigr),
\end{equation}
where $\bar{\beta}(\psi_k)$ is a constant depending on $\Lambda$ and
$\psi_k$.
In addition,
%
\begin{equation}
\label{eqe} \bigl|\exp\bigl(-s_1(L)-s_2(L)\bigr)
\psi_k(L)\bigr|\ge\exp\bigl(-s_2(L)\bigr) \bigl(\tilde{\beta
}(\psi _k)/3\bigr) r(L)^k,
\end{equation}
if $s_1(L)\le1$, with $\tilde{\beta}(\psi_k)$ as in \eqref{eq0}.

Let $k_0$ be the smallest $k\in\{0,\ldots,d\}$ such that $a_k\neq0$. Then,
if $R(L)$ is sufficiently small, we get
\begin{eqnarray*}
&&\Biggl|\sum_{k=0}^d a_k
\psi^*_k(L) \Biggr|\\
&&\qquad= \Biggl|\sum_{k=k_0}^da_k
\psi ^*_k(L) \Biggr|
\\
&&\qquad= \Biggl|-a_{k_0} \exp\bigl(-s_1(L)-s_2(L)\bigr)
\psi_{k_0}(L)+a_{k_0}R_{k_0}(L) +\sum
_{k=k_0+1}^da_k \psi^*_k(L) \Biggr|
\\
&&\qquad\ge|a_{k_0}| \bigl|\exp\bigl(-s_1(L)-s_2(L)\bigr)
\psi_{k_0}(L)\bigr|-\bigl\llvert a_{k_0}R_{k_0}(L)\bigr
\rrvert -\sum_{k=k_0+1}^d|a_k|\bigl |
\psi^*_k(L)\bigr|
\\
&&\qquad\ge\exp\bigl(-s_2(L)\bigr) \bigl(|a_{k_0}| \bigl(\tilde{
\beta}(\psi_{k_0})/3\bigr) r(L)^{k_0}-\beta^*
R(L)^{k_0+1} \bigr),
\end{eqnarray*}
where we used \eqref{eqb} and \eqref{eqe}, for $k=k_0$, and \eqref{eqd}
for $k\ge k_0+1$. Here, we denote by $\beta^*$ a constant which
depends on
$a_{k_0},\ldots,a_d,\psi_{k_0},\ldots,\psi_d,\Lambda$.
The lower bound thus obtained is positive if $R(L)$ is sufficiently
small and $R(L)/r(L)\le c_0$,
for some constant $c_0$. The proof is completed by an application of
Lemma~\ref{lem:Intersection} below.
\end{pf}

The following lemma on the ratio of circumradius and inradius of
translates of convex bodies
is a key argument in the proof of Theorem~\ref{posdef}.

\begin{lemma}\label{lem:Intersection}
For all $K_1,\ldots,K_{d+1}\in\cK^d$ with nonempty interior there is a
constant $c_0>0$ such that
\begin{eqnarray*}
&&\lambda_d^{d} \Biggl( \Biggl\{(x_2,
\ldots,x_{d+1})\in\bigl(\R^d\bigr)^{d}\dvtx R(L)<
c_0 r(L) \mbox{ and } R(L) \leq r \\
&&\hspace*{143pt}\mbox{for } L=K_1\cap
\bigcap_{i=2}^{d+1}(K_i+x_i)
\Biggr\} \Biggr)>0
\end{eqnarray*}
for all $r>0$.
\end{lemma}

\begin{pf}
Let $u_1,\ldots,u_{d+1}\in\R^d$ be unit vectors whose
endpoints are the vertices of a regular simplex.
For $i=1,\ldots,d+1$ let $x_i$ be a point in the boundary of $K_i$ which
has $u_i$ as an exterior normal vector. The
support cone $S(K_i,x_i)$ of $K_i$ at~$x_i$ (cf.~\cite{Sch93}, page~81)
then satisfies
\[
K_i-x_i\subset S(K_i,x_i):=
\mathrm{cl} \biggl(\bigcup_{
t>0}t(K_i-x_i)
\biggr) \subset H^-(K_i,u_i)-x_i,
\]
where $H^-(K_i,u_i)$ is the supporting half-space of $K_i$ with
exterior unit normal $u_i$
and $\mathrm{cl}$ denotes the closure.
By \cite{SW08}, Theorem~12.2.2, it follows that $t(K_i-x_i)\to S(K_i,x_i)$
in the topology of closed
convergence as $t\to\infty$. Moreover, since $K_1,\ldots,K_{d+1}$
have nonempty interiors, there are vectors $z_1,\ldots,z_{d+1}\in\R^d$
such that the origin is an interior point of
\[
S_0:=\bigcap_{i=1}^{d+1}
\bigl(S(K_i,x_i)+z_i \bigr)\subset\bigcap
_{i=1}^{d+1} \bigl(H^-(K_i,u_i)-x_i+z_i
\bigr)
\]
and the circumradius of the intersection on the right-hand side is less
than $1$ (say).
Then \cite{Sch93}, Theorem~1.8.10
and \cite{SW08}, Theorem~12.3.3, imply that
\[
S_0=\lim_{t\to\infty} \Biggl(t \bigcap
_{i=1}^{d+1}\bigl(K_i+x_i(t)
\bigr) \Biggr),
\]
where $x_i(t):=-x_i+t^{-1}z_i$ and the convergence is with respect
to the Hausdorff distance. Since the inradius and the circumradius of the
intersection of translates of convex bodies are continuous with respect
to the translations as long as the intersection has nonempty interior,
there is some $t_0>1$ such that the ratio between inradius and
circumradius of
\[
t \bigcap_{i=1}^{d+1} \bigl(K_i+x_i(t)
\bigr)
\]
is close to the corresponding ratio of $S_0$, for $t\ge t_0$
and, therefore, also
\[
1\le\frac{R (\bigcap_{i=1}^{d+1}
 (K_i+x_i(t) ) )}{r (\bigcap_{i=1}^{d+1}
 (K_i+x_i(t) ) )}<\tilde{c}_0,
\]
with a constant $\tilde{c}_0>1$ which depends only on $K_1,\ldots,K_{d+1}$.
Moreover, for $t\ge t_0>1$ we have
\[
R \Biggl(t \bigcap_{i=1}^{d+1}
\bigl(K_i+x_i(t)\bigr) \Biggr)\le R \Biggl(\bigcap
_{i=1}^{d+1} \bigl(H^-(K_i,u_i)-x_i+z_i
\bigr) \Biggr)< 1
\]
and thus
\[
R \Biggl( \bigcap_{i=1}^{d+1}
\bigl(K_i+x_i(t)\bigr) \Biggr)<\frac{1}{t}.
\]
Therefore, if $r<1/(2t_0)$ the proof of the lemma is completed by
remarking that
the intersections are continuous with respect to
translations as long as the intersection has nonempty
interior and by using the translation invariance of Lebesgue measure.
Clearly, this proves the lemma for all $r>0$.
\end{pf}

\section{Some integral formulas for intrinsic volumes}\label{sintegra}

We shall see in the next section that
in the particularly important case of intrinsic volumes and
under the assumption of isotropy the asymptotic covariances
of Theorem~\ref{thm:variance} can be expressed in terms of the numbers
%
\begin{eqnarray}
\label{rhoijlimit} &&\rho_{i,j}:=\gamma \sum^\infty_{n=1}
\frac{1}{n!}\iint V_{i}(K_1\cap\cdots\cap
K_n) V_{j}(K_1\cap\cdots\cap K_n)
\nonumber
\\[-8pt]
\\[-8pt]
\nonumber
&&\hspace*{87pt}{}\times
\Lambda^{n-1}\bigl(d(K_2,\ldots,K_n)\bigr)
\mathbb{Q}(dK_1).
\end{eqnarray}
In this section, we study these numbers \emph{without isotropy assumption}
on $Z$. The results are of independent interest.

For $W\in\cK^d$ and $i,j\in\{0,\ldots,d\}$, we define
%
\begin{eqnarray}
\label{ijW} &&\rho_{i,j}(W):= \sum^\infty_{n=1}
\frac{1}{n!}\int V_{i}(K_1\cap\cdots\cap
K_n\cap W) V_{j}(K_1\cap\cdots\cap
K_n\cap W)
\nonumber
\\[-8pt]
\\[-8pt]
\nonumber
&&\hspace*{90pt}{}\times\Lambda^n\bigl(d(K_1,
\ldots,K_n)\bigr),
\end{eqnarray}
which is a finite window version of $\rho_{i,j}$. The numbers $\rho
_{i,j}(W)$ are further studied
in \cite{HLSarxiv}, Appendix B.
The relationship between \eqref{rhoijlimit} and \eqref{ijW} is given in
the next corollary.

\begin{corol}\label{cor:alimit}
Let $i,j\in\{0,\ldots,d\}$. If \eqref{vi2} is satisfied, then $\rho
_{i,j}<\infty$ and
%
\begin{equation}
\label{rhoij} \lim_{r(W)\to\infty}\frac{\rho_{i,j}(W)}{V_d(W)}=
\rho_{i,j}.
\end{equation}
If \eqref{vi3} is satisfied, then there is a constant $c_{i,j}$ such that
\[
\biggl\llvert \rho_{i,j} - \frac{\rho_{i,j}(W)}{V_d(W)}\biggr\rrvert \le
\frac
{c_{i,j}}{r(W)}
\]
for $W\in\cK^d$ with $r(W)\geq1$.
\end{corol}

\begin{pf}
This can be proved in a similar way as Theorem~\ref{thm:variance}.
\end{pf}

The previous corollary describes $\rho_{i,j}$ as the limit of
$V_d(W)^{-1}\rho_{i,j}(W)$ for
observation windows with $r(W)\to\infty$. It is, however, more
convenient to work with the series representation \eqref{rhoijlimit}.
We shall see that this series can be expressed in
terms of a finite family of (curvature) measures $H_{i,j}$ to be
introduced below.

For $j\in\{0,\ldots,d\}$ and $K\in\mathcal{K}^d$, we let
$\Phi_j(K;\cdot)$ denote the $j$th curvature measure
of $K$ (see \cite{SW08}, Section~14.2).
In particular, $\Phi_d(K;\cdot)$ is the restriction of Lebesgue
measure to $K$ while $\Phi_{d-1}(K;\cdot)$
is half the $(d-1)$-dimensional Hausdorff measure restricted to
the boundary of $K$ (if the affine hull of $K$ has full dimension).
Furthermore, $\Phi_j(K;\R^d)=V_j(K)$ for all $j\in\{0,\ldots,d\}$.
For $j\in\{0,\ldots,d-1\}$, $n\in\N$, and
$K_1,\ldots,K_n\in\mathcal{K}^d$ we define
%
\begin{equation}
\label{78} \Phi_j(K_1,\ldots,K_n;
\cdot):= \Phi_j(K_1\cap\cdots\cap K_n;
\partial K_1\cap\cdots\cap\partial K_n\cap\cdot).
\end{equation}
Since $\Phi_j(K_1;\cdot)$, $j\in\{0,\ldots,d-1\}$, is concentrated on
the boundary
$\partial K_1$ of $K_1$, this definition is consistent with the
case $n=1$.
For $i\in\{1,\ldots,d-1\}$ and $k\in\{1,\ldots,d-i\}$,
we define a measure $H^{k}_{i,d}$ on $\R^d$ by
%
\begin{eqnarray}
\label{3.1neu}&& H^{k}_{i,d}:=\gamma\iiiint\I\{y-z\in\cdot\}
\I\{z\in K_1\cap \cdots\cap K_k\} \Phi_i(K_1,
\ldots,K_k;dy) \,dz
\nonumber
\\[-8pt]
\\[-8pt]
\nonumber
&&\hspace*{75pt}{}\times
\Lambda^{k-1}\bigl(d(K_1,\ldots,K_{k-1})\bigr)
\BQ(dK_k),
\end{eqnarray}
with the appropriate interpretation of the case $k=1$.

For $i,j\in\{1,\ldots,d-1\}$, $k\in\{1,\ldots,d-i\}$,
$l\in\{1,\ldots,d-j\}$, and $m\in\{0,\ldots,k\wedge l\}$
we define a measure $H^{k,l,m}_{i,j}$ on $\R^d$ by
%
\begin{eqnarray}
\label{3.101} H^{k,l,m}_{i,j}\dvtx &=&\gamma\iiiint\I\{y-z\in
\cdot\}\nonumber\\
&&\hspace*{41pt}{}\times \I\bigl\{y\in K^\circ_{k+1}\cap\cdots\cap
K^\circ_{k+l-m},z\in K^\circ _1\cap
\cdots\cap K^\circ_{k-m}\bigr\}
\nonumber
\\[-8pt]
\\[-8pt]
\nonumber
&&\hspace*{41pt}{}\times \Phi_i(K_1,\ldots,K_{k};dy)
\Phi_j(K_{k+1-m},\ldots,K_{k+l-m};dz)
\\
&&\hspace*{41pt}{}\times \Lambda^{k+l-m-1}\bigl(d(K_1,\ldots,K_{k+l-m-1})\bigr)
\BQ(dK_{k+l-m}),\nonumber
\end{eqnarray}
where $K^\circ$ denotes the interior of $K\in\mathcal{K}^d$ and with
the appropriate interpretation of the cases $m=k$ or $m=l$. Let
\begin{eqnarray*}
H_{i,d}&:=&\sum^{d-i}_{k=1}
\frac{1}{k!}H^k_{i,d}, i\in\{1,\ldots ,d-1\},
\\
H_{i,j}&:=&\sum^{d-i}_{k=1}\sum
^{d-j}_{l=1}\sum
^{k\wedge l}_{m=0} \frac{1}{m!(k-m)!(l-m)!}H^{k,l,m}_{i,j},\qquad
i,j\in\{1,\ldots,d-1\},
\end{eqnarray*}
and, for $j\in\{0,\ldots,d-1\}$,
%
\begin{equation}
\label{3.9}\qquad h_{0,j}:=\sum^{d-j}_{l=1}
\frac{\gamma}{l!}\iint\Phi_j\bigl(K_1,\ldots
,K_{l};\R^d\bigr) \Lambda^{l-1}
\bigl(d(K_1,\ldots,K_{l-1})\bigr) \BQ(dK_l).
\end{equation}
Moreover, we define
$
H_{d,d}(dx):= (1-e^{-\gamma C_d(x)} ) \,dx$,
$H_{0,j}:=H_{j,0}:=h_{0,j}\delta_0$ for $j\in\{0,\ldots,d-1\}$,
and $H_{0,d}:=H_{d,0}:= (1-e^{-\gamma v_d} )\delta_0$,
where $\delta_0$ is the Dirac measure concentrated at the origin and
$C_d(x)$ is
the mean covariogram of the typcial grain as defined in
\eqref{meancov}.

Subsequently, we assume that
%
\begin{equation}
\label{condfull} \Q\bigl(\bigl\{K\in\mathcal{K}^d\dvtx
V_d(K)>0\bigr\}\bigr)=1,
\end{equation}
that is, the typical grain almost surely has
nonempty interior.

\begin{theorem}\label{tdeltaij} Assume that \eqref{vi2} and \eqref
{condfull} are satisfied. Then the measures $H_{i,j}$ are all finite.
Moreover, the limits \eqref{rhoij} are given by
%
\begin{equation}
\label{formularhoij} \rho_{i,j}=\int e^{\gamma C_d(x)} H_{i,j}(dx),\qquad
i,j\in\{0,\ldots,d\}.
\end{equation}
For $i=d$ or $j=d$, the result remains true without assumption \eqref
{condfull}.
\end{theorem}

In particular, we thus have
%
\begin{equation}
\label{3.11} \rho_{d,d}=\int \bigl(e^{\gamma C_d(x)}-1 \bigr) \,dx,
\end{equation}
$\rho_{0,d}=e^{\gamma v_d}-1$, and $\rho_{0,j}=e^{\gamma v_d}h_{0,j}$
for $j\in\{0,\ldots,d-1\}$.

The proof of Theorem~\ref{tdeltaij} is based on the following
geometric result. Here, we use the abbreviation $[n]=\{1,\ldots,n\}$.

\begin{lemma}\label{ldecomp}
Let $K_1, K_2',\ldots,K_n'\in\cK^d$, $n\in\N$, have nonempty interiors,
and let $i\in\{0,\ldots,d-1\}$. Then
\[
\Curvi(K_1\cap\cdots\cap K_n;\cdot) =\mathop{\sum
_{\varnothing\neq I\subset[n]}}_{ |I|\le d-i} \Curvi \biggl(\bigcap
_{r\in I}K_r;\cdot\cap \bigcap
_{r\in I}\partial K_r\cap\bigcap
_{s\notin I}K_s^\circ \biggr),
\]
for almost all translates $K_i$ of $K_i'$ for $i=2,\ldots,n$.
\end{lemma}

Hence, if \eqref{condfull} is satisfied and $K_1\in\mathcal{K}^d$ has
nonempty interior, then this lemma
can be applied for $\Lambda^{n-1}$-a.e. $(K_2,\ldots,K_n)\in(\cK^d)^{n-1}$.

Before we prove Lemma~\ref{ldecomp}, we provide two auxiliary results.

\begin{lemma}\label{zeroexception}
Let $K_1,\ldots,K_m\in\cK^d$, $m\ge2$, have nonempty interiors.
Then, for $\mathcal{H}^{d(m-1)}$-almost all $(t_2,\ldots,t_m)\in\R
^{d(m-1)}$,
if $K_1\cap( K_2+t_2)\cap\cdots\cap(K_m+t_m)\neq\varnothing$, then
$ (K_1)^\circ\cap(K_2+t_2)^\circ\cap\cdots\cap(K_m+t_m)^\circ
\neq
\varnothing$.
\end{lemma}

\begin{pf} The assertion is proved by induction over $m\ge2$. For
$m=2$, the assertion holds, since any $t_2\in\R^d$ such that
$K_1\cap(K_2+t_2)\neq\varnothing$ and $K_1^\circ\cap
(K_2^\circ+t_2)=\varnothing$ is contained in the boundary of
$K_1+(-K_2)$, which has $d$-dimensional Hausdorff measure zero. The
induction step follows from the case $m=2$ and Fubini's theorem. For
further details, see \cite{HLSarxiv}.
\end{pf}

For the following lemma, we use basic notions from geometric measure
theory (see, e.g., \cite{Federer1969}).

\begin{lemma}\label{rectify}
Let $K_1,\ldots,K_m\in\cK^d$, $m\in\N$. If $m\le d$,
then for $\mathcal{H}^{d(m-1)}$-almost all translates
$(t_2,\ldots,t_m)\in\R^{d(m-1)}$, the intersection
$\partial K_1\cap(\partial K_2+t_2)\cap\cdots\cap(\partial K_m+t_m)$
has finite
$(d-m)$-dimensional Hausdorff measure. For $m>d$, the intersection is the
empty set for almost all translation vectors.
\end{lemma}

\begin{pf}
Since for $m=1$ there is nothing to show, we assume that $m\in\{
2,\ldots
,d\}$.
Let $W:=\partial K_1\times\cdots\times\partial K_m\subset\R^{dm}$,
let $Z\subset\R^{d(m-1)}$ be the compact image set of the Lipschitz map
$T\dvtx W\to Z\subset\R^{d(m-1)}$, $(x_1,\ldots,x_m)\mapsto
(x_1-x_2,\ldots
,x_1-x_m)$.
Then the assumptions of the coarea theorem (\cite{Federer1969},
Theorem~3.2.22
(2))
are satisfied. Thus, for $\mathcal{H}^{d(m-1)}$-almost all
$(t_2,\ldots,t_m)\in Z$, the set $T^{-1}\{(t_2,\ldots,t_m)\}$
has finite $\mathcal{H}^{d-m}$ measure.
Identify $\R^{dm}$ with $(\R^d)^m$ and denote by $\pi_1\dvtx (\R^d)^m\to
\R^d$
the projection to the first component, which is a Lipschitz map. Then
$
\partial K_1\cap(\partial K_2+t_2)\cap\cdots\cap(\partial K_m+t_m)
=\pi_1 (T^{-1}\{(t_2,\ldots,t_m)\} )
$
has finite $(d-m)$-dimensional Hausdorff measure for $\mathcal
{H}^{d(m-1)}$-almost all
$(t_2,\ldots,t_m)\in Z$. [If $(t_2,\ldots,t_m)\notin Z$, the
intersection is the empty set.]

The assertion for $m>d$ easily follows from the one for $m=d$.
\end{pf}

\begin{pf*}{Proof of Lemma~\ref{ldecomp}}
There is nothing to show for $n=1$ so that we assume that $n\ge2$.
By Lemmas \ref{zeroexception} and \ref{rectify},
we can assume that $K_1,\ldots,K_n$ have a common interior point and
for $\varnothing\neq I\subset[n]$ each intersection
$\bigcap_{r\in I}\partial K_r$ has finite $(d-|I|)$-dimensional
Hausdorff measure if $|I|\le d$, and otherwise is the empty set.

Since $\Curvi(K_1\cap\cdots\cap K_n,\cdot)$ is concentrated on the boundary
of $K_1\cap\cdots\cap K_n$, the measure property yields that
\[
\Curvi(K_1\cap\cdots\cap K_n;\cdot)=\sum
_{\varnothing\neq I\subset[n]} \Curvi \biggl(K_1\cap\cdots\cap
K_n;\cdot\cap \bigcap_{r\in I}\partial
K_r\cap\bigcap_{s\notin I}K_s^\circ
\biggr).
\]
The intersection $U:=\bigcap_{s\notin I}K_s^\circ$ is open,
$U':=\bigcap_{r\in I}\partial K_r\cap\bigcap_{s\notin I}K_s^\circ
\subset U$,
and $K_1\cap\cdots\cap K_n\cap U=\bigcap_{r\in I}K_r\cap U$. Hence,
since $\Curvi$ is locally determined (see~\cite{Sch93}, page~215), it
follows that
\[
\Curvi(K_1\cap\cdots\cap K_n;\cdot)=\sum
_{\varnothing\neq I\subset[n]} \Curvi \biggl( \bigcap_{r\in I}
K_r;\cdot\cap\bigcap_{r\in
I}\partial
K_r\cap \bigcap_{s\notin I}K_s^\circ
\biggr).
\]
Since $\bigcap_{r\in I}\partial K_r$ has finite $(d-|I|)$-dimensional
Hausdorff measure for $|I|\in\{1,\ldots,d\}$, and is the empty set
for $|I|>d$, we conclude that if $d\ge|I|>d-i$, then $\bigcap_{r\in
I}\partial K_r$
has $i$-dimensional Hausdorff measure zero.
A special case of \cite{ColesantiHug}, Theorem~5.5, then yields that
\[
\Curvi \biggl( \bigcap_{r\in I} K_r;\cdot
\cap\bigcap_{r\in
I}\partial K_r\cap \bigcap
_{s\notin I}K_s^\circ \biggr)=0,
\]
which completes the proof.
\end{pf*}

\begin{pf*}{Proof of Theorem~\ref{tdeltaij}}
We start with showing that the measures $H_{i,j}$ are finite. Let
$i,j\in\{1,\ldots,d-1\}$, $k\in\{1,\ldots,d-i\}$,
$l\in\{1,\ldots,d-j\}$, and $m\in\{0,\ldots,k\wedge l\}$. Then
\begin{eqnarray*}
H_{i,j}^{k,l,m}\bigl(\R^d\bigr) & \leq&\gamma\iiiint
\I\{K_1\cap\cdots\cap K_{k+l-m}\neq\varnothing\}
\Phi_i(K_1,\ldots,K_k;dy)
\\
&&\hspace*{39pt}{}\times \Phi_j(K_{k+1-m},\ldots,K_{k+l-m};dz)\\
&&\hspace*{39pt}{}\times \Lambda
^{k+l-m-1}\bigl(d(K_1,\ldots ,K_{k+l-m-1})\bigr)
\Q(dK_{k+l-m})
\\
& \leq&\gamma\iint V_0(K_1\cap\cdots\cap
K_{k+l-m}) V_i(K_1) V_j(K_{k+l-m})
\\
&&\hspace*{27pt}{}\times \Lambda^{k+l-m-1}\bigl(d(K_1,\ldots,K_{k+l-m-1})\bigr)
\Q(dK_{k+l-m}).
\end{eqnarray*}
For $k+l-m=1$ the right-hand side is finite because of assumption
\eqref
{vi2}. Otherwise, we obtain by Lemmas \ref{lem:kinematic} and
\ref
{lem:translativeIntegral} that
\begin{eqnarray*}
&& H_{i,j}^{k,l,m}\bigl(\R^d\bigr)
\\
&&\qquad \leq\gamma^2 \alpha^{k+l-m-2} \iint\sum
_{r=0}^d V_r\bigl((K_1+x)
\cap K_{k+l-m}\bigr) V_i(K_1)\\
&&\hspace*{124pt}{}\times V_j(K_{k+l-m}) \,dx \Q^2\bigl(d(K_1,K_{k+l-m})
\bigr)
\\
&&\qquad \leq(d+1) \gamma^2 \alpha^{k+l-m-2}\beta_1 \int
\sum_{r=0}^d V_r(K_1)
\sum_{r=0}^d V_r(K_{k+l-m})
 V_i(K_1)V_j(K_{k+l-m})\\
 &&\hspace*{210pt}{}\times
\Q^2\bigl(d(K_1,K_{k+l-m})\bigr).
\end{eqnarray*}
Now it follows from \eqref{vi2} that the right-hand side is finite.
Similar (but easier) arguments show that the other measures are also finite.

Note that $\rho_{i,j}=\rho_{j,i}$ for $i,j\in\{0,\ldots,d\}$. To prove
that the series \eqref{rhoijlimit} is given by
\eqref{formularhoij}, we distinguish different cases
and start with $i=j=d$. Then we have
\begin{eqnarray*}
\rho_{d,d}&=&\gamma\sum^\infty_{n=1}
\frac{1}{n!}\iint V_{d}(K_1\cap\cdots\cap
K_n)^2 \Lambda^{n-1}\bigl(d(K_2,
\ldots,K_n)\bigr) \mathbb{Q}(dK_1)
\\
&=& \sum^\infty_{n=1}\frac{\gamma^{n}}{n!}
\idotsint \I\bigl\{y\in K_1\cap(K_2+x_2)\cap
\cdots\cap(K_n+x_n)\bigr\}
\\
&& \hspace*{68pt}{}\times\I\bigl\{z\in K_1\cap(K_2+x_2)\cap\cdots\\
&&\hspace*{144pt}{}\cap(K_n+x_n)\bigr\} \,dy \,dz \,dx_2\cdots
\,dx_n\\
&&\hspace*{68pt}{}\times \mathbb{Q}^n\bigl(d(K_1,
\ldots,K_n)\bigr)
\\
&=& \sum^\infty_{n=1}\frac{\gamma^{n}}{n!}
\iiint V_d\bigl((K_2-y)\cap (K_2-z)\bigr)
\cdots V_d\bigl((K_n-y)\cap(K_n-z)\bigr)
\\
&&\hspace*{56pt}{}\times \I\{y\in K_1\}\I\{z\in K_1\} \,dy \,dz
\mathbb{Q}^n\bigl(d(K_1,\ldots ,K_n)\bigr)
\\
&=& \sum^\infty_{n=1}\frac{\gamma^{n}}{n!}
\iiint\bigl(\mathbb {E}V_d\bigl(Z_0\cap
(Z_0+y-z)\bigr)\bigr)^{n-1} \I\{y,z\in K_1\} \,dy
\,dz \mathbb{Q}(dK_1)
\\
&=& \sum^\infty_{n=1}\frac{\gamma^{n}}{n!}
\iiint\bigl(\mathbb {E}V_d\bigl(Z_0\cap
(Z_0+y)\bigr)\bigr)^{n-1} \I\{y+z\in K_1\}\\
&&\hspace*{54pt}{}\times \I\{z
\in K_1\} \,dy \,dz \mathbb{Q}(dK_1)
\\
&=& \sum^\infty_{n=1}\frac{\gamma^{n}}{n!}\int
C_d(y)^n \,dy = \int \bigl(e^{\gamma C_d(y)}-1 \bigr) \,dy.
\end{eqnarray*}

For $i=0$ and $j=d$, we get by an even simpler calculation
\[
\rho_{0,d} = \sum_{n=1}^\infty
\frac{\gamma^n}{n!} \bigl(\BE V_d(Z_0)\bigr)^n=
e^{\gamma v_d}-1.
\]
%
This and the preceding calculation do not depend on assumption \eqref
{condfull}.

Next, we turn to $i=0$ and $j\in\{0,\ldots,d-1\}$. Then, using
$V_j(L)=\Phi_j(L;\R^d)$,
for $L\in\mathcal{K}^d$ and Lemma~\ref{ldecomp}, we get
\begin{eqnarray*}
&&\rho_{0,j} =\gamma\sum^\infty_{n=1}
\frac{1}{n!} \sum_{l=1}^{d-j}\mathop{
\sum_{J\subset[n]}}_{ |J|=l}\iiint \I\biggl\{z\in
\bigcap_{s\notin J}K_s^\circ\biggr\}
\Phi_j(K_J;dz) \Lambda^{n-1}
\bigl(d(K_2,\ldots,K_n)\bigr) \\
&&\hspace*{134pt}{}\times\mathbb{Q}(dK_1),
\end{eqnarray*}
%
where $\Phi_j(K_J;\cdot)=\Phi_j(K_{j_1},\ldots,K_{j_l};\cdot)$ for
$J=\{
j_1,\ldots,j_l\}$ [see \eqref{78}].
At this stage and also later, we use the covariance property
%
\begin{eqnarray}
\label{3.67a} &&\int h(y) \Phi_i(K_1,
\ldots,K_l;dy)
\nonumber
\\[-8pt]
\\[-8pt]
\nonumber
&&\qquad= \int h(y+x) \Phi_i(K_1-x,
\ldots,K_l-x;dy),\qquad x\in\R^d,
\end{eqnarray}
which holds for all measurable $h\dvtx \R^d\rightarrow[0,\infty]$.
This follows from the definition \eqref{78} and \cite{SW08},
Theorem~14.2.2.
Using \eqref{3.67a} and then the invariance of $\Lambda$ under translations,
it is easy to check that, for instance,
\begin{eqnarray*}
&&\iiint \I\bigl\{z\in K_{l+1}^\circ\cap\cdots\cap
K_n^\circ\bigr\} \Phi_j(K_{\{1,\ldots,l\}};dz)
\Lambda^{n-1}\bigl(d(K_1,\ldots,K_{n-1})\bigr)
\mathbb{Q}(dK_n)
\\
&&\qquad =\iiint \I\bigl\{z\in K_{l+1}^\circ\cap\cdots\cap
K_n^\circ\bigr\} \Phi_j(K_{\{1,\ldots,l\}};dz)\\
&&\hspace*{21pt}\qquad\quad{}\times\Lambda^{n-1}\bigl(d(K_2,\ldots,K_{n})\bigr)
\mathbb {Q}(dK_1).
\end{eqnarray*}
From such symmetry relations, we deduce that
\begin{eqnarray*}
\rho_{0,j} &=&\gamma\sum_{l=1}^{d-j}
\sum^\infty_{n=l}\frac{1}{n!}\pmatrix{n
\cr
l}\iiint \I\bigl\{z\in K_{l+1}^\circ\cap\cdots\cap
K_n^\circ\bigr\}
\\
&&\hspace*{101pt}{}\times \Phi_j(K_1,\ldots,K_l;dz)
\\
&&\hspace*{101pt}{}\times \Lambda^{n-1}\bigl(d(K_2,\ldots,K_n)\bigr)
\mathbb{Q}(dK_1)
\\
&=&\gamma\sum_{l=1}^{d-j} \sum
^\infty_{n=l}\frac{1}{n!}\pmatrix{n
\cr
l}
\gamma^{n-l}\iiint V_d(K_{l+1})\cdots
V_d(K_n) \Phi_j\bigl(K_1,
\ldots,K_l;\R^d\bigr)
\\
&&\hspace*{124pt}{}\times \mathbb{Q}^{n-l}\bigl(d(K_{l+1},\ldots,K_n)
\bigr) \\
&&\hspace*{124pt}{}\times\Lambda^{l-1}\bigl(d(K_2,\ldots,K_l)
\bigr) \mathbb{Q}(dK_1)
\\
&=&\gamma\sum_{l=1}^{d-j}\frac{1}{l!}
\sum^\infty_{n=l}\frac{(\gamma v_d)^{n-l}}{(n-l)!}\iint\Phi
_j\bigl(K_1,\ldots,K_l;\R^d
\bigr) \\
&&\hspace*{112pt}{}\times\Lambda^{l-1}\bigl(d(K_2,\ldots,K_l)
\bigr) \mathbb{Q}(dK_1) =e^{\gamma v_d} h_{0,j}.
\end{eqnarray*}

Next, we address the case $i\in\{1,\ldots,d-1\}$ and $j=d$. Using again
Lemma~\ref{ldecomp} and a symmetry argument (as above), we
obtain
\begin{eqnarray*}
&&\rho_{i,d} = \gamma\sum^\infty_{n=1}
\frac{1}{n!} \sum_{k=1}^{d-i}\pmatrix{n
\cr
k}\iiiint\I\bigl\{y\in K_{k+1}^\circ\cap \cdots \cap
K_n^\circ\bigr\}\\
&&\hspace*{138pt}{}\times \I\{z\in K_1\cap\cdots\cap
K_n\}
\, dz\\
&&\hspace*{138pt}{}\times \Phi_i(K_1,\ldots,K_k;dy)\\
&&\hspace*{138pt}{}\times\Lambda^{n-1}\bigl(d(K_2,\ldots,K_n)\bigr)
\mathbb{Q}(dK_1).
\end{eqnarray*}
%
Then we interchange the order of summation to get
\begin{eqnarray*}
\rho_{i,d}&=& \gamma\sum_{k=1}^{d-i}
\sum^\infty_{n=k}\frac{\gamma^{n-k}}{k!(n-k)!}\idotsint
\I\bigl\{x_{k+1}\in\bigl(K_{k+1}^\circ-y\bigr)\cap(
K_{k+1}-z)\bigr\} \cdots
\\
&&\hspace*{123pt}{} \times\I\bigl\{x_{n}\in\bigl(K_{n}^\circ-y
\bigr)\cap( K_{n}-z)\bigr\}\\
&&\hspace*{123pt}{} \times \mathbb{Q}(dK_{k+1})\cdots
\mathbb{Q}(dK_{n}) \,dx_{k+1}\cdots dx_n
\\
&&\hspace*{123pt}{}\times \I\{z\in K_1\cap\cdots\cap K_k\}
\Phi_i(K_1,\ldots,K_k;dy) \,dz\\
&&\hspace*{123pt}{}\times\Lambda^{k-1} \bigl(d(K_2,\ldots,K_k)\bigr)
\mathbb{Q}(dK_1)
\\
&=& \gamma\sum_{k=1}^{d-i} \sum
^\infty_{n=k}\frac{\gamma^{n-k}}{k!(n-k)!} \iiiint \bigl(
\mathbb{E} V_d\bigl(Z_0\cap(Z_0+y-z)\bigr)
\bigr)^{n-k}\\
&&\hspace*{119pt}{}\times\I\{ z\in K_1\cap\cdots\cap K_k\}
\\
&&\hspace*{119pt}{}\times \Phi_i(K_1,\ldots,K_k;dy) \,dz\\
&&\hspace*{119pt}{}\times\Lambda^{k-1} \bigl(d(K_2,\ldots,K_k)\bigr)
\mathbb{Q}(dK_1)
\\
&=& \gamma\sum_{k=1}^{d-i}\frac{1}{k!}
\iiiint e^{\gamma C_d(y-z)} \I\{z\in K_1\cap\cdots\cap K_k\}
\\
&&\hspace*{67pt}{}\times \Phi_i(K_1,\ldots,K_k;dy) \,dz
\Lambda^{k-1} \bigl(d(K_2,\ldots,K_k)\bigr)
\mathbb{Q}(dK_1),
\end{eqnarray*}
which yields that
\[
\rho_{i,d} =\sum_{k=1}^{d-i}
\frac{1}{k!}\int e^{\gamma C_d(x)} H^k_{i,d}(dx) =
\int e^{\gamma C_d(x)} H_{i,d}(dx).
\]

Finally, we consider the case where $i,j\in\{1,\ldots,d-1\}$.
Again by Lemma~\ref{ldecomp}, we get
\begin{eqnarray*}
\rho_{i,j} & =&\gamma\sum^\infty_{n=1}
\frac{1}{n!}\sum_{k=1}^{d-i}\sum
_{l=1}^{d-j} \mathop{\sum
_{I\subset[n]}}_{ |I|=k}\mathop{\sum
_{J\subset[n]}}_{
|J|=l}\iiiint \I\biggl\{y\in\bigcap
_{r\notin I}K_r^\circ,z\in\bigcap
_{s\notin J}K_s^\circ\biggr\}
\\
&&\hspace*{150pt}{}\times \Phi_i(K_I;dy) \Phi_j(K_J;dz)
\\
&&\hspace*{150pt}{}\times\Lambda^{n-1}\bigl(d(K_2,\ldots,K_n)\bigr)
\mathbb{Q}(dK_1)
\\
& =&\gamma\sum^\infty_{n=1}\frac{1}{n!}
\sum_{k=1}^{d-i} \sum
_{l=1}^{d-j}\sum_{m=0}^{k\wedge l}
\mathop{\sum_{I,J\subset[n]}}_{ |I|=k,|J|=l,|I\cap J|=m }\iiiint \I\biggl
\{y\in\bigcap_{r\notin I}K_r^\circ,z
\in\bigcap_{s\notin J}K_s^\circ
\biggr\}
\\
&&\hspace*{199pt}{}\times \Phi_i(K_I;dy) \Phi_j(K_J;dz)\\
&&\hspace*{199pt}{}\times\Lambda^{n-1}\bigl(d(K_2,\ldots,K_n)\bigr)
\\
&&\hspace*{199pt}{}\times\mathbb{Q}(dK_1).
\end{eqnarray*}
A symmetry argument shows (as before) that for each choice of $I,J$
such that
$|I|=k$, $|J|=l$ and $|I\cap J|=m$,
the preceding integral has the same value.
There are ${n\choose k}{k\choose m}{n-k\choose l-m}$
possible choices of $I,J$ with these properties. Thus, we obtain
\begin{eqnarray*}
\rho_{i,j} & =&\gamma\sum^\infty_{n=1}
\frac{1}{n!}\sum_{k=1}^{d-i} \sum
_{l=1}^{d-j}\sum
_{m=0}^{k\wedge l}\pmatrix{n
\cr
k}\pmatrix{k
\cr
m}
\pmatrix{n-k
\cr
l-m}\\
&&{}\times\idotsint\I\bigl\{y\in K_{k+1}^\circ\cap
\cdots \cap K_n^\circ\bigr\}
\\
&&\hspace*{47pt}{}\times \I\bigl\{z\in K_{1}^\circ\cap\cdots\cap
K_{k-m}^\circ\cap K_{k+l-m+1}^\circ\cap\cdots
\cap K_n^\circ\bigr\} \\
&&\hspace*{47pt}{}\times\Phi_i(K_1,
\ldots,K_k;dy)
\\
&&\hspace*{47pt}{}\times \Phi_j(K_{k+1-m},\ldots,K_{k+l-m};dz)\\
&&\hspace*{47pt}{}\times \Lambda
^{k+l-m-1}\bigl(d(K_1,\ldots ,K_{k+l-m-1})\bigr)
\\
&&\hspace*{47pt}{}\times \mathbb{Q}(dK_{k+l-m}) \Lambda^{n-(k+l-m)}(dK_{k+l-m+1},
\ldots,K_n)
\\
& =&\gamma\sum^\infty_{n=1}\sum
_{k=1}^{d-i}\sum_{l=1}^{d-j}
\sum_{m=0}^{k\wedge l} \frac{\I\{n\ge k+l-m\}\gamma
^{n-(k+l-m)}}{m!(k-m)!(l-m)!(n-(k+l-m))!}
\\
&&{}\times \idotsint\prod_{r=k+l-m+1}^n\I\bigl
\{x_r\in\bigl(K_r^\circ-y\bigr)\cap
\bigl(K_r^\circ -z\bigr)\bigr\}\, d x_{k+l-m+1}\cdots dx_n \\
&&\hspace*{98pt}{}\times\mathbb{Q}(dK_{k+l-m+1})
\cdots\mathbb{Q}(dK_n)
\\
&&\hspace*{98pt}{}\times \I\bigl\{y\in K_{k+1}^\circ\cap\cdots\cap
K_{k+l-m}^\circ\bigr\} \\
&&\hspace*{98pt}{}\times\I\bigl\{z\in K_{1}^\circ
\cap\cdots\cap K_{k-m}^\circ\bigr\} \Phi _i(K_1,
\ldots ,K_k;dy)
\\
&&\hspace*{98pt}{}\times \Phi_j(K_{k+1-m},\ldots,K_{k+l-m};dz)\\
&&\hspace*{98pt}{}\times
\Lambda^{k+l-m-1}\bigl(d(K_1,\ldots,K_{k+l-m-1})\bigr)
\mathbb{Q}(dK_{k+l-m}),
\end{eqnarray*}
and hence
\begin{eqnarray*}
\rho_{i,j} & =&\gamma\sum_{k=1}^{d-i}
\sum_{l=1}^{d-j}\sum
_{m=0}^{k\wedge l}\frac{1}{m!(k-m)!(l-m)!}\\
&&{}\times\iiiint\sum
_{n= k+l-m}^\infty\frac{ (\gamma C_d(y-z)
)^{n-(k+l-m)}}{(n-(k+l-m))!}
\\
&&\hspace*{82pt}{}\times\I\bigl\{y\in K_{k+1}^\circ\cap\cdots\cap
K_{k+l-m}^\circ\bigr\} \I\bigl\{z\in K_{1}^\circ
\cap\cdots\cap K_{k-m}^\circ\bigr\}
\\
&&\hspace*{82pt}{}\times \Phi_i(K_1,\ldots,K_k;dy)
\Phi_j(K_{k+1-m},\ldots,K_{k+l-m};dz)
\\
&&\hspace*{82pt}{}\times \Lambda^{k+l-m-1}\bigl(d(K_1,\ldots,K_{k+l-m-1})\bigr)
\mathbb{Q}(dK_{k+l-m})
\\
& =& \sum_{k=1}^{d-i}\sum
_{l=1}^{d-j}\sum_{m=0}^{k\wedge l}
\frac
{1}{m!(k-m)!(l-m)!} \int e^{\gamma C_d(x)} H_{i,j}^{k,l,m}(dx).
\end{eqnarray*}
This completes the proof of the theorem.
\end{pf*}

Some of the measures in \eqref{3.1neu} and \eqref{3.101} can be
expressed in terms of the
mixed moment measures
\[
M_{i,j}:=\BE\iint\I\bigl\{(y,z)\in\cdot\bigr\} \Phi_i(Z_0;dy)
\Phi_j(Z_0;dz),\qquad i,j\in\{1,\ldots,d\},
\]
and the functions $C_j\dvtx \R^d\rightarrow[0,\infty)$, $j\in\{1,\ldots
,d-1\}
$, defined by
\[
C_j(x):=\BE\Phi_j\bigl(Z_0;Z^\circ_0+x
\bigr),\qquad x\in\R^d.
\]

\begin{lemma}\label{limxed} Assume that \eqref{vi2} is satisfied. Then,
for any $i,j\in\{1,\ldots,d-1\}$,
%
\begin{eqnarray}
\label{101} H^{1}_{i,d}&=&\gamma\int\I\{y-z\in\cdot\}
M_{i,d}\bigl(d(y,z)\bigr),
\\
\label{102} H^{1,1,0}_{i,j}&=&\gamma^2 \int\I
\{y-z\in\cdot\}C_i(y-z) M_{j,d}\bigl(d(z,y)\bigr),
\\
\label{103} H^{1,1,1}_{i,j}&=&\gamma\int\I\{y-z\in\cdot\}
M_{i,j}\bigl(d(y,z)\bigr).
\end{eqnarray}
\end{lemma}

\begin{pf} Equations \eqref{101} and \eqref{103}
follow directly from the definitions, while~\eqref{102}
follows from an easy calculation using the covariance property \eqref{3.67a}.
\end{pf}

In the next section, we will use the following consequences of Lemma~\ref{limxed}:
%
\begin{eqnarray}
\label{d-1,d} H_{d-1,d}&= &\gamma\int\I\{y-z\in\cdot\} M_{d-1,d}
\bigl(d(z,y)\bigr),
\\
\label{d-1,d-1}
H_{d-1,d-1}&= &\gamma^2 \int\I\{y-z\in
\cdot\}C_{d-1}(y-z) M_{d-1,d}\bigl(d(z,y)\bigr)
\nonumber
\\[-8pt]
\\[-8pt]
\nonumber
&&{}+\gamma\int\I\{y-z\in\cdot\} M_{d-1,d-1}\bigl(d(y,z)\bigr).
\end{eqnarray}

\section{Covariance structure in the isotropic case}\label{seciso}

In this section, we assume that the typical grain is \emph{isotropic},
that is, its distribution $\BQ$ is invariant under rotations and that
the moment assumption \eqref{vi2} is satisfied.
Our aim is to derive more explicit formulas
for the asymptotic covariances
%
\begin{equation}\qquad
\label{sigmaij} \sigma_{i,j}:= \lim_{r(W)\to\infty}
\frac{\C(V_i(Z\cap W),V_j(Z\cap W))}{V_d(W)},\qquad i,j\in\{0,\ldots,d\};
\end{equation}
confer the statement of Theorem~\ref{thm:variance}.

Using the iterated version of the local kinematic formula
(\cite{SW08}, Theorem~5.3.2), which is obtained by combining
\cite{SW08}, Theorem~6.4.1, (b) and (6.15) and \cite{SW08}, Theorem~6.4.2,
(6.20),
we get for $j\in\{0,\ldots,d-1\}$ and $l\in\{1,\ldots,d-j\}$ that
\begin{eqnarray*}
&&\gamma\iint\Phi_j\bigl(K_1,\ldots,K_{l},
\R^d\bigr) \Lambda^{l-1}\bigl(d(K_1,
\ldots,K_{l-1})\bigr) \BQ(dK_l)\\
&&\qquad =\mathop{\sum
^{d-1}_{m_1,\ldots,m_l=j}}_{ m_1+\cdots+m_l=(l-1)d+j} c^{d}_j
\prod^l_{i=1}c^{m_i}_d
\gamma v_{m_i},
\end{eqnarray*}
where, as in \cite{SW08}, (5.4),
\[
c^m_j:=\frac{m!\kappa_m}{j!\kappa_j},\qquad  m,j\in\{0,\ldots,d\}.
\]
Combining this with \eqref{3.9} and Theorem~\ref{tdeltaij}, and under
assumption
\eqref{condfull}, we deduce
%
\begin{equation}
\label{3.94} \rho_{0,j}=e^{\gamma v_d}P_j(\gamma
v_j,\ldots,\gamma v_{d-1}),\qquad j\in \{ 0,\ldots,d-1\},
\end{equation}
where $P_{j}$ (a multivariate polynomial on $\R^{d-j}$ of degree $d$)
is defined by
\[
P_{j}(t_j,\ldots,t_{d-1}) :=c^d_j
\sum^{d-j}_{l=1}\frac{1}{l!}\mathop{
\sum^{d-1}_{m_1,\ldots
,m_l=j}}_{ m_1+\cdots+m_l=(l-1)d+j} \prod
^l_{i=1}c^{m_i}_d
t_{m_i}.
\]

The following main result of this section shows that
the asymptotic covariances~\eqref{sigmaij} are linear
combinations of the numbers $\rho_{i,j}$ given by \eqref{rhoijlimit}.
To describe
the coefficients, we define
for any $j\in\{0,\ldots,d-1\}$ and $l\in\{j,\ldots,d\}$
a polynomial $P_{j,l}$ on $\R^{d-j}$ of degree $l-j$ by
%
\begin{equation}\qquad
\label{Polyjl} P_{j,l}(t_j,\ldots,t_{d-1}):=\I
\{l=j\}+c^l_j\sum^{l-j}_{s=1}
\frac{(-1)^s}{s!} \mathop{\sum^{d-1}_{m_1,\ldots,m_s=j}}_{ m_1+\cdots+m_s=sd+j-l}
\prod^s_{i=1}c^{m_i}_dt_{m_i}
\end{equation}
and complement this definition by $P_{d,d}:=1$.

\begin{theorem}\label{tsigmaij} Assume that the typical grain is
isotropic and suppose that~\eqref{vi2} holds.
Then
\[
\sigma_{i,j}=(1-p)^2\sum^{d}_{k=i}
\sum^{d}_{l=j}P_{i,k}(\gamma
v_i,\ldots,\gamma v_{d-1}) P_{j,l}(\gamma
v_j,\ldots,\gamma v_{d-1})\rho_{k,l},
\]
for all $i,j\in\{0,\ldots,d\}$.
\end{theorem}

\begin{pf} The formula preceding Theorem~9.1.4 in \cite{SW08}
is the finite volume version of the fundamental result
of \cite{Miles76} and \cite{Davy76} on the densities of
intrinsic volumes. Using this result, we obtain
for all $i\in\{0,\ldots, d-1\}$ and $A\in\cK^d$ that
%
\begin{equation}
\label{e4.56} \BE V_i(Z\cap A)-V_i(A)=-(1-p)\sum
^d_{k=i} V_k(A)P_{i,k}(
\gamma v_i,\ldots,\gamma v_{d-1}).
\end{equation}
For $i=d$, equation \eqref{e4.56} is a direct consequence of stationarity
and the definition $P_{d,d}=1$.
Using this formula in \eqref{eqn:sigma}, we obtain
the assertion from \eqref{rhoijlimit}.
\end{pf}

\begin{corol}\label{surfvol} Assume that \eqref{vi2} is satisfied. Then,
for $i,j\in\{d-1,d\}$, the assertions of Theorem~\ref{tsigmaij}
remain true in the general stationary case (without isotropy assumption).
Moreover,
\begin{eqnarray*}
\sigma_{d,d}&=&(1-p)^2\int \bigl(e^{\gamma C_d(x)}-1 \bigr)
\,dx,
\\
\sigma_{d-1,d}&=& -(1-p)^2 \gamma v_{d-1}\int
\bigl(e^{\gamma C_d(x)}-1 \bigr) \,dx \\
&&{}+(1-p)^2\gamma\int e^{\gamma C_d(x-y)}
M_{d-1,d}\bigl(d(x,y)\bigr).
\end{eqnarray*}
If, in addition, \eqref{condfull} holds, then
\begin{eqnarray*}
\sigma_{d-1,d-1}&=&(1-p)^2\gamma^2
v^2_{d-1}\int \bigl(e^{\gamma C_d(x)}-1 \bigr) \,dx
\\
&&{} +(1-p)^2\gamma^2\int e^{\gamma C_d(x-y)}
\bigl(C_{d-1}(x-y)-2v_{d-1}\bigr) M_{d-1,d}\bigl(d(y,x)
\bigr)
\\
& &{}+(1-p)^2\gamma\int e^{\gamma C_d(x-y)} M_{d-1,d-1}\bigl(d(x,y)
\bigr).
\end{eqnarray*}
\end{corol}

\begin{pf} The formula preceding Theorem~9.1.4 in \cite{SW08}
does not require isotropy for $j=d-1$. Therefore, for $i,j\in\{d-1,d\}$,
the proof of Theorem~\ref{tsigmaij} applies without this assumption.

By definition \eqref{Polyjl}, $P_{d-1,d-1}=P_{d,d}=1$ and
$P_{d-1,d}(\gamma v_{d-1})=-\gamma v_{d-1}$. Therefore, we obtain from
Theorem~\ref{tsigmaij} that $\sigma_{d,d}=(1-p)^2\rho_{d,d}$,
$\sigma
_{d-1,d}=(1-p)^2 (\rho_{d-1,d}-\gamma v_{d-1}\rho_{d,d}
)$, and
\[
\sigma_{d-1,d-1} =(1-p)^2\bigl(\rho_{d-1,d-1}-\gamma
v_{d-1}\rho_{d-1,d} -\gamma v_{d-1}
\rho_{d,d-1}+\gamma^2 v^2_{d-1}
\rho_{d,d}\bigr).
\]
Inserting first \eqref{3.11}, \eqref{formularhoij} and then \eqref
{d-1,d} and
\eqref{d-1,d-1}, we obtain the result.
\end{pf}

Together with Corollary~\ref{surfvol} the next corollary
provides rather explicit formulas for the asymptotic covariance
in the two-dimensional isotropic case.

\begin{corol}\label{2d} Let $d=2$, assume that the typical grain is
isotropic, and suppose that \eqref{vi2} and \eqref{condfull} are
satisfied. Then
\begin{eqnarray*}
\sigma_{0,0}&=&(1-2p) (1-p)\gamma+(1-p) (2p-3)\frac{\gamma^2v_1^2}{\pi}
\\
&&{}+(1-p)^2 \biggl(\gamma-\frac{\gamma^2v_1^2}{\pi} \biggr)^2 \int
\bigl(e^{\gamma C_2(x)}-1 \bigr) \,dx
\\
&&{} +(1-p)^2\int\chi(x-y) M_{1,2}\bigl(d(y,x)\bigr)\\
&&{} +
\frac{4}{\pi^2}(1-p)^2\gamma^3 v^2_1
\int e^{\gamma C_2(x-y)} M_{1,1}\bigl(d(x,y)\bigr),
\\
\sigma_{0,1}&=&(1-p)^2\gamma v_1
+(1-p)^2 \biggl(\gamma^2v_1-
\frac{\gamma^3v_1^3}{\pi} \biggr)\int \bigl(e^{\gamma C_2(x)}-1 \bigr) \,dx
\\
&&{} +(1-p)^2\int\tilde{\chi}(x-y) M_{1,2}\bigl(d(y,x)\bigr)
\\
&&{}-(1-p)^2\frac{2\gamma^2 v_1}{\pi}\int e^{\gamma C_2(x-y)} M_{1,1}
\bigl(d(x,y)\bigr),
\\
\sigma_{0,2}&=&p(1-p) -(1-p)^2 \biggl(\gamma-
\frac{\gamma^2v_1^2}{\pi} \biggr)\int \bigl(e^{\gamma
C_2(x)}-1 \bigr) \,dx
\\
&&{} -(1-p)^2\frac{2\gamma^2 v_1}{\pi}\int e^{\gamma C_2(x-y)} M_{1,2}
\bigl(d(x,y)\bigr),
\end{eqnarray*}
where
\begin{eqnarray*}
\chi(z)&:=&e^{\gamma C_2(z)} \biggl(\frac{4\gamma^4 v^2_1}{\pi^2}\bigl(C_1(z)-v_1
\bigr)+\frac{4\gamma^3
v_1}{\pi
} \biggr),
\\
\tilde{\chi}(z)&:=& e^{\gamma C_2(z)} \biggl(\frac{3\gamma^3 v^2_1}{\pi}-
\frac{2\gamma^3 v_1}{\pi
}C_1(z)-\gamma ^2 \biggr).
\end{eqnarray*}
The formula for $\sigma_{0,2}$ remains true without assumption \eqref
{condfull}.
\end{corol}

\begin{pf} We have $P_{0,0}(t_0,t_1)=1$, $P_{0,1}(t_0,t_1)=-\frac
{2}{\pi}t_1$,
$P_{0,2}(t_0,t_1)=-t_0+\frac{1}{\pi}t_1^2$,
$P_{1,1}(t_1)=1$, $P_{1,2}(t_1)=-t_1$, and $P_{2,2}(t_1)=1$.
Moreover, we have\break $P_0(t_0,  t_1)=t_0+\frac{1}{\pi}t_1^2$ and
$P_1(t_1)=t_1$. Using \eqref{3.94}, Theorem~\ref{tdeltaij} and Lemma~\ref{limxed}, we obtain
\begin{eqnarray*}
\rho_{0,0}& =&e^{\gamma v_2} \biggl(\gamma+\frac{\gamma^2v_1^2}{\pi
}
\biggr), \qquad\rho_{0,1}=e^{\gamma v_2}\gamma v_1,\qquad
\rho_{0,2}= e^{\gamma v_2}-1,
\\
\rho_{1,1}&=& \gamma^2\int e^{\gamma C_2(y-z)}
C_1(y-z) M_{1,2}\bigl(d(z,y)\bigr) +\gamma\int
e^{\gamma C_2(y-z)} M_{1,1}\bigl(d(y,z)\bigr),
\\
\rho_{1,2}&=& \gamma\int e^{\gamma C_2(y-z)} M_{1,2}\bigl(d(y,z)
\bigr),\qquad \rho_{2,2}= \int \bigl(e^{\gamma C_2(x)}-1 \bigr) \,dx.
\end{eqnarray*}
The result follows by substituting these expressions into Theorem~\ref
{tsigmaij}.
\end{pf}

The proof of Theorem~\ref{tsigmaij} also yields the following
nonasymptotic result for which definition \eqref{ijW} should be
recalled. The case $d=2$ is further discussed in Appendix B of \cite{HLSarxiv}.

\begin{theorem}\label{tsigmaijW}
Assume that the typical grain is
isotropic and that \eqref{vi2} holds.
Let $W\in\cK^d$ and $i,j\in\{0,\ldots,d\}$.
Then
\begin{eqnarray*}
&&\C\bigl(V_i(Z\cap W),V_j(Z\cap W)\bigr)\\
&&\qquad=(1-p)^2
\sum^{d}_{k=i}\sum
^{d}_{l=j}P_{i,k}(\gamma v_i,
\ldots,\gamma v_{d-1}) P_{j,l}(\gamma v_j,\ldots,
\gamma v_{d-1})\rho_{k,l}(W).
\end{eqnarray*}
\end{theorem}

\section{The spherical Boolean model}
\label{secspherical}

In this section, we show how some of the formulas of Section~\ref{seciso} can be used to
determine explicitly the covariances of a stationary and isotropic
Boolean model
whose typical grain is the unit ball $B^d$. In this particular case,
we get from Corollary~\ref{surfvol} that
\begin{eqnarray*}
&&\sigma_{d-1,d}=(1-p)^2\gamma \biggl[-v_{d-1}\int
\bigl(e^{\gamma C_d(x)}-1\bigr) \,dx\\
&&\hspace*{95pt}{}+\frac{1}{2}\int_{{\S}^{d-1}}
\int_{B^d}e^{\gamma C_d(x-y)} \,dy \mathcal {H}^{d-1}(dx)
\biggr],
\end{eqnarray*}
where $C_d(x)=V_d(B^d\cap(B^d+x))$ and $\mathcal{H}^j$ denotes the
$j$-dimensional Hausdorff measure. Clearly, $\bar{C}_d(t):=V_d(B^d\cap
(B^d+tv))$, for $t\ge0$ and $v\in\S^{d-1}$, is independent of the
choice of the unit vector $v$ and
\begin{eqnarray*}
\bar{C}_d(t)&=&2\kappa_{d-1}\int_{{t}/{2}}^1
\sqrt{1-u^2}^{d-1} \,du \\
&=&2\frac{\pi^{{(d-1)}/{2}}}{\Gamma ({(d+1)}/{2}
)}\int
_{{t}/{2}}^1\sqrt{1-u^2}^{d-1}
\,du,\qquad t\in[0,2].
\end{eqnarray*}
Introducing polar coordinates, we get
\begin{eqnarray*}
F_d(\gamma)&:=&v_{d-1}\int\bigl(e^{\gamma C_d(x)}-1\bigr) \,dx
=v_{d-1} \,d\kappa_d\int_0^2
\bigl(e^{\gamma\bar{C}_d(t)}-1\bigr)t^{d-1} \,dt\\
&=:&v_{d-1}
f_d(\gamma),
\end{eqnarray*}
where $v_{d-1}=d\kappa_d/2$.
On the other hand, for an arbitrary unit vector $v\in\S^{d-1}$, by the
rotation invariance of $B^d$ we get
\[
G_d(\gamma):=\frac{1}{2}\int_{{\S}^{d-1}}\int
_{B^d}e^{\gamma C_d(x-y)} \,dy \mathcal{H}^{d-1}(dx)
=v_{d-1} \int_{B^d}e^{\gamma C_d(v-y)} \,dy.
\]
We parameterize $y$ in the form
\[
y=(1-t)v+\sqrt{1-(1-t)^2}s w,\qquad t\in[0,2], s\in[0,1], w\in
v^\perp \cap \S^{d-1},
\]
and hence we obtain
%
\begin{eqnarray}
\label{Gstar} G_d(\gamma) &=&v_{d-1} (d-1)
\kappa_{d-1} \int_{0}^2\int
_0^1\exp \bigl(\gamma \bar {C}_d
\bigl(\sqrt{(2-t)^2+t(2-t)s^2} \bigr)
\bigr)\nonumber\\
&&\hspace*{106pt}{}\times s^{d-2}\sqrt{t(2-t)}^{d-1} \,ds \,dt
\nonumber
\\
&=:&v_{d-1} g_d(\gamma).\nonumber
\end{eqnarray}
Therefore, we have
\[
\sigma_{d-1,d}=(1-p)^2\gamma v_{d-1}
\bigl(-f_d(\gamma)+g_d(\gamma ) \bigr),
\]
which shows that the sign of the covariance $\sigma_{d-1,d}$ is
completely determined by the sign of the
function $g_d-f_d$.

It is preferable to plot the covariances as functions of the intensity.
Here, we get
\[
\sigma_{d-1,d}(\gamma)=\gamma e^{-2\kappa_d\gamma}v_{d-1} \bigl(
g_d(\gamma)-f_d(\gamma) \bigr).
\]
Figure~\ref{covcolorall} shows the result for various dimensions.

\begin{figure}[b]

\includegraphics{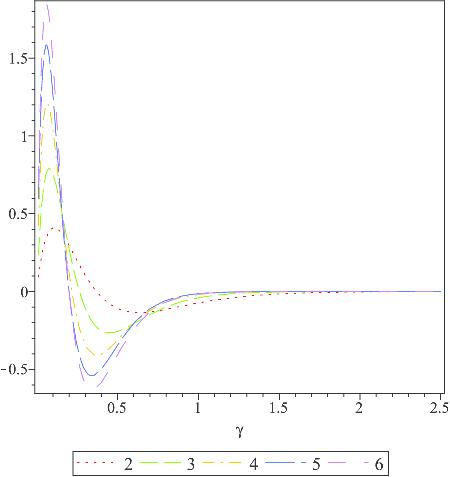}

\caption{$\sigma_{d-1,d}(\gamma)$ for $d=2, \ldots, d=6$.}\label
{covcolorall}
\end{figure}

Next, we determine the correlation coefficient $\cor_{d-1,d}(\gamma)$,
as a function of the intensity $\gamma$. For this, we also have to
determine explicitly
$\sigma_{d,d}$ and $\sigma_{d-1,d-1}$, which requires some further
calculations. First, we have
\[
\sigma_{d,d}=(1-p)^2\int\bigl(e^{\gamma C_d(x)}-1\bigr)
\,dx=(1-p)^2 f_d(\gamma),
\]
hence $\sqrt{\sigma_{d,d}}=(1-p)\sqrt{f_d(\gamma)}$; second,
\begin{eqnarray*}
\sigma_{d-1,d-1}&=&(1-p)^2\gamma^2 \biggl[
(v_{d-1} )^2f_d(\gamma )\\
&&\hspace*{54pt}{}+\int
e^{\gamma C_d(x-y)}C_{d-1}(x-y) M_{d-1,d}\bigl(d(y,x)\bigr)
\\
&&\hspace*{54pt}{} -2v_{d-1} \int e^{\gamma C_d(x-y)} M_{d-1,d}
\bigl(d(y,x)\bigr)\\
&&\hspace*{89pt}{}+
\frac{1}{\gamma
}\int e^{\gamma C_d(x-y)} M_{d-1,d-1}\bigl(d(x,y)\bigr)
\biggr].
\end{eqnarray*}
Since $C_{d-1}(x)$ depends only on $\|x\|$, we denote it by $\bar
{C}_{d-1}(\|x\|)$. For $0<\|x\|\le2$, we then get
\[
\bar{C}_{d-1}\bigl(\|x\|\bigr)=\frac{1}{2}\mathcal{H}^{d-1}
\bigl({\S}^{d-1}\cap \bigl(B^d+x\bigr)\bigr)=
\frac{1}{2}(d-1) \kappa_{d-1}\int_{{\|x\|}/{2}}^1
\sqrt{1-s^2}^{d-3} \,ds.
\]
Let $v\in{\S}^{d-1}$ be fixed. Then, arguing as in the derivation of
\eqref{Gstar}, we obtain
\begin{eqnarray*}
&&\int e^{\gamma C_d(x-y)}C_{d-1}(x-y) M_{d-1,d}\bigl(d(y,x)\bigr)\\
&&\qquad
=v_{d-1}\int_{B^d}e^{\gamma C_d(x-v)}C_{d-1}(x-v)
\,dx
\\
&&\qquad=v_{d-1}(d-1)\kappa_{d-1}\int_{0}^2
\int_0^1 s^{d-2}\sqrt
{t(2-t)}^{d-1}\\
&&\hspace*{139pt}{}\times \exp \bigl(\gamma\bar{C}_d \bigl(\sqrt
{(2-t)^2+t(2-t)s^2} \bigr) \bigr)
\\
&&\hspace*{139pt}{}\times \bar{C}_{d-1} \bigl(\sqrt{(2-t)^2+t(2-t)s^2}
\bigr) \,ds \,dt
\\
&&\qquad=:v_{d-1}(d-1)\kappa_{d-1}h_d(\gamma).
\end{eqnarray*}
Furthermore, we have
\[
2v_{d-1}\int e^{\gamma C_d(x-y)} M_{d-1,d}\bigl(d(x,y)
\bigr)=2(v_{d-1})^2 g_d(\gamma)=
\frac{(d\kappa_d)^2}{2} g_d(\gamma).
\]
Finally, since
\[
M_{d-1,d-1}=\frac{1}{4}\int_{{\S}^{d-1}}\int
_{{\S}^{d-1}}\mathbf {1}\bigl\{ (y,z)\in\cdot\bigr\}
\mathcal{H}^{d-1}(dy) \mathcal{H}^{d-1}(dz),
\]
we get (with an arbitrary unit vector $v_0$)
\begin{eqnarray*}
&&\int e^{\gamma C_d(x-y)} M_{d-1,d-1}\bigl(d(x,y)\bigr)
\\
&&\qquad=\frac{d\kappa_d}{4}\int_{{\S}^{d-2}}\int_0^{\pi}
\exp \bigl(\gamma C_d \bigl(v_0- [\cos\theta
v_0+\sin\theta v ] \bigr) \bigr)\sin^{d-2}\theta \,d\theta
\mathcal{H}^{d-2}(dv)
\\
&&\qquad=\frac{d\kappa_d(d-1)\kappa_{d-1}}{4}\int_0^\pi
\sin^{d-2}\theta \exp \bigl(\gamma\bar{C}_d \bigl(\sqrt{2(1-
\cos\theta)} \bigr) \bigr) \,d\theta
\\
&&\qquad=\frac{d\kappa_d(d-1)\kappa_{d-1}}{4}\int_{0}^2
\sqrt{s(2-s)}^{d-3} \exp \bigl(\gamma\bar{C}_d \bigl(
\sqrt{2(2-s)} \bigr) \bigr) \,d s \\
&&\qquad=:\frac{d\kappa_d(d-1)\kappa_{d-1}}{4}k_d(\gamma).
\end{eqnarray*}
Hence, we have
\begin{eqnarray*}
\frac{\sigma_{d-1,d-1}}{(1-p)^2\gamma^2}&=& \biggl(\frac{d\kappa
_d}{2} \biggr)^2f_d(
\gamma)-\frac{(d\kappa_d)^2}{2}g_d(\gamma)+ \frac{d\kappa_d(d-1)\kappa_{d-1}}{2}
h_d(\gamma)\\
&&{}+\frac{d\kappa
_d(d-1)\kappa_{d-1}}{4\gamma}k_d(\gamma).
\end{eqnarray*}
This finally implies that
\begin{eqnarray*}
&&\cor_{d-1,d}(\gamma)\\
&&\qquad=\biggl(\frac{d\kappa_d}{2} \bigl(g_d(\gamma
)-f_d(\gamma) \bigr)\biggr)\\
&&\qquad\quad{}\Big/\biggl(\sqrt{f(\gamma)}\biggl(\biggl(\frac{d\kappa
_d}{2} \biggr)^2f_d(\gamma)-\frac{(d\kappa_d)^2}{2}g_d(\gamma)\\
&&\hspace*{87pt}{}+
\frac{d\kappa_d(d-1)\kappa_{d-1}}{2} h_d(\gamma)+\frac{d\kappa
_d(d-1)\kappa_{d-1}}{4\gamma}k_d(\gamma)\biggr)^{1/2}\biggr).
\end{eqnarray*}
From these considerations, we also deduce the plausible fact that
\[
\lim_{\gamma\downarrow0} \cor_{d-1,d}(\gamma)=\lim
_{\gamma
\downarrow
0}\frac{({1}/{2})d\kappa_d\kappa_d}{\sqrt{\gamma\int C_d(x) \,dx}
\sqrt{({1}/{\gamma}) ({d\kappa_d}/{2} )^2}}=1,
\]
which is confirmed by our numerical calculations. Plots of $\cor
_{d-1,d}(\cdot)$ for $d=2,\ldots,6$ are shown in Figure~\ref{figcorr2}.

\begin{figure}

\includegraphics{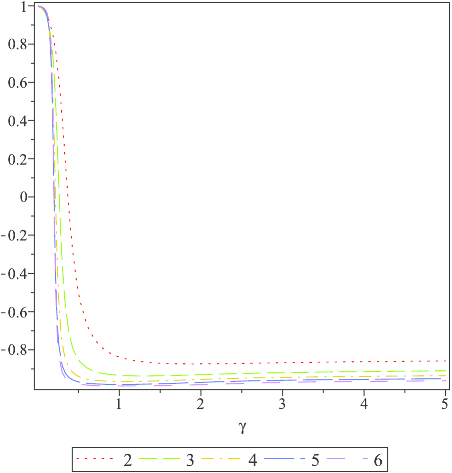}

\caption{$\cor_{d-1,d}(\gamma)$ for $d=2,\ldots,6$.}\label{figcorr2}
\end{figure}

In a similar way, the formulas from Corollary~\ref{2d} can be specified
in the case of a planar Boolean model
with the unit circle as deterministic typical grain. Then we have
\begin{eqnarray*}
\chi(r,\gamma)&=&4\gamma^3 e^{\gamma\bar C_2(r)} \bigl(\gamma\bar
C_1(r)-\pi\gamma+1 \bigr) \quad\mbox{and}\\
\tilde \chi(r,\gamma)&=&
\gamma^2 e^{\gamma\bar C_2(r)} \bigl(3\pi\gamma-2 \bar C_1(r)
\gamma-1 \bigr),
\end{eqnarray*}
and, for instance,
\begin{eqnarray*}
\sigma_{0,0}(\gamma)&=&(1-2p) (1-p)\gamma+(1-p) (2p-3)
\gamma^2\pi\\
&&{} +(1-p)^2\gamma^2(1-\pi
\gamma)^2f_2(\gamma)
\\
&&{} +(1-p)^22\pi\int_0^2\int
_0^1\chi \bigl(\sqrt {(2-t)^2+t(2-t)s^2},
\gamma \bigr) \sqrt{t(2-t)} \,ds \,dt
\\
&&{} +(1-p)^2\gamma^3 4\pi\int_0^{\pi}
\exp \bigl(\gamma\bar C_2 \bigl(\sqrt{2\bigl(1-\cos(t)\bigr)} \bigr)
\bigr) \,dt,
\end{eqnarray*}
where $p=p(\gamma)=1-e^{-\pi\gamma}$. Moreover,
\begin{eqnarray*}
\sigma_{0,1}(\gamma)&=&(1-p)^2\gamma\pi+(1-p)^2
\gamma^2\pi(1-\pi \gamma )f_2(\gamma)
\\
&&{} +(1-p)^22\pi\int_0^2\int
_0^1\tilde\chi \bigl(\sqrt {(2-t)^2+t(2-t)s^2},
\gamma \bigr)\sqrt{t(2-t)} \,ds \,dt
\\
& &{}- (1-p)^2\gamma^2 2\pi\int_0^{\pi}
\exp \bigl(\gamma\bar C_2 \bigl(\sqrt{2\bigl(1-\cos(t)\bigr)} \bigr)
\bigr) \,dt,
\\
\sigma_{0,2}(\gamma)&=&p(1-p)-(1-p)^2\gamma(1-\pi
\gamma)f_2(\gamma )-(1-p)^22\gamma^2\pi
g_2(\gamma),
\\
\sigma_{2,2}(\gamma)&=&(1-p)^2f_2(\gamma).
\end{eqnarray*}
The variances and covariances as well as the correlation functions for
the planar case are plotted in Figures~\ref{figcorr3} and~\ref{figcorr4}.

\begin{figure}

\includegraphics{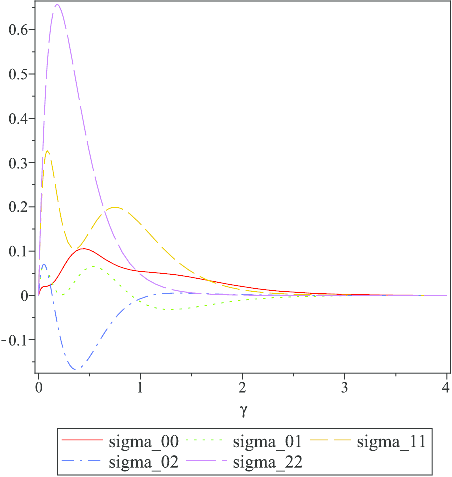}

\caption{Variances/covariances for $d=2$.}\label{figcorr3}
\end{figure}

\begin{figure}

\includegraphics{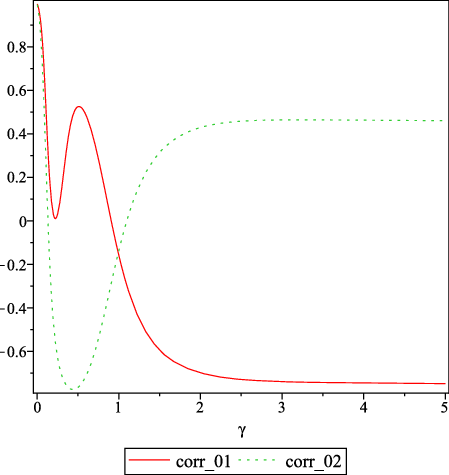}

\caption{Correlation functions for $d=2$.}\label{figcorr4}
\end{figure}

\section{Normal approximation via the Malliavin--Stein method}
\label{secclt}

In this section, we prepare the central limit theorems for geometric
functionals of a
Boolean model by proving a general result on the normal approximation
of Poisson functionals. Our approach is based on recent
findings in \cite{Peccatietal2010,PeccatiZheng2010} and uses similar
arguments as in \cite{ReitznerSchulte2011}.

Throughout this section, let $\eta$ be a Poisson process on a
measurable space $(\BX,{\mathcal X})$ with a
$\sigma$-finite intensity measure $\lambda$;
see \cite{Kallenberg}, Chapter~12.
Consider a $[-\infty,\infty]$-valued
random variable $F$ such that $\BP(|F|<\infty)=1$
and $F=f(\eta)$
$\BP$-a.s. for some measurable $f\dvtx \bN\rightarrow\R$. Any such $f$ is
called a
\emph{representative} of the \emph{Poisson functional} $F$.
If $f$ is a (fixed) representative of $F$, we define
\[
D^n_{x_1,\ldots,x_n}F:=D^n_{x_1,\ldots,x_n}f(\eta), \qquad n\in
\N, x_1,\ldots ,x_n\in\BX,
\]
where $D^n$ is the $n$th iterated difference operator used in Section~\ref{sec:3}.
If $\tilde{f}$ is another representative of $F$, then
the multivariate Mecke equation (see, e.g., \cite{LaPe11}, (2.10))
implies that
$D^n_{x_1,\ldots,x_n}f(\eta)=D^n_{x_1,\ldots,x_n}\tilde{f}(\eta)$
$\BP$-a.s. and for $\lambda^n$-a.e. $(x_1,\ldots,x_n)\in\BX^n$.
Let $L^2_\eta$ denote the space of all Poisson functionals $F$ such
that $\BE F^2<\infty$.
For $F\in L^2_\eta$ we define $f_n\dvtx \BX^n\to\R$ by
\[
f_n(x_1,\ldots,x_n)=\frac{1}{n!}\E
D^n_{x_1,\ldots,x_n}F.
\]
It was shown in \cite{LaPe11}, Theorem~1.1, that $f_n$ belongs to the space
$L^2_s(\lambda^n)$ of $\lambda^n$-almost everywhere symmetric functions
on $\BX^n$ that are
square-integrable with respect to $\lambda^n$.
Now the Fock space representation (see \cite{LaPe11}, Theorem~1.1)
tells us that
%
\begin{equation}
\label{Fock} \V F=\sum_{n=1}^\infty n!
\|f_n\|_n^2,
\end{equation}
where $\|\cdot\|_n$ denotes the norm in $L^2(\lambda^n)$.
Moreover, it is known from \cite{LaPe11}, Theorem~1.3, that $F$ has the
representation
%
\begin{equation}
\label{eqn:WienerItoChaos} F=\E F + \sum_{n=1}^\infty
I_n(f_n),
\end{equation}
where $I_n(\cdot)$ stands for the $n$th multiple Wiener--It\^o integral,
and the right-hand side converges in $L^2(\P)$.
The identity \eqref{eqn:WienerItoChaos} is called Wiener--It\^o chaos expansion
of $F$. The multiple Wiener--It\^o integrals are defined
for square integrable symmetric functions and are orthogonal in the
sense that
\[
\E I_n(f) I_m(g) = \cases{ n! \langle f,g
\rangle_n, &\quad $n=m$,\vspace*{2pt}
\cr
0, &\quad  $n\neq m$,}
\]
for $f\in L_s^2(\lambda^n)$, $g\in L^2_s(\lambda^m)$, and $n,m\in\N
$, where
$\langle\cdot,\cdot\rangle_n$ denotes the scalar product in
$L^2(\lambda^n)$.

If the condition
%
\begin{equation}
\label{eqn:domD} \sum_{n=1}^\infty n n!
\|f_n\|_n^2 <\infty
\end{equation}
is satisfied, the difference operator \eqref{addone} has the representation
%
\begin{equation}
\label{Dx} D_x F=\sum_{n=1}^\infty
n I_{n-1}\bigl(f_n(x,\cdot)\bigr)
\end{equation}
$\P$-a.s. for $\lambda$-a.e. $x\in\BX$ (see, e.g., \cite{LaPe11}, Theorem~3.3). From now on, we write $F\in\dom D$ if $F\in L^2_\eta$
satisfies \eqref{eqn:domD}. The Ornstein--Uhlenbeck generator
associates with any Poisson functional $F\in L^2_\eta$ such that $\sum_{n=1}^\infty n^2 n! \|f_n\|_n^2<\infty$ the random variable
\[
L F= -\sum_{n=1}^{\infty} n
I_n(f_n),
\]
and its pseudo-inverse is given by
%
\begin{equation}
\label{eqn:L-1} L^{-1} F = -\sum_{n=1}^\infty
\frac{1}{n} I_n(f_n)
\end{equation}
for $F\in L^2_\eta$. These operators together with the difference
operator and the Skorohod integral, which is not used in this paper,
are called Malliavin operators. Combining \eqref{Dx} and \eqref
{eqn:L-1}, we see that
%
\begin{equation}
\label{DxL} D_x L^{-1} F = -\sum
_{n=1}^\infty I_{n-1}\bigl(f_n(x,
\cdot)\bigr)
\end{equation}
$\P$-a.s. for $\lambda$-a.e. $x\in\BX$. More details on the
Wiener--It\^o chaos expansion
and the Malliavin operators can be found in \cite{LaPe11} and the
references therein. In \cite{Peccatietal2010,PeccatiZheng2010},
the Malliavin operators and Stein's method are combined to derive
bounds for the normal approximation of Poisson functionals. In the following,
we evaluate bounds obtained by this technique, which
is called the Malliavin--Stein method.

To measure the distance between two real-valued random variables $Y_1,Y_2$,
we use the Wasserstein distance that is given by
\[
\mathbf{d}_W(Y_1,Y_2)=\sup
_{h\in\Lip(1)}\bigl|\E h(Y_1)-\E h(Y_2)\bigr|.
\]
Here, $\Lip(1)$ stands for the set of all functions $h\dvtx \R\to\R$
with a Lipschitz constant less than or equal to one. For two
$m$-dimensional random vectors $Y_1,Y_2$, we define
\[
\mathbf{d}_3(Y_1,Y_2)=\sup
_{h\in{\mathcal H}}\bigl|\E h(Y_1)-\E h(Y_2)\bigr|,
\]
where ${\mathcal H}$ is the set of all three times continuously differentiable
functions $h\dvtx \R^m\to\R$ such that
\[
\max_{i,j=1,\ldots,m} \sup_{x\in\R^m} \biggl\llvert
\frac{\partial^2 h}{\partial x_i\, \partial x_j}(x)\biggr\rrvert \leq1
\quad \mbox{and}\quad \max_{i,j,k=1,\ldots,m}
\sup_{x\in\R^m} \biggl\llvert \frac{\partial^3 h}{\partial x_i\, \partial x_j\, \partial
x_k}(x)\biggr\rrvert
\leq1.
\]
Convergence in the Wasserstein distance or in the
$\mathbf{d}_3$-distance implies convergence in distribution.

In the following, we establish an upper bound for the $\mathbf
{d}_3$-distance between a Gaussian random vector and
a random vector $F=(F^{(1)},\ldots,F^{(m)})$ of Poisson functionals
$F^{(1)},\ldots,F^{(m)}\in L^2_\eta$.
Each of these components has a Wiener--It\^o chaos expansion
\[
F^{(k)}=\E F^{(k)}+\sum_{n=1}^\infty
I_n\bigl(f_n^{(k)}\bigr)
\]
with $f^{(k)}_n\in L_s^2(\lambda^n)$, $n\in\N$. We also state a bound
for the Wasserstein distance between the normalization
of a Poisson functional $F$ and a standard Gaussian random variable.

We need to introduce some notation. Consider functions
$g_1\dvtx \BX^{n_1} \to\R$ and $g_2\dvtx \BX^{n_2} \to\R$, where
$n_1,n_2\in\N$.
The tensor product $g_1\otimes g_2$ is the function on $\BX^{n_1+n_2}$
which maps each $(x_1,\ldots,x_{n_1+n_2})$ to
$g_1(x_1,\ldots,x_{n_1})g_2(x_{n_1+1},\ldots,\break x_{n_1+n_2})$.
This definition can be iterated in the obvious way.
Fix two integers $i,j\ge1$ and consider functions
$f\dvtx \BX^i\to\R$ and $g\dvtx \BX^j\to\R$. Let $\sigma$ be a partition
of $I_{ij}:=\{1,\ldots,2i+2j\}$ and let $|\sigma|$ be the number of blocks
(i.e., the disjoint sets constituting the partition) of $\sigma$.
The function
$(f\otimes f\otimes g \otimes g)_\sigma\dvtx  \BX^{|\sigma|}\rightarrow
\R$
is defined by replacing all variables whose indices belong to the same
block of $\sigma$
by a new common variable.
Let $\pi=\{J_1,\ldots,J_4\}$ be the partition of $I_{ij}$ into the sets
$J_1:=\{1,\ldots,i\}$, $J_2:=\{i+1,\ldots,2i\}$, $J_3:=\{2i+1,\ldots
,2i+j\}$,
and $J_4:=\{2i+j+1,\ldots,2i+2j\}$.
Let $\Pi_{ij}$ be the set of all partitions $\sigma$ of $I_{ij}$ such that
$|J\cap J'|\le1$ for all $J\in\pi$ and all $J'\in\sigma$.
By $\widetilde{\Pi}_{ij}$ we denote the set of all partitions $\sigma
\in
\Pi_{ij}$ such that:
\begin{longlist}[(iii)]
\item[(i)] $\{1,2i+1\},\{i+1,2i+j+1\}\in\sigma$ or $\{
1,i+1,2i+1,2i+j+1\}\in\sigma$;
\item[(ii)] each block of $\sigma$ has at least two elements;
\item[(iii)] for every partition of $\{1,2,3,4\}$ in two disjoint
nonempty sets $M_1,M_2$
there are $u\in M_1$, $v\in M_2$ such that $J_u$ and $J_v$ are both
intersected by
one block of~$\sigma$.
\end{longlist}
Let $\widetilde{\Pi}^{(1)}_{ij}$ (resp., $\widetilde{\Pi}^{(2)}_{ij}$)
be the set of all partitions $\sigma\in\widetilde{\Pi}_{ij}$ such that
$\{1,2i+1\},\{i+1,2i+j+1\}\in\sigma$ (resp. $\{1,i+1,2i+1,2i+j+1\}
\in
\sigma$). In the terminology of diagram formulae as it is used in
\cite{PeccatiTaqqu2010},
Chapter~4, condition (iii) means that $\pi$ and
$\sigma$ generate a ``connected diagram.''

Now we are able to state the main result of this section.

\begin{theorem}\label{thm:abstract}
Assume that $F^{(k)}\in L^2_\eta$ and
%
\begin{equation}
\label{eqn:assumptionintegrability} \int\bigl|\bigl(f_i^{(k)}\otimes
f_i^{(k)}\otimes f_j^{(l)}\otimes
f_j^{(l)}\bigr)_\sigma\bigr| \,d\lambda^{|\sigma|} <
\infty
\end{equation}
for all $\sigma\in\Pi_{ij}$, $i,j\in\N$, and $k,l\in\{1,\ldots,
m\}$.
Further, assume
that there are $a>0$ and $b\geq1$ such that
%
\begin{equation}
\label{eqn:assumptionbounds} \int\bigl|\bigl(f_i^{(k)}\otimes
f_i^{(k)}\otimes f_j^{(l)}\otimes
f_j^{(l)}\bigr)_\sigma\bigr| \,d\lambda^{|\sigma|}
\leq\frac{a b^{i+j}}{(i!)^2(j!)^2}
\end{equation}
for all $\sigma\in\widetilde{\Pi}_{ij}$, $i,j\in\N$, and $k,l \in
\{
1,\ldots, m\}$.
Let $F:=(F^{(1)},\ldots,F^{(m)})$ and let $N$ be a centered Gaussian random
vector with a given positive semidefinite covariance matrix
$(\sigma_{k,l})_{k,l=1,\ldots,m}$. Then
\begin{eqnarray*}
\label{eqn:CLTmulti} \mathbf{d}_3(F-\E F,N) & \le&\frac{m}{2} \sum
_{k,l=1}^m\bigl |\sigma _{k,l}-\C
\bigl(F^{(k)},F^{(l)}\bigr)\bigr|
\\
&&{} + \Biggl(\frac{m}{2}+\frac{m}{4}\sum
_{n=1}^m \sqrt{\V F^{(n)}}
\Biggr)2^{{13}/{2}} m^2 \sum_{i=1}^\infty
i^{17/2}\frac{b^i}{\lfloor i/14\rfloor!} \sqrt {a}.
\nonumber
\end{eqnarray*}
\end{theorem}

In the univariate case, we have the following result for the
Wasserstein distance.

\begin{corol}\label{corol:univariat}
Let $F\in L^2_\eta$ be such that $\V F>0$ and the assumptions~\eqref
{eqn:assumptionintegrability} and \eqref{eqn:assumptionbounds} are
satisfied and let
$N$ be a standard Gaussian random variable. Then
\[
\mathbf{d}_W \biggl(\frac{F-\E F}{\sqrt{\V F}},N \biggr)
\le2^{{15}/{2}}\sum_{i=1}^\infty
i^{17/2}\frac{b^i}{\lfloor
i/14\rfloor!} \frac{\sqrt{a}}{\V F}.
\]
\end{corol}

We prepare the proof of Theorem~\ref{thm:abstract} by two lemmas and a
proposition.

\begin{lemma}\label{lem:Rij} Let $i,j\in\N$, $f\in L_s^2(\lambda^i)$,
$g\in L^2_s(\lambda^j)$,
and assume that
\[
\int\bigl|(f\otimes f\otimes g\otimes g)_\sigma\bigr| \,d\lambda^{|\sigma|} <
\infty, \qquad\sigma\in\Pi_{ij}.
\]
Then
%
\begin{eqnarray}
\label{eqn:Rij1}&& \V \biggl(\int I_{i-1}\bigl(f(z,\cdot)\bigr)
I_{j-1}\bigl(g(z,\cdot)\bigr) \lambda (dz) \biggr)
\nonumber
\\[-8pt]
\\[-8pt]
\nonumber
&&\qquad = \sum
_{\sigma\in\widetilde{\Pi}^{(1)}_{ij}}\int(f \otimes f \otimes g \otimes g)_{\sigma} \,d
\lambda^{|\sigma|},
\\
\label{eqn:Rij2}
&&\E\int I_{i-1}\bigl(f(z,\cdot)\bigr)^2
I_{j-1}\bigl(g(z,\cdot)\bigr)^2 \lambda(dz)
\nonumber
\\[-8pt]
\\[-8pt]
\nonumber
&&\qquad = \sum
_{\sigma\in\widetilde{\Pi}^{(2)}_{ij}}\int(f \otimes f \otimes g \otimes
g)_{\sigma} \,d\lambda^{|\sigma|}.
\end{eqnarray}
\end{lemma}

\begin{pf}
Combining the formulas
\begin{eqnarray*}
&& \E \biggl(\int I_{i-1}\bigl(f(z,\cdot)\bigr) I_{j-1}
\bigl(g(z,\cdot)\bigr) \lambda (dz) \biggr)^2
\\
&&\qquad =\iint\E I_{i-1}\bigl(f(y,\cdot)\bigr) I_{i-1}\bigl(f(z,
\cdot)\bigr) I_{j-1}\bigl(g(y,\cdot)\bigr) I_{j-1}\bigl(g(z,
\cdot)\bigr) \lambda(dy) \lambda(dz)
\end{eqnarray*}
and
%
\begin{equation}\quad
\label{eqn:orthogonal} \E\int I_{i-1}\bigl(f(z,\cdot)\bigr) I_{j-1}
\bigl(g(z,\cdot)\bigr) \lambda(dz) = \cases{ (i-1)! \langle f,g
\rangle_i, &\quad $i=j$,\vspace*{2pt}
\cr
0, & \quad $i\neq j$,}
\end{equation}
with Theorem~3.1 in \cite{LPST}
(see also \cite{PeccatiTaqqu2010}, Corollary~7.2 and \cite{Surgailis1984},
Proposition~3.1) proves
the first equation. The second identity is a consequence of
\[
\E\int I_{i-1}\bigl(f(z,\cdot)\bigr)^2 I_{j-1}
\bigl(g(z,\cdot)\bigr)^2 \lambda(dz) =\int\E I_{i-1}
\bigl(f(z,\cdot)\bigr)^2 I_{j-1}\bigl(g(z,\cdot)
\bigr)^2 \lambda(dz)
\]
and, again, Theorem~3.1 in \cite{LPST}.
\end{pf}

\begin{proposition}\label{prop:abstractbound}
Let $F^{(1)},\ldots,F^{(m)} \in L^2_\eta$ be such that \eqref
{eqn:assumptionintegrability}
holds. Let $F:=(F^{(1)},\ldots,F^{(m)})$ and let $N$ be a centered
Gaussian random
vector with a given positive semidefinite covariance
matrix $(\sigma_{k,l})_{k,l=1,\ldots,m}$. Then
%
\begin{eqnarray}\label{eqn:boundd3}
&&\mathbf{d}_3(F-\E F, N)\nonumber\\
&&\qquad\leq\frac{m}{2} \sum
_{k,l=1}^m \bigl|\sigma _{k,l}-\C
\bigl(F^{(k)},F^{(l)}\bigr)\bigr|
\nonumber
\\[-8pt]
\\[-8pt]
\nonumber
&&\qquad\quad{} + \Biggl(\frac{m}{2}+\frac{m}{4}\sum
_{n=1}^m\sqrt{\V F^{(n)}} \Biggr)\\
&&\qquad\qquad{}\times\sum
_{k,l=1}^m \sum
_{i,j=1}^\infty ij \sqrt{\sum
_{\sigma\in\widetilde{\Pi}_{i,j}} \int\bigl|\bigl(f_i^{(k)}\otimes
f_i^{(k)} \otimes f_j^{(l)}\otimes
f_j^{(l)}\bigr)_\sigma\bigr| \,d\lambda^{|\sigma|}}
\nonumber
.
\end{eqnarray}
\end{proposition}

\begin{pf}
To avoid convergence issues, we start by proving \eqref{eqn:boundd3}
for $F_s:=(F_s^{(1)},\ldots,F_s^{(m)})$ with the truncated Poisson
functionals $F_s^{(l)}:=\BE F^{(l)}+ \sum_{n=1}^s I_n(f_n^{(l)})$,
$l\in
\{1,\ldots,m\}$, for a fixed $s\in\N$. By construction, we have
$F_s^{(1)},\ldots,F_s^{(m)}\in\operatorname{dom} D$. From \cite{PeccatiZheng2010},
Theorem~4.2, it is known that
%
\begin{eqnarray}
\label{eqn:boundPeccatiZheng} \mathbf{d}_3(F_s-\E F_s,N)
&\le& \frac{m}{2} \sqrt{\sum_{k,l=1}^m
\E \biggl(\sigma_{k,l}- \int D_zF_s^{(k)}
\bigl(-D_zL^{-1}F_s^{(l)}\bigr)
\lambda(dz) \biggr)^2}
\nonumber
\\[-8pt]
\\[-8pt]
\nonumber
& &{}+\frac{1}{4}\int \E \Biggl(\sum_{k=1}^m
\bigl|D_zF_s^{(k)}\bigr| \Biggr)^2 \sum
_{l=1}^m\bigl|D_zL^{-1}F_s^{(l)}\bigr|
\lambda(dz).
\end{eqnarray}
We bound the two summands on the above right-hand side separately.
For the first one, we have
\begin{eqnarray*}
&& \sqrt{\sum_{k,l=1}^m \E
\biggl(\sigma_{k,l}- \int D_zF_s^{(k)}
\bigl(-D_zL^{-1}F_s^{(l)}\bigr)
\lambda(dz) \biggr)^2}
\\
& &\qquad\leq\sum_{k,l=1}^m \biggl(\E
\biggl(\sigma_{k,l}-\C \bigl(F_s^{(k)},F_s^{(l)}
\bigr)+\C\bigl(F_s^{(k)},F_s^{(l)}
\bigr)\\
&&\hspace*{106pt}{}- \int D_zF_s^{(k)}\bigl(-D_zL^{-1}F_s^{(l)}
\bigr) \lambda(dz) \biggr)^2\biggr)^{1/2}
\\
&&\qquad \leq\sum_{k,l=1}^m \biggl(\bigl|
\sigma_{k,l}-\C \bigl(F_s^{(k)},F_s^{(l)}
\bigr)\bigr|\\
&&\hspace*{60pt}{}+\sqrt {\E \biggl(\C\bigl(F_s^{(k)},F_s^{(l)}
\bigr)- \int D_zF_s^{(k)}\bigl(-D_zL^{-1}F_s^{(l)}
\bigr) \lambda(dz) \biggr)^2} \biggr).
\end{eqnarray*}
Put $g_n^{(l)}(z):=I_{n-1}(f_n^{(l)}(z,\cdot))$. From \eqref{Dx},
\eqref
{DxL}, the covariance version of \eqref{Fock} [see \eqref{2.8}] and
\eqref{eqn:orthogonal},
we obtain that
\begin{eqnarray*}
a_s^{k,l} &:=&\E \biggl(\int D_zF_s^{(k)}
\bigl(-D_zL^{-1}F_s^{(l)}\bigr)
\lambda (dz)-\C\bigl(F_s^{(k)},F_s^{(l)}
\bigr) \biggr)^2
\\
& =& \E \Biggl(\int\sum_{i=1}^s i
g_i^{(k)}(z) \sum_{j=1}^s
g_j^{(l)}(z) \lambda(dz)-\sum_{n=1}^s
n! \bigl\langle f_n^{(k)},f_n^{(l)}
\bigr\rangle_n \Biggr)^2
\\
& =& \V \Biggl(\sum_{i,j=1}^s i\int
g_i^{(k)}(z) g_j^{(l)}(z) \lambda(dz)
\Biggr).
\end{eqnarray*}
Note that the right-hand side is well defined since \eqref
{eqn:assumptionintegrability} and Lemma~\ref{lem:Rij} ensure that each
of the summands is square integrable.

Since $\sqrt{\V(Y_1+Y_2)} \leq\sqrt{\V Y_1}+\sqrt{\V Y_2}$ for random
variables $Y_1,Y_2$, we obtain
\begin{eqnarray*}
\sqrt{a_s^{k,l}} & \le&\sum
_{i,j=1}^s i \sqrt{\V \biggl(\int
g_i^{(k)}(z) g_j^{(l)}(z) \lambda(dz)
\biggr)}
\\
& \le&\sum_{i,j=1}^s i \sqrt{\sum
_{\sigma\in\widetilde{\Pi}_{ij}} \int \bigl| \bigl(f_i^{(k)}
\otimes f_i^{(k)}\otimes f_j^{(l)}
\otimes f_j^{(l)}\bigr)_\sigma \bigr| \,d
\lambda^{|\sigma|}},
\end{eqnarray*}
where we have applied \eqref{eqn:Rij1} in Lemma~\ref{lem:Rij} to get
the final inequality.

By Jensen's inequality and the definitions of the Malliavin operators,
we obtain
for the second summand in \eqref{eqn:boundPeccatiZheng} that
\begin{eqnarray*}
&& \int\E \Biggl(\sum_{k=1}^m
\bigl|D_zF_s^{(k)}\bigr| \Biggr)^2 \sum
_{l=1}^m\bigl|D_zL^{-1}F_s^{(l)}\bigr|
\lambda(dz)
\\
&&\qquad \le m \sum_{k,l=1}^m \int\E
\bigl(D_zF_s^{(k)}\bigr)^2\bigl|D_zL^{-1}F_s^{(l)}\bigr|
\lambda(dz)
\\
&&\qquad \le m \sum_{k,l=1}^m \sum
_{i,j=1}^s i j \int\E\bigl|g^{(k)}_i(z)\bigr|
\bigl|g^{(k)}_j(z)\bigr| \bigl|D_zL^{-1}F_s^{(l)}\bigr|
\lambda (dz)
\\
&&\qquad \leq m \sum_{k,l=1}^m \sum
_{i,j=1}^s i j \sqrt{\int\E
g_i^{(k)}(z)^2 g_j^{(k)}(z)^2
\lambda(dz)} \sqrt{\int\E\bigl(D_zL^{-1}F_s^{(l)}
\bigr)^2 \lambda(dz)}.
\end{eqnarray*}
Combining \eqref{DxL} and \eqref{eqn:orthogonal} with \eqref{Fock},
we get
\[
\int\E\bigl(D_zL^{-1}F_s^{(l)}
\bigr)^2 \lambda(dz) = \sum_{n=1}^s
(n-1)! \bigl\| f_n^{(l)}\bigr\|_n^2 \leq\V
F_s^{(l)}.
\]
Now \eqref{eqn:Rij2} in Lemma~\ref{lem:Rij} completes the proof of
\eqref{eqn:boundd3} for $F_s$. By the triangle inequality for the
$\mathbf{d}_3$-distance and \cite{LPST}, Lemma~5.5, we have that
\begin{eqnarray*}
\mathbf{d}_3(F-\BE F, N) & \leq&\mathbf{d}_3(F-\BE F,
F_s-\BE F_s) +\mathbf{d}_3(F_s-
\BE F_s, N)
\\
& \leq& m\sqrt{\BE\|F-\BE F\|^2+\BE\|F_s-\BE
F_s\|^2} \sqrt{\BE\| F-F_s\|^2}\\
&&{}+\mathbf{d}_3(F_s-\BE F_s, N),
\end{eqnarray*}
where $\|\cdot\|$ stands for the Euclidean norm in $\R^m$. Since
$F^{(l)}_s\to F^{(l)}$ in $L^2_\eta$ as $s\to\infty$, the first summand
vanishes as $s\to\infty$. Applying \eqref{eqn:boundd3} to the second
summand and letting $s\to\infty$ completes the proof.
\end{pf}

\begin{lemma}\label{lem:Rij22} For any integers $i,j\ge1$,
\[
|\widetilde{\Pi}_{i,j}| \le \frac{(i!)^2(j!)^2\max\{i+1,j+1\}^{11}}{\lceil{\max\{i,j\}/
7}\rceil!}.
\]
\end{lemma}

\begin{pf}
For a fixed partition $\sigma\in\widetilde{\Pi}_{ij}$, let $k_{uv}$
with $u,v\in\{1,2,3,4\}$ and $u<v$ be the number of blocks $A\in
\sigma$
such that $|A\cap J_u|=|A\cap J_v|=1$ and $A\cap(J_u \cup J_v)=A$. We
define $k_{uvw}$ for $u,v,w\in\{1,2,3,4\}$ with $u<v<w$ and $k_{1234}$
in the same way. For a possible combination of fixed numbers
$k_{12},\ldots,k_{1234}$ the number of partitions $\sigma\in
\widetilde
{\Pi}_{ij}$
having this form is less than
\[
\frac
{(i!)^2(j!)^2}{k_{12}!k_{13}!k_{14}!k_{23}!k_{24}!k_{34}!k_{123}!k_{124}!k_{134}!k_{234}!k_{1234}!} \le\frac{(i!)^2(j!)^2}{\lceil{\max\{i,j\}/ 7}\rceil!}.
\]
To get this inequality, we have used the fact that
\[
k_{12}+k_{13}+k_{14}+k_{123}+k_{124}+k_{134}+k_{1234}=i,
\]
whence one of the factors in the denominator is at least $\lceil
i/7\rceil$. For a similar reason, there must be a factor in the
denominator that is at least $\lceil j/7\rceil$.

Moreover, there are less than $\max\{i+1,j+1\}^{11}$ possible choices for
$k_{12},\ldots, k_{1234}$, which completes the proof.
\end{pf}

Note that we have not used the first and the third condition of the
definition of $\widetilde{\Pi}_{ij}$ in the proof of Lemma~\ref
{lem:Rij22}, whence the inequality even holds for a larger class of
partitions. Now we are prepared for the proofs of Theorem~\ref
{thm:abstract} and Corollary~\ref{corol:univariat}.

\begin{pf*}{Proof of Theorem~\ref{thm:abstract}}
We aim at applying Proposition~\ref{prop:abstractbound}. Combining
Lemma~\ref{lem:Rij22}
and assumption \eqref{eqn:assumptionbounds}, we get
\begin{eqnarray*}
&&\sum_{i,j=1}^\infty ij \sqrt{\sum
_{\sigma\in\widetilde{\Pi}_{ij}} \int \bigl| \bigl(f_i^{(k)}
\otimes f_i^{(k)}\otimes f_j^{(l)}
\otimes f_j^{(l)}\bigr)_\sigma\bigr | \,d
\lambda^{|\sigma|}} \\
&&\qquad\le\sum_{i,j=1}^\infty
i j \sqrt{\frac{\max
\{
i+1,j+1\}^{11}
b^{i+j}a}{\lceil{\max\{i,j\}/ 7}\rceil!}}.
\end{eqnarray*}
A straightforward computation and the inequality $\sqrt{m!}\geq
\lfloor
m/2\rfloor!$
for $m\in\N$ show that
\begin{eqnarray*}
\sum_{i,j=1}^\infty i j \sqrt{
\frac{\max\{i+1,j+1\}^{11}
b^{i+j}}{\lceil{\max\{i,j\}/7}\rceil!}} &\le&2^{{13}/{2}} \sum_{1\leq j \leq i}
i^2 \sqrt{\frac{\max\{i,j\}^{11} b^{i+j}}{\lceil{\max\{i,j\}/7}\rceil!}} \\
&\le&2^{{13}/{2}} \sum
_{i=1}^\infty i^{17/2} \frac{b^i}{\lfloor{i/14}\rfloor!},
\end{eqnarray*}
where the right-hand side converges. Thus, Theorem~\ref{thm:abstract}
is a consequence of Proposition~\ref{prop:abstractbound}.
\end{pf*}

\begin{pf*}{Proof of Corollary~\ref{corol:univariat}}
We define the truncated Poisson functional $F_s:=\BE F+\sum_{n=1}^s
I_n(f_n)$ for $s\in\N$. By the triangle inequality for the Wasserstein
distance and combining the definition of the Wasserstein distance with
the Cauchy--Schwarz inequality, we obtain that
\begin{eqnarray*}
\mathbf{d}_W \biggl(\frac{F-\BE F}{\sqrt{\V F}},N \biggr) & \leq&{\mathbf
d_W} \biggl(\frac{F-\BE F}{\sqrt{\V F}}, \frac{F_s-\BE F_s}{\sqrt{\V F}} \biggr) +
\mathbf{d}_W \biggl(\frac{F_s-\BE F_s}{\sqrt{\V F}},N \biggr)
\\
& \leq&\frac{\sqrt{\BE(F-F_s)^2}}{\sqrt{\V F}} + \mathbf{d}_W \biggl(
\frac
{F_s-\BE F_s}{\sqrt{\V F}},N \biggr).
\end{eqnarray*}
Here, the first summand vanishes as $s\to\infty$ since $F_s\to F$ in
$L^2_\eta$ as $s\to\infty$. For the second term, we know from \cite{Peccatietal2010},
Theorem~3.1, that
\begin{eqnarray*}
&&\mathbf{d}_W \biggl(\frac{F_s-\E F_s}{\sqrt{\V F}},N \biggr) \\
&&\qquad\le
\frac{\V F -\V F_s}{\V F}+\frac{1}{\V F} \sqrt{\E \biggl(\V F_s-\int
D_zF_s\bigl(-D_zL^{-1}F_s
\bigr) \lambda(dz) \biggr)^2}
\\
&&\qquad\quad{} +\frac{1}{(\V F)^{3/2}}\int\E(D_z F_s)^2
\bigl|D_zL^{-1}F_s\bigr| \lambda(dz).
\end{eqnarray*}
Now we can use the same arguments as in
the proofs of Proposition~\ref{prop:abstractbound} and Theorem~\ref
{thm:abstract}.
\end{pf*}

\section{Central limit theorems for geometric functionals}\label{secmclt}

In the following, we use the general normal approximation results of
the previous section to
derive central limit theorems for geometric functionals of the Boolean
model \eqref{Bool}.
We establish central limit theorems under the minimal moment assumption
\eqref{vi2}, but we need a stronger moment assumption in order to
derive rates of convergence. For the Berry--Esseen bounds, we assume
that the typical grain $Z_0$ of the Boolean model satisfies the moment
assumption
%
\begin{equation}
\label{eqn:moment3plus} \BE V_i(Z_0)^{3+\varepsilon}<\infty,\qquad  i
\in\{0,\ldots,d\},
\end{equation}
for a fixed $\varepsilon\in(0,1]$. This allows us to state central
limit theorems with rates of convergence depending on $\varepsilon$.

\begin{theorem}\label{thm:multiGeneral}
Let $\psi_1,\ldots,\psi_m$ be geometric functionals on $\cR^d$ and let
$\Psi:=(\psi_1, \ldots, \psi_m)$.
Assume that \eqref{vi2} is satisfied and let $N$ be an $m$-dimensional centered
Gaussian random vector with covariance matrix $(\sigma_{\psi_k,\psi
_l})_{k,l=1,\ldots,m}$
given by \eqref{eqn:limitCovariance}. Then
\[
\frac{1}{\sqrt{V_d(W)}}\bigl(\Psi(Z\cap W)-\E\Psi(Z\cap W)\bigr)
 \stackrel {d} {
\longrightarrow} N \qquad\mbox{as } r(W)\to\infty.
\]
If \eqref{eqn:moment3plus} is satisfied, there is a constant $c_{\psi
_1,\ldots,\psi_m}$
depending on $\psi_1,\ldots,\psi_m,\Lambda$, and $\varepsilon$
such that
%
\begin{equation}
\label{No1} \mathbf{d}_3 \biggl(\frac{1}{\sqrt{V_d(W)}}\bigl(\Psi(Z\cap
W)-\E\Psi (Z\cap W)\bigr),N \biggr) \le\frac{c_{\psi_1,\ldots,\psi_m}}{r(W)^{\min\{\varepsilon d/2,1\}}}
\end{equation}
for $W\in\cK^d$ with $r(W)\geq1$.
\end{theorem}

\begin{remark}\label{remboundsI}
We will see in the proof of Theorem~\ref{thm:multiGeneral} that the
translation invariance of $\psi_1,\ldots,\psi_m$ is only used to ensure
the existence of an asymptotic covariance matrix. Hence, such a
multivariate central limit theorem still holds for functionals $\psi
_1,\ldots,\psi_m$ which are not translation invariant if we can
establish the existence of an asymptotic covariance matrix. In this
case, the rate of convergence depends on the rate of convergence for
the covariances.
\end{remark}

In the univariate case, we can rescale by the square root of the
variance, whence the existence of the asymptotic variance is not
necessary. Thus, translation invariance of the functional is not
required. We only need to assume that the variance does not degenerate
as $r(W)\to\infty$, which, for instance, holds under the conditions of
Section~\ref{secpos}.

\begin{theorem}\label{thm:univariat}
Assume that \eqref{vi2} is satisfied and let $\psi$ be an additive,
locally bounded and measurable functional on
$\cR^d$ with constants $r_0\geq1$ and $\sigma_0>0$ such that
%
\begin{equation}
\label{eqn:assumptionVariance} \frac{\V\psi(Z\cap W)}{V_d(W)} \geq\sigma_0
\end{equation}
for $W\in\cK^d$ with $r(W)\geq r_0$. Denote by $N$ a standard Gaussian
random variable. Then
\[
\frac{\psi(Z\cap W)-\E\psi(Z\cap W)}{\sqrt{\V\psi(Z\cap W)}}
 \stackrel {d} {\longrightarrow} N \qquad\mbox{as } r(W)\to\infty.
\]
If \eqref{eqn:moment3plus} is satisfied, there is a constant $c_\psi$
depending on $\psi,\Lambda,\sigma_0,r_0$, and $\varepsilon$
such that
%
\begin{equation}
\label{No2} \mathbf{d}_W \biggl(\frac{\psi(Z\cap W)-\E\psi(Z\cap W)}{\sqrt{\V
\psi(Z\cap
W)}},N \biggr)
\le\frac{c_{\psi}}{V_d(W)^{\min\{\varepsilon/2,1/2\}}}
\end{equation}
for $W\in\cK^d$ with $r(W)\ge r_0$.
\end{theorem}

\begin{remark}
Together with the well-known fact that the Kolmogorov distance to a
standard Gaussian random variable is always bounded by the square
root of the Wasserstein distance to a standard Gaussian random
variable (see Proposition~1.2 in \cite{Ross2011}, e.g.), it
follows from \eqref{No2} that
\[
\sup_{x\in\R} \biggl| \P \biggl(\frac{\psi(Z\cap W)-\E\psi(Z\cap
W)}{\sqrt{\V\psi(Z\cap W)}}\leq x \biggr)-\P(N
\leq x) \biggr| \le \frac{\sqrt{c_{\psi}}}{V_d(W)^{\min\{\varepsilon/4,1/4\}}}
\]
for $W\in\cK^d$ with $r(W)\ge r_0$. However, this approach leads to a
weaker rate of convergence than for the Wasserstein distance, which
might be suboptimal since for many central limit theorems one has the
same rate for both distances.
\end{remark}

\begin{remark}\label{remboundsII}
By replacing in \eqref{No2} the volume of $W$ by the volume of its inball,
we obtain a rate of order $r(W)^{-\min\{\varepsilon d/2, d/2\}}$.
Comparing \eqref{No1} and \eqref{No2}, we see
that for $\varepsilon=1$ and $d\geq3$ the rate of convergence in the
multivariate case is weaker
than in the univariate case. This is caused by the slow rate of convergence
in Theorem~\ref{thm:variance} since we need to bound
\[
\sum_{k,l=1}^m \biggl\llvert
\sigma_{\psi_k,\psi_l} -\frac{\C(\psi_k(Z\cap W),\psi_l(Z\cap W))}{V_d(W)}\biggr\rrvert
\]
in order to apply Theorem~\ref{thm:abstract}. In the univariate
analogue, which is
Corollary~\ref{corol:univariat}, we normalize with the exact variance and
do not have such a term. If we replace the Gaussian random vector $N$
by a centered
Gaussian random vector $N(W)$, having the covariance matrix of
$V_d(W)^{-1/2}\Psi(Z\cap W)$,
the sum above vanishes and we obtain
\[
\mathbf{d}_3 \biggl(\frac{1}{\sqrt{V_d(W)}}\bigl(\Psi(Z\cap W) - \E\Psi (Z
\cap W)\bigr),N(W) \biggr)
\le\frac{c_{\psi_1,\ldots,\psi_m}}{V_d(W)^{\min\{\varepsilon
/2,1/2\}}},
\]
which is the same rate as in the univariate case.

For $k,l\in\{1,\ldots,m\}$, we obtain by choosing $g(x)=x_k x_l/2$ as a
test function in the definition of the $\mathbf{d}_3$-distance that
\begin{eqnarray*}
&&\mathbf{d}_3 \biggl(\frac{1}{\sqrt{V_d(W)}}\bigl(\Psi(Z\cap W) - \E\Psi (Z
\cap W)\bigr),N \biggr) \\
&&\qquad\geq\frac{1}{2} \biggl\llvert \sigma_{\psi_k,\psi_l}-
\frac
{\C(\psi
_k(Z\cap W),\psi_l(Z\cap W))}{V_d(W)}\biggr\rrvert .
\end{eqnarray*}
Hence, Proposition~\ref{prop:optimalrate} shows that the rate in
\eqref
{No1} is optimal for $\varepsilon=1$ and $d\geq2$.
\end{remark}

We organize the proofs of Theorems \ref{thm:multiGeneral} and
\ref{thm:univariat} such that we
first impose the moment assumption \eqref{eqn:moment3plus} and
establish \eqref{No1} and \eqref{No2}.
In a second step, we prove that convergence in distribution is still
obtained (without convergence rates)
under the weaker moment assumption \eqref{vi2}.

\begin{pf*}{Proof of \eqref{No1} in Theorem~\ref{thm:multiGeneral} under
assumption \eqref{eqn:moment3plus}}
From now on, we write
\[
f_i^{(k)}(K_1,\ldots,K_i):=
\frac{(-1)^i}{i!}\psi^*_k(K_1\cap\cdots \cap
K_i\cap W)
\]
for $K_1,\ldots,K_i\in\cK^d$, $1\leq k \leq m$, and $i\geq1$.
It is a direct consequence of \eqref{Tn} that $f_i^{(k)}$
is the $i$th kernel of the Wiener--It\^o chaos expansion of
the Poisson functional $\psi_k(Z\cap W)$.

The integrability condition \eqref{eqn:assumptionintegrability} is
satisfied since the kernels are bounded by~\eqref{est3a} for every
$W\in
\cK^d$ and the
measure of the grains hitting $W$ is also finite.

In the sequel, we check assumption \eqref{eqn:assumptionbounds} for the
cases $\sigma\in\widetilde{\Pi}^{(1)}_{ij}$ and $\sigma\in
\widetilde{\Pi
}^{(2)}_{ij}$ separately.
We start with the first case. Let $k,l\in\{1,\ldots,m\}$ and $\sigma
\in
\widetilde{\Pi}^{(1)}_{ij}$.
From~\eqref{est3a} in Lemma~\ref{lem:boundintrinsic}, it follows that
\begin{eqnarray*}
&& \int\bigl|\bigl(f_i^{(k)}\otimes f_i^{(k)}
\otimes f_j^{(l)} \otimes f_j^{(l)}
\bigr)_\sigma\bigr| \,d\Lambda^{|\sigma|}
\\
& &\qquad\le\frac{(\beta(\psi_k) \beta(\psi_l))^2}{(i!)^2(j!)^2} \int \prod_{p=1}^4
\sum_{r=0}^d V_r \biggl(
\bigcap_{n\in N_p(\sigma)} K_n\cap W \biggr)
\Lambda^{|\sigma|} \bigl(d(K_1,\ldots,K_{|\sigma|})\bigr)
\end{eqnarray*}
with nonempty sets $N_p(\sigma)\subset\{1,\ldots,|\sigma|\}$,
$p=1,\ldots,4$,
depending on $\sigma$. Every $n \in\{1,\ldots,|\sigma|\}$ is
contained in
at least two of these sets. By removing the index $n$ from the sets until
it occurs only in one set, we increase the integral and can use Lemma~\ref{lem:kinematic} to integrate over $K_n$. Due to the special
structure of $\sigma\in\widetilde{\Pi}^{(1)}_{ij}$, we obtain by
iterating this step and using the abbreviation
\[
h_{W}(A)=\sum_{r=0}^d
V_r(A\cap W)
\]
that
\begin{eqnarray*}
&& \int\bigl|\bigl(f_i^{(k)}\otimes f_i^{(k)}
\otimes f_j^{(l)} \otimes f_j^{(l)}
\bigr)_\sigma\bigr| \,d\Lambda^{|\sigma|}
\\
&&\qquad \leq\frac{(\beta(\psi_k) \beta(\psi_l))^2}{(i!)^2(j!)^2}
\alpha ^{|\sigma|-3}\\
&&\hspace*{31pt}{}\times \int h_W(K_1)
h_W(K_1\cap K_2)h_W(K_2
\cap K_3) h_W(K_3) \Lambda ^3
\bigl(d(K_1,K_2,K_3)\bigr)
\\
&&\qquad = \frac{(\beta(\psi_k) \beta(\psi_l))^2}{(i!)^2(j!)^2} \alpha ^{|\sigma|-3} \int \biggl( \int
h_W(K_1) h_W(K_1\cap
K_2) \Lambda(dK_1) \biggr)^2
\Lambda(dK_2).
\end{eqnarray*}
For a fixed $K_2\in\cK^d$, Lemma~\ref{lem:translativeIntegral} implies
the second inequality in
\begin{eqnarray*}
&&\int h_W(K_1) h_W(K_1\cap
K_2) \Lambda(dK_1) \\
&&\qquad \leq\gamma\E \Biggl[ \sum
_{r=0}^d V_r(Z_0) \int\sum
_{s=0}^d V_s
\bigl((Z_0+x)\cap K_2\cap W\bigr) \,dx \Biggr]
\\
&&\qquad \leq(d+1)\gamma\beta_1 \E \Biggl[ \Biggl(\sum
_{r=0}^d V_r(Z_0)
\Biggr)^2 \Biggr] \sum_{s=0}^d
V_s(K_2\cap W).
\end{eqnarray*}
Putting $c_7:=(d+1)\gamma\beta_1 \E [ (\sum_{r=0}^d
V_r(Z_0)
)^2 ]$
and applying Lemma~\ref{lem:translativeIntegral} again, we get
\begin{eqnarray*}
&&\int \biggl( \int h_W(K_1)  h_W(K_1
\cap K_2)\Lambda(dK_1) \biggr)^2
\Lambda(dK_2)
\\
& &\qquad\leq c_7^2 \int \Biggl(\sum
_{r=0}^d V_r(K_2\cap W)
\Biggr)^2 \Lambda (dK_2)
\\
&&\qquad \leq\gamma c_7^2 \E \Biggl[ \sum
_{r=0}^d V_r(Z_0) \int\sum
_{s=0}^d V_s
\bigl((Z_0+x)\cap W\bigr) \,dx \Biggr]
\\
& &\qquad\leq(d+1)\gamma\beta_1 c_7^2 \E \Biggl[
\Biggl( \sum_{r=0}^d V_r(Z_0)
\Biggr)^2 \Biggr] \sum_{s=0}^d
V_s(W) = c_8 \sum_{r=0}^d
V_r(W)
\end{eqnarray*}
with $c_8:=c_7^3$. Finally, since $|\sigma|\le i+j$ for
$\sigma\in\widetilde{\Pi}_{ij}$ we have
%
\begin{eqnarray}
\label{eqn:condition1} \int\bigl|\bigl(f_i^{(k)}\otimes
f_i^{(k)} \otimes f_j^{(l)} \otimes
f_j^{(l)}\bigr)_\sigma\bigr| \,d\Lambda^{|\sigma|}
&\leq&\frac{c_8 (\beta(\psi_k) \beta(\psi_l))^2
}{(i!)^2(j!)^2} \alpha^{|\sigma|-3} \sum_{r=0}^d
V_r(W)
\nonumber
\\[-8pt]
\\[-8pt]
\nonumber
&\leq&\frac{a_1
b_1^{i+j}}{(i!)^2(j!)^2}
\end{eqnarray}
with $a_1:=\max_{1\le k,l\leq m} \alpha^{-3}c_8 (\beta(\psi_k)
\beta(\psi_l))^2
\sum_{r=0}^d V_r(W)$ and $b_1:=\max\{\alpha,1\}$.
It follows from Lemma~\ref{lem:BoundInradius} that there is a
constant $c_9$ depending on $\psi_1,\ldots,\psi_m$ and $\Lambda$
such that
%
\begin{equation}
\label{eqn:bounda1} \frac{a_1}{V_d(W)^2} \leq\frac{c_9}{V_d(W)}
\end{equation}
for $W\in\cK^d$ with $r(W)\geq1$.

For $\sigma\in\widetilde{\Pi}_{i,j}^{(2)}$, we obtain from \eqref
{est3a} in Lemmas \ref{lem:boundintrinsic}
and \ref{lem:kinematic} as above that
\begin{eqnarray*}
&&\int\bigl|\bigl(f_i^{(k)}\otimes f_i^{(k)}
\otimes f_j^{(l)}\otimes f_j^{(l)}
\bigr)_\sigma\bigr| \,d\Lambda^{|\sigma|} \\
&&\qquad\le\frac{(\beta(\psi_k) \beta(\psi_l))^2}{(i!)^2(j!)^2} \alpha
^{|\sigma|-1} \int \Biggl(\sum_{r=0}^d
V_r(K_1\cap W) \Biggr)^4
\Lambda(dK_1).
\end{eqnarray*}
A further application of Lemma~\ref{lem:translativeIntegral} yields the
second inequality in
\begin{eqnarray*}
&& \int \Biggl(\sum_{r=0}^d
V_r(K_1\cap W) \Biggr)^4
\Lambda(dK_1)
\\
&&\qquad \leq\gamma\BE \Biggl[ \Biggl(\sum_{r=0}^d
\min\bigl\{V_r(Z_0),V_r(W)\bigr\}
\Biggr)^3 \int\sum_{s=0}^d
V_s\bigl((Z_0+x)\cap W\bigr) \,dx \Biggr]
\\
&&\qquad \leq(d+1)\gamma\beta_1 \BE \Biggl[ \Biggl(\sum
_{r=0}^d \min\bigl\{ V_r(Z_0),V_r(W)
\bigr\} \Biggr)^3 \sum_{s=0}^d
V_s(Z_0) \Biggr] \sum_{u=0}^d
V_u(W).
\end{eqnarray*}
Consequently, we have
%
\begin{equation}
\label{eqn:condition2} \int\bigl|\bigl(f_i^{(k)}\otimes
f_i^{(k)}\otimes f_j^{(l)}\otimes
f_j^{(l)}\bigr)_\sigma\bigr| \,d\Lambda^{|\sigma|}
\leq\frac{a_2 b_2^{i+j}}{(i!)^2(j!)^2}
\end{equation}
with
\begin{eqnarray*}
a_2&=&(d+1)\gamma\beta_1\max_{1\le k,l\leq m}
\frac{(\beta(\psi_k)
\beta
(\psi_l))^2}{\alpha} \\
&&{}\times\BE \Biggl[ \Biggl(\sum_{r=0}^d
\min\bigl\{ V_r(Z_0),V_r(W)\bigr\}
\Biggr)^3 \sum_{s=0}^d
V_s(Z_0) \Biggr] \sum_{u=0}^d
V_u(W)
\end{eqnarray*}
and $b_2=\max\{\alpha,1\}$. Since
\begin{eqnarray*}
&& \frac{1}{V_d(W)^2}\BE \Biggl[ \Biggl(\sum_{r=0}^d
\min\bigl\{ V_r(Z_0),V_r(W)\bigr\}
\Biggr)^3 \sum_{s=0}^d
V_s(Z_0) \Biggr] \sum_{u=0}^d
V_u(W)
\\
&&\qquad \leq\BE \Biggl[ \Biggl(\sum_{r=0}^d
V_r(Z_0) \Biggr)^3 \sum
_{s=0}^d \frac
{V_s(Z_0)^{\varepsilon} V_s(W)^{1-\varepsilon}\}}{V_d(W)} \Biggr] \sum
_{u=0}^d \frac{V_u(W)}{V_d(W)},
\end{eqnarray*}
Lemma~\ref{lem:BoundInradius} and the moment assumption \eqref
{eqn:moment3plus} imply that there is
a constant $c_{10}$ depending on $\psi_1,\ldots,\psi_m$, $\Lambda$, and
$\varepsilon$ such that, for $W\in\cK^d$ with $r(W)\geq1$,
%
\begin{equation}
\label{eqn:bounda2} \frac{a_2}{V_d(W)^2} \leq\frac{c_{10}}{V_d(W)^\varepsilon}.
\end{equation}

If we rescale the Poisson functionals by $V_d(W)^{-{1}/{2}}$,
\eqref
{eqn:condition1} and \eqref{eqn:condition2} imply that~\eqref
{eqn:assumptionbounds} holds with $a=\max\{a_1,a_2\}V_d(W)^{-2}$ and
$b=\max\{b_1,b_2\}$. By \eqref{eqn:bounda1} and~\eqref{eqn:bounda2},
$a$ is of the order $V_d(W)^{-\min\{1,\varepsilon\}}$. Now \eqref{No1}
is a consequence of Theorems~\ref{thm:variance} and \ref{thm:abstract}.
\end{pf*}

\begin{pf*}{Proof of \eqref{No2} in Theorem~\ref{thm:univariat} under
assumption \eqref{eqn:moment3plus}}
By the same arguments as in the previous proof and analogous choices
for $a_1,a_2$, the conditions of Corollary~\ref{corol:univariat} are
satisfied with $a=\max\{a_1,a_2\}$. It follows from assumption~\eqref
{eqn:assumptionVariance} that
\[
\frac{a}{(\V\psi(Z\cap W))^2}=\frac{\max\{a_1,a_2\}}{(\V\psi
(Z\cap
W))^2} \leq\frac{1}{\sigma_0^2}
\frac{\max\{a_1,a_2\}}{V_d(W)^2}
\]
for $W\in\cK^d$ with $r(W)\geq r_0$. Combining this with \eqref
{eqn:bounda1} and \eqref{eqn:bounda2} completes the proof.
\end{pf*}

For $W\in\cK^d$ with $r(W)>0$, we define the set
\[
M_W= \bigl\{K\in\cK^d\dvtx V_j(K)\leq
\sqrt{V_d(W)}, j\in\{0,\ldots ,d\} \bigr\}.
\]
The restriction of $\eta$ to $M_W$ is a stationary Poisson process,
which generates the stationary Boolean model
\[
Z_W=\bigcup_{K\in\eta\cap M_W}K.
\]
The idea of the proofs of Theorems \ref{thm:multiGeneral} and
\ref{thm:univariat}
under the weaker moment assumption \eqref{vi2} is to approximate the Boolean
model $Z$ by the Boolean model $Z_W$.
A~similar truncation has been used in \cite{HeiSp13}
to prove the central limit theorem for the volume of
a more general Boolean model based on a Poisson process
of cylinders.

The restriction of $\eta$ to $M_W$ has the intensity $\gamma
_W:=\gamma
\P(Z_0\in M_W)$ [note that $\gamma_W>0$ for $r(W)$ sufficiently large]
and its typical grain $Z_{0,W}$ has the distribution $\P(Z_{0,W}\in
\cdot):=\P(Z_0\in\cdot\cap M_W)/\P(Z_0\in M_W)$. For the Boolean model
$Z_W$, obviously all previous results hold if we replace $Z_0$ and
$\gamma$ by $Z_{0,W}$ and $\gamma_W$. But Lemmas \ref
{lem:boundintrinsic} and \ref{lem:kinematic} as well as the upper
bounds in the proof of Theorem~\ref{thm:multiGeneral} under the
stronger assumption \eqref{eqn:moment3plus} remain true for $Z_W$ if we
take the same constants as for $Z$, which we do in the sequel. The
reason for this is that the constants do only depend on the product of
intensity and grain distribution and that the intrinsic volumes are monotone.

The Boolean models $Z_W$ and $Z$ satisfy the following relation.

\begin{lemma}\label{lem:approximation}
Let $\psi$ be an additive, locally bounded and measurable functional on
$\cR^d$ and assume that \eqref{vi2} is satisfied. Then
\[
\lim_{r(W)\to\infty}\frac{\BE ( (\psi(Z\cap W)-\BE\psi
(Z\cap
W) ) -  (\psi(Z_W\cap W)-\BE\psi(Z_W\cap W) )  )^2}{V_d(W)}=0
\]
and
\begin{eqnarray*}
&&\limsup_{r(W)\to\infty}\frac{\BE (\psi(Z\cap W)-\BE\psi
(Z\cap
W) )^2 +\BE (\psi(Z_W\cap W)-\BE\psi(Z_W\cap W)
)^2}{V_d(W)}\\
&&\qquad<\infty.
\end{eqnarray*}
\end{lemma}

\begin{pf}
Define, for $K_1,\ldots,K_n\in\cK^d$,
\begin{eqnarray*}
&&g_{n,W}(K_1,\ldots,K_n)\\
&&\qquad : =
\frac{(-1)^n}{n! \sqrt{V_d(W)}} \bigl(\BE \psi(Z\cap K_1\cap\cdots\cap
K_n\cap W)-\psi(K_1\cap\cdots\cap K_n\cap W)
\bigr).
\end{eqnarray*}
Further, we define
\begin{eqnarray*}
&&h_{n,W}(K_1,\ldots,K_n) \\
&&\qquad:= \frac{(-1)^n}{n!\sqrt{V_d(W)}}
\bigl(\BE \psi (Z_W\cap K_1\cap\cdots\cap
K_n\cap W)- \psi(K_1\cap\cdots\cap K_n\cap W)
\bigr)
\end{eqnarray*}
for $K_1,\ldots,K_n\in M_W$ and $h_{n,W}(K_1,\ldots,K_n):=0$ if there
is a $j\in\{1,\ldots,n\}$ with $K_j\notin M_W$.

In view of Lemma~\ref{lem:kernels} and the Fock space representation
\eqref{2.8}, the assertions of this lemma are equivalent to
%
\begin{eqnarray}
\label{eqn:VarianceDifference} \lim_{r(W)\to\infty} \sum_{n=1}^{\infty}
n! \|g_{n,W}-h_{n,W}\| _n^2 &= &0\quad
\mbox{and}
\nonumber
\\[-8pt]
\\[-8pt]
\nonumber
 \limsup_{r(W)\to\infty} \sum_{n=1}^{\infty}
n! \bigl( \| g_{n,W}\|_n^2+\|h_{n,W}
\|_n^2 \bigr) &<&\infty,
\end{eqnarray}
which we shall prove in the following. For $n\in\N$, we have
\begin{eqnarray*}
&&\|g_{n,W}-h_{n,W}\|_n^2 \\
&&\qquad = \gamma
\iiint\bigl(g_{n,W}(K_1+x,K_2,\ldots
,K_n)\\
&&\hspace*{66pt}{}-h_{n,W}(K_1+x,K_2,
\ldots,K_n)\bigr)^2
 dx \Lambda^{n-1}\bigl(d(K_2,\ldots,K_n)\bigr)
\Q(dK_1).
\end{eqnarray*}
Our aim is to apply the dominated convergence theorem to the outer
integral. For any $K_1\in\cK^d_o$, it follows from Lemmas \ref
{lem:boundintrinsic}, \ref{lem:translativeIntegral} and
\ref
{lem:kinematic} similarly as in \eqref{eq:majorant} that
\begin{eqnarray*}
&& \iint\bigl(g_{n,W}(K_1+x,K_2,
\ldots,K_n)\\
&&\quad{}-h_{n,W}(K_1+x,K_2,
\ldots,K_n)\bigr)^2 \,dx \Lambda^{n-1}
\bigl(d(K_2,\ldots,K_n)\bigr)
\\
&&\qquad \leq2 \iint \bigl( g_{n,W}(K_1+x,K_2,\ldots
,K_n)^2+h_{n,W}(K_1+x,K_2,
\ldots,K_n)^2 \bigr) \,dx\\
&&\hspace*{55pt}{}\times \Lambda ^{n-1}
\bigl(d(K_2,\ldots,K_n)\bigr)
\\
&&\qquad \leq2 (d+1) \beta_1 \beta(\psi)^2 \Biggl(\sum
_{i=0}^d V_i(K_1)
\Biggr)^2 \sum_{i=0}^d
\frac{V_i(W)}{V_d(W)}\frac{\alpha^{n-1}}{(n!)^2}.
\end{eqnarray*}
The right-hand side of the previous inequality is uniformly bounded for
$r(W)\geq1$ because of Lemma~\ref{lem:BoundInradius}. Moreover, the
sum over $n$ is integrable with respect to $K_1$ due to \eqref{vi2}.
Thus, limit and summation in the first sum in \eqref
{eqn:VarianceDifference} can be interchanged. By the same arguments,
the second inequality above yields the second formula in \eqref
{eqn:VarianceDifference}.

Next, we show that, for any $K_1\in\cK^d_o$,
%
\begin{eqnarray}\qquad
\label{eqn:limitdifference} &&\lim_{r(W)\to\infty}
\iint\bigl(g_{n,W}(K_1+x,K_2,
\ldots ,K_n)-h_{n,W}(K_1+x,K_2,
\ldots,K_n)\bigr)^2\,
dx
\nonumber
\\[-8pt]
\\[-8pt]
\nonumber
&&\hspace*{54pt}{}\times \Lambda^{n-1}\bigl(d(K_2,\ldots,K_n)\bigr)
=0.
\end{eqnarray}

For $K_1,\ldots,K_n \in M_W$, we have
\begin{eqnarray*}
&& g_{n,W}(K_1,\ldots,K_n)-h_{n,W}(K_1,
\ldots,K_n)
\\
&&\qquad = \frac{1}{n!}\frac{(-1)^n}{\sqrt{V_d(W)}} \BE \bigl[\psi(Z\cap K_1
\cap \cdots\cap K_n\cap W) \\
&&\hspace*{75pt}\qquad{}- \psi(Z_{W}\cap
K_1\cap\cdots\cap K_n\cap W) \bigr].
\end{eqnarray*}
Let us denote by $Z_1,\ldots,Z_{N_{K_1\cap W}}$ the grains of $\eta$
that intersect $K_1\cap W$ and are not in $M_W$. Then $N_{K_1\cap W}$
follows a Poisson distribution with mean
$
\Lambda(\{K\notin M_W\dvtx K\cap K_1\cap W\neq\varnothing\})$.
Since $Z\cap K_1\cap W= (Z_W\cup Z_1 \cup\cdots\cup Z_{N_{K_1\cap
W}})\cap K_1 \cap W$, it follows from the inclusion--exclusion formula that
\begin{eqnarray*}
&& \bigl|\psi(Z\cap K_1\cap\cdots\cap K_n\cap W) -
\psi(Z_W\cap K_1\cap \cdots \cap K_n\cap W)\bigr|
\\
&&\qquad \leq\sum_{\varnothing\neq J\subset\{1,\ldots,N_{K_1\cap W}\}} \biggl|\psi \biggl(Z_W\cap
\bigcap_{j\in J}Z_j \cap K_1
\cap\cdots\cap K_n\cap W \biggr) \biggr|
\\
&&\quad\qquad{} +\sum_{\varnothing\neq J\subset\{1,\ldots,N_{K_1\cap W}\}}\biggl |\psi \biggl(\bigcap
_{j\in J}Z_j \cap K_1\cap\cdots\cap
K_n\cap W \biggr) \biggr|.
\end{eqnarray*}
Recall the definitions of the constants $c_1,c_2$ and $c_4$ from
Section~\ref{sec:3}. Denoting by $\BP_W$ the distribution of the
restriction of $\eta$ to $M_W$, we obtain by \eqref{2.1a} and the
monotonicity of the intrinsic volumes that
\begin{eqnarray*}
&& \int \biggl|\psi \biggl(Z(\mu)\cap\bigcap_{j\in J}Z_j
\cap K_1\cap \cdots \cap K_n\cap W \biggr) \biggr|
\BP_W(d\mu)
\\
&&\qquad \leq c_1 M(\psi) \sum_{i=0}^d
V_i \biggl(\bigcap_{j\in J}Z_j
\cap K_1\cap\cdots\cap K_n\cap W \biggr) \\
&&\qquad\leq
c_1 M(\psi) \sum_{i=0}^d
V_i(K_1\cap\cdots\cap K_n\cap W).
\end{eqnarray*}
Applying \eqref{est0a} and the monotonicity of the intrinsic volumes yields
\begin{eqnarray*}
&&\biggl|\psi \biggl(\bigcap_{j\in J}Z_j \cap
K_1\cap\cdots\cap K_n\cap W \biggr)\biggr | \\[-2pt]
&&\qquad \leq
c_2c_4 M(\psi) \sum_{i=0}^d
V_i \biggl(\bigcap_{j\in
J}Z_j
\cap K_1\cap\cdots\cap K_n\cap W \biggr)
\\[-2pt]
&&\qquad \leq c_2c_4 M(\psi) \sum
_{i=0}^d V_i(K_1\cap\cdots
\cap K_n\cap W).
\end{eqnarray*}
Since the restrictions of $\eta$ to $M_W$ and to its complement are
stochastically independent, altogether we have that, for $K_1,\ldots
,K_n\in M_W$,
\begin{eqnarray*}
&& \bigl|g_{n,W}(K_1,\ldots,K_n)-
h_{n,W}(K_1,\ldots,K_n)\bigr|
\\[-2pt]
&&\qquad \leq\frac{(c_1+c_2c_4)M(\psi)}{n! \sqrt{V_d(W)}} \BE \bigl[2^{N_{K_1\cap W}}-1 \bigr] \sum
_{i=0}^d V_i(K_1\cap\cdots
\cap K_n\cap W)
\\[-2pt]
& &\qquad\leq\frac{\hat\beta(\psi)}{n!\sqrt{V_d(W)}} \bigl(\exp \bigl(p_W(K_1)
\bigr)-1 \bigr) \sum_{i=0}^d
V_i(K_1\cap\cdots\cap K_n\cap W)
\end{eqnarray*}
with $p_W(K_1)=\Lambda(\{K\notin M_W\dvtx K\cap K_1\neq\varnothing\}$ and $
\hat\beta(\psi)= (c_1+c_2c_4) M(\psi)$.

If there is a $j\in\{1,\ldots,n\}$ such that $K_j\notin M_W$, we have
$g_{n,W}-h_{n,W}=g_{n,W}$, and it follows from Lemma~\ref
{lem:boundintrinsic} that
\[
\bigl|g_{n,W}(K_1,\ldots,K_n)\bigr| \leq
\frac{\beta(\psi)}{n! \sqrt{V_d(W)}} \sum_{k=0}^d
V_k( K_1\cap\cdots\cap K_n\cap W).
\]

For a fixed $K_1\in\cK^d_o$ and $r(W)$ sufficiently large such that
$K_1\in M_W$, we have that
\begin{eqnarray*}
& &\iint\bigl(g_{n,W}(K_1+x,K_2,
\ldots,K_n)-h_{n,W}(K_1+x,K_2,
\ldots,K_n)\bigr)^2 \,dx \\[-2pt]
&&\quad{}\times\Lambda^{n-1}
\bigl(d(K_2,\ldots,K_n)\bigr)
\\[-2pt]
&&\qquad \leq\iint \Biggl(\frac{\hat\beta(\psi)^2}{(n!)^2V_d(W)} \Biggl( \sum
_{k=0}^d V_k\bigl((K_1+x)
\cap K_2 \cap\cdots\cap K_n\cap W\bigr)
\Biggr)^2 \\[-2pt]
&&\hspace*{48pt}{}\times\bigl(\exp\bigl(p_{W}(K_1)\bigr)-1
\bigr)^2
\\[-2pt]
&&\hspace*{15pt}\qquad\quad{} + \frac{\beta(\psi)^2}{(n!)^2V_d(W)} \Biggl( \sum_{k=0}^d
V_k\bigl((K_1+x)\cap K_2 \cap\cdots\cap
K_n\cap W\bigr) \Biggr)^2 \\[-2pt]
&&\hspace*{224pt}{}\times\sum
_{i=2}^n \I\{ K_i\notin
M_{W}\} \Biggr) \,dx \\[-2pt]
&&\hspace*{49pt}{}\times\Lambda^{n-1}\bigl(d(K_2,\ldots,K_n)\bigr)
\\[-2pt]
& &\qquad\leq\frac{\hat\beta(\psi)^2 (d+1)\beta_1 \alpha^{n-1}}{(n!)^2} \frac
{1}{V_d(W)} \\[-2pt]
&&\hspace*{32pt}{}\times\sum_{i=0}^d
V_i(W) \Biggl(\sum_{r=0}^d
V_r(K_1) \Biggr)^2 \bigl(\exp
\bigl(p_W(K_1)\bigr)-1 \bigr)^2
\\[-2pt]
&&\qquad\quad{}+ \frac{\beta(\psi)^2 (d+1)\beta_1 \alpha^{n-2}}{(n!)^2} \frac
{1}{V_d(W)}\\[-2pt]
&&\quad\hspace*{32pt}{}\times \sum_{i=0}^d
V_i(W) \Biggl(\sum_{r=0}^d
V_r(K_1) \Biggr)^2 (n-1) p_W(K_1),
\end{eqnarray*}
where we have used Lemmas \ref{lem:translativeIntegral} and \ref
{lem:kinematic}.

Then Lemma~\ref{lem:BoundInradius} and $p_W(K_1)\to0$ as $r(W)\to
\infty
$ show that the right-hand side vanishes for $r(W)\to\infty$. This
proves \eqref{eqn:limitdifference} so that the dominated convergence
theorem yields the first formula in \eqref{eqn:VarianceDifference},
which completes the proof.
\end{pf}

\begin{pf*}{Proof of Theorem~\ref{thm:multiGeneral} under assumption
\eqref{vi2}}
The triangle inequality for the $\mathbf{d}_3$-distance yields
%
\begin{eqnarray}
\label{eqn:triangled3} &&\mathbf{d}_3 \biggl(\frac{1}{\sqrt{V_d(W)}}\bigl(\Psi(Z
\cap W)-\E\Psi (Z\cap W)\bigr),N \biggr)\nonumber
\\
&&\qquad \leq\mathbf{d}_3 \biggl(\frac{1}{\sqrt{V_d(W)}}\bigl(\Psi(Z\cap W)-\E
\Psi (Z\cap W)\bigr),
\nonumber
\\[-8pt]
\\[-8pt]
\nonumber
&&\hspace*{52pt}{}\frac{1}{\sqrt{V_d(W)}}\bigl(\Psi(Z_W\cap W)-\E
\Psi(Z_W\cap W)\bigr) \biggr)
\nonumber
\\
&&\qquad\quad{} +\mathbf{d}_3 \biggl(\frac{1}{\sqrt{V_d(W)}}\bigl(\Psi(Z_W
\cap W)-\E \Psi (Z_W\cap W)\bigr),N \biggr).
\nonumber
\end{eqnarray}
In the sequel, we show that both terms on the right-hand side of \eqref
{eqn:triangled3} vanish as $r(W)\to\infty$. By \cite{LPST}, Lemma~5.5,
the first expression is bounded by
\begin{eqnarray*}
&& m\bigl(\BE\bigl\|\Psi(Z\cap W)-\E\Psi(Z\cap W)\bigr\|^2/V_d(W)\\
&&\hspace*{12pt}{}+
\bigl\|\Psi (Z_W\cap W)-\E\Psi(Z_W\cap W)\bigr\|^2
/V_d(W)\bigr)^{1/2}
\\
&&\qquad{} \times\bigl(\BE\bigl\|\Psi(Z\cap W)-\BE\Psi(Z\cap W)\\
&&\hspace*{36pt}{}-
\Psi(Z_W\cap W)+
\E \Psi(Z_W\cap W)\bigr\|^2/V_d(W)\bigr)^{1/2},
\end{eqnarray*}
where $\|\cdot\|$ stands for the Euclidean norm in $\R^m$. Since, by
Lemma~\ref{lem:approximation}, the first factor is bounded and the
second factor vanishes as $r(W)\to\infty$, the first expression on the
right-hand side of \eqref{eqn:triangled3} vanishes as $r(W)\to\infty$.

By applying Theorem~\ref{thm:abstract} to the vector $\Psi(Z_W\cap W)$
of Poisson functionals depending on the restriction of $\eta$ to $M_W$,
we shall prove that
%
\begin{equation}
\label{eqn:d3ZW} \lim_{r(W)\to\infty} \mathbf{d}_3 \biggl(
\frac{1}{\sqrt
{V_d(W)}}\bigl(\Psi (Z_W\cap W)-\BE\Psi(Z_W
\cap W)\bigr),N \biggr)=0.
\end{equation}
Theorem~\ref{thm:abstract} yields this without a rate of convergence if
%
\begin{equation}
\label{eqn:covarianceZW} \lim_{r(W)\to\infty} \frac{\C(\psi_k(Z_W\cap W),\psi_l(Z_W\cap
W))}{V_d(W)}=
\sigma_{\psi_k,\psi_l}
\end{equation}
for $k,l\in\{1,\ldots,m\}$ and if \eqref{eqn:assumptionbounds} holds
with a fixed $b\geq1$ and $a\geq0$ depending on $W$ such that $a$
tends to zero as $r(W)\to\infty$.

Condition \eqref{eqn:covarianceZW} is satisfied because of
Lemma~\ref{lem:approximation} and Theorem~\ref{thm:variance}.
Inequalities \eqref{eqn:condition1} and \eqref{eqn:condition2} also
hold for the Boolean model $Z_W$ with the same $a_1,b_1,b_2$ as in the
proof of \eqref{No1} under assumption \eqref{eqn:moment3plus} and
\begin{eqnarray*}
a_2&=&c_{11}\max_{1\le k,l\leq m}
\frac{(\beta(\psi_k) \beta(\psi
_l))^2\gamma_W}{\alpha}\\
&&{}\times \BE \Biggl[ \Biggl(\sum_{r=0}^d
\min\bigl\{ V_r(Z_{0,W}),V_r(W)\bigr\}
\Biggr)^3 \sum_{s=0}^d
V_s(Z_{0,W}) \Biggr]\sum_{u=0}^d
V_u(W)
\end{eqnarray*}
with $c_{11}:=(d+1)\beta_1$. This is the case since the derivations of
\eqref{eqn:condition1} and \eqref{eqn:condition2} require only finite
second moments and we can use the constants related to $Z$ as discussed
before Lemma~\ref{lem:approximation}.
Consequently, \eqref{eqn:assumptionbounds} is satisfied with $a=\max\{
a_1,a_2\}/V_d(W)^2$ and $b=\max\{b_1,b_2\}$. Since \eqref{eqn:bounda1}
only requires that the second moments, which are contained in $c_8$,
are finite, we obtain that $a_1/V_d(W)^2$ tends to zero as $r(W)\to
\infty$. On the other hand, $\lim_{r(W)\to\infty}a_2/V_d(W)^2=0$ is
equivalent to
%
\begin{equation}\qquad
\label{eqn:limita2} \lim_{r(W)\to\infty}\gamma_W\BE \Biggl[
\frac{1}{V_d(W)} \Biggl(\sum_{r=0}^d
\min\bigl\{V_r(Z_{0,W}),V_r(W)\bigr\}
\Biggr)^3 \sum_{s=0}^d
V_s(Z_{0,W}) \Biggr]=0.\hspace*{-30pt}
\end{equation}
The expression in the limit can be rewritten as
\[
\gamma\int\frac{1}{V_d(W)} \Biggl(\sum_{r=0}^d
\min\bigl\{ V_r(K),V_r(W)\bigr\} \Biggr)^3
\sum_{s=0}^d V_s(K) \I\{K\in
M_W\} \BQ(dK).
\]
For $K\in\cK^d_o\cap M_W$, we have $V_r(K)\leq\sqrt{V_d(W)}$ for
$r\in\{0,\ldots,d\}$ and, therefore,
\begin{eqnarray*}
&&\frac{1}{V_d(W)} \Biggl(\sum_{r=0}^d
\min\bigl\{V_r(K),V_r(W)\bigr\} \Biggr)^3 \sum
_{s=0}^d V_s(K) \I\{K\in
M_W\} \\
&&\qquad\leq (d+1)^2 \Biggl(\sum
_{r=0}^d V_r(K) \Biggr)^2,
\end{eqnarray*}
which is independent of $W$ and integrable with respect to $\BQ$. For
any fixed $K\in\cK^d_o$ the left-hand side vanishes as $r(W)\to
\infty$
so that the dominated convergence theorem implies \eqref{eqn:limita2},
and hence $a$ tends to zero as $r(W)\to\infty$. Finally, Theorem~\ref
{thm:abstract} yields \eqref{eqn:d3ZW}, which completes the proof of
Theorem~\ref{thm:multiGeneral}.
\end{pf*}

\begin{remark}
As discussed in Remark~\ref{remboundsII}, \eqref{No1} still holds if we
replace the centered Gaussian random vector $N$ with the asymptotic
covariance matrix by a centered Gaussian random vector $N(W)$ with the
exact covariance matrix. This can be done even if the functionals are
not translation invariant since in this case we do not need Theorem~\ref
{thm:variance}. The second part of the proof of Theorem~\ref
{thm:multiGeneral} still holds because \eqref{eqn:covarianceZW} is not
required in this situation. This means that under condition~\eqref{vi2}
for additive, locally bounded and measurable functionals $\psi
_1,\ldots
,\psi_m$,
\[
\lim_{r(W)\to\infty} \mathbf{d}_3 \biggl(\frac{1}{\sqrt{V_d(W)}}
\bigl(\Psi (Z\cap W)-\BE\Psi(Z\cap W)\bigr), N(W) \biggr)=0.
\]
\end{remark}

\begin{pf*}{Proof of Theorem~\ref{thm:univariat} under assumption
\eqref{vi2}}
For $m=1$ and a centered Gaussian random variable $N(W)$ with variance
$\V\psi(Z\cap W)/V_d(W)$, the previous remark implies that
%
\begin{equation}
\label{eqn:limitd3univariate} \lim_{r(W)\to\infty} \mathbf{d}_3 \biggl(
\frac{\psi(Z\cap W)-\BE
\psi(Z\cap
W)}{\sqrt{V_d(W)}},N(W) \biggr)=0.
\end{equation}
It follows from the definition of the $\mathbf{d}_3$-distance that for
random vectors $Y_1,Y_2$ and any $c>0$,
\[
\mathbf{d}_3(cY_1,cY_2) \leq\max\{1,c
\}^3 \mathbf{d}_3(Y_1,Y_2).
\]
With $c_W:=\sqrt{V_d(W)}/\sqrt{\V\psi(Z\cap W)}$ and a standard
Gaussian random variable~$N$, this yields
\begin{eqnarray*}
&&\mathbf{d}_3 \biggl(\frac{\psi(Z\cap W)-\BE\psi(Z\cap W)}{\sqrt{\V
\psi
(Z\cap W)}},N \biggr) \\
&&\qquad\leq\max
\{1,c_W\}^3 \mathbf{d}_3 \biggl(
\frac
{\psi
(Z\cap W)-\BE\psi(Z\cap W)}{\sqrt{V_d(W)}},N(W) \biggr).
\end{eqnarray*}
Since $c_W$ is bounded by assumption \eqref{eqn:assumptionVariance},
\eqref{eqn:limitd3univariate} completes the proof.
\end{pf*}

\begin{remark}
In Theorem~\ref{thm:univariat}, it is possible to weaken the assumption
that the Poisson process is stationary. In the proof, we only need to
find\vadjust{\goodbreak} upper bounds for the kernels and some integrals. This is, for
instance, still possible if the intensity measure is of the form
\[
\Lambda(\cdot)=\iint\I\{K+x\in\cdot\} f(x) \,dx \BQ(dK)
\]
with a nonnegative bounded function $f\dvtx \R^d\to\R$. Now we always get
upper bounds if we replace this intensity measure by the measure in
\eqref{Lambda} with $\gamma=\sup_{x\in\R^d}|f(x)|<\infty$.

For the multivariate central limit, this argument does not work in
general since its proof makes use of Theorem~\ref{thm:variance}, which
depends on the translation invariance of the intensity measure. But if
one can prove by other methods the existence of an asymptotic
covariance matrix, it is still possible to weaken the stationarity
assumption as described above.
\end{remark}


\printaddresses
\end{document}